\documentclass[12pt,oneside,english]{amsart}
\usepackage[T1]{fontenc}
\usepackage[latin9]{inputenc}
\usepackage{amssymb}
\makeatletter
\numberwithin{equation}{section}
\numberwithin{figure}{section}
\theoremstyle{plain}
\newtheorem{thm}{\protect\theoremname}[section]
  \theoremstyle{definition}
  \newtheorem{defn}[thm]{\protect\definitionname}
  \theoremstyle{remark}
  \newtheorem{rem}[thm]{\protect\remarkname}
  \theoremstyle{plain}
  \newtheorem{prop}[thm]{\protect\propositionname}
  \theoremstyle{plain}
  \newtheorem{lem}[thm]{\protect\lemmaname}
  \theoremstyle{plain}
  \newtheorem{cor}[thm]{\protect\corollaryname}

\usepackage{amsfonts}
\usepackage{amscd}

\newtheorem{theorem}{Theorem}[section]

\newtheorem{examples}[theorem]{Examples}

\newcommand{\be}{\begin{equation}}
\newcommand{\ee}{\end{equation}}

\usepackage{babel}
\providecommand{\definitionname}{Definition}
  \providecommand{\lemmaname}{Lemma}
  \providecommand{\remarkname}{Remark}
\providecommand{\theoremname}{Theorem}

\usepackage{babel}
\providecommand{\definitionname}{Definition}
  
  \providecommand{\lemmaname}{Lemma}
  \providecommand{\propositionname}{Proposition}
  \providecommand{\remarkname}{Remark}
\providecommand{\theoremname}{Theorem}

\makeatother

\usepackage{babel}
  \providecommand{\corollaryname}{Corollary}
  \providecommand{\definitionname}{Definition}
  \providecommand{\lemmaname}{Lemma}
  \providecommand{\propositionname}{Proposition}
  \providecommand{\remarkname}{Remark}
\providecommand{\theoremname}{Theorem}

\begin{document}

\title{Matricial Function Theory and Weighted Shifts}

\author{Paul S. Muhly}

\address{Department of Mathematics \\
University of Iowa\\
 Iowa City, IA 52242}

\email{e-mail: paul-muhly@uiowa.edu }

\author{Baruch Solel}

\address{Department of Mathematics\\
 Technion\\
32000 Haifa, Israel}

\email{e-mail: mabaruch@techunix.technion.ac.il}


\thanks{The research of both authors was supported in part by a US-Israel
Binational Science Foundation grant.}

\begin{abstract}
Let $\mathcal{T}_{+}(E)$ be the tensor algebra of a $W^{*}$-correspond\-ence
$E$ over a $W^{*}$-algebra $M$. In earlier work, we showed that the 
completely contractive representations of $\mathcal{T}_{+}(E)$, whose restrictions to $M$ are normal, are parametrized
by certain discs or balls $\overline{D(E,\sigma)}$ indexed by  the
normal $*$-repres\-ent\-ations $\sigma$ of $M$. Each disc 
has analytic structure, and each element $F\in \mathcal{T}_{+}(E) $ gives
rise to an operator-valued function $\widehat{F}_{\sigma}$
on $\overline{D(E,\sigma)}$ that is continuous and analytic on the interior.  In this paper, we explore the effect of adding operator-valued weights to the theory.  While the statements of many of the results in the weighted theory are anticipated by those in the unweighted setting, substantially different proofs are required.  Interesting new connections with the theory of completely positive are developed.  Our perspective has been inspired by work of Vladimir M\"{u}ller \cite{Muller1988} in which he studied operators that can be modeled by parts of  weighted shifts.   Our results may be interpreted as providing a description of operator algebras that can be modeled by weighted tensor algebras.  Our results also extend work of Gelu Popescu \cite{Popescu2010}, who investigated similar questions. 
\end{abstract}



\maketitle

\section{Introduction\label{sec:Introduction}}

This paper is a sequel to \cite{Muhly2013}, where we pursue a function
theory that derives from thinking about elements of tensor algebras
and Hardy algebras, based on $C^{*}$- and $W^{*}$- correspondences,
as functions on the spaces of representations of the algebras. It
turns out that the spaces of representations have complex manifold
structures and the functions that arise are genuinely holomorphic.
The spaces and functions also have a matricial structure that is very
much like that first studied by Joseph Taylor in \cite{Tay72c}, and
the functions we studied are very similar to the functions he connected
to elements in free algebras. Taylor's work has been the focus of
a lot of interest in recent years and an account of a large portion
of the current state-of-the-art has been written by Dmitry Kaliuzhnyi-Verbovetskyi
and Victor Vinnikov \cite{K-VV2014}. We refer to their work for many
references and background. 

The current paper arose from our efforts to probe broader contexts
for the ideas that we employed in \cite{Muhly2013}. We have been
especially inspired by the work of Vladimir M\"{u}ller \cite{Muller1988}
and Gelu Popescu \cite{Popescu2010}. To understand our motivation,
recall that in his seminal paper \cite{Rota1960} Rota essentially
proved that the study of contraction operators on Hilbert space is
coextensive with the study of the parts of unilateral shifts (of all
multiplicities). If one interprets contraction operators as the representations
of the disk algebra, on the one hand, and if one interprets unilateral
shifts as defining induced representations on certain Fock spaces,
on the other - as we did in \cite{Muhly1999} - then one has a perfect
analogy between Rota's work and ours. M\"{u}ller asked the question:
``Which operators can be modeled by parts of \emph{weighted shifts}?''
He did not give a complete answer. Indeed, his analysis leads naturally
to the view that no simple definitive answer can be given. However,
he was able to formulate some very general conditions on a weight
sequence under which one can characterize those operators which are
unitarily equivalent to a part of the unilateral shift with those
weights. (See in particular \cite[Theorem 2.2]{Muller1988}.) In \cite{Popescu2010},
Popescu began with the conclusion of M\"{u}ller's Theorem 2.2 and
expanded it to the setting of noncommutative tuples of operators.
On the basis of this expansion, he formulated a very broad generalization
of the noncommutative function theory he has been developing over
the years. 

Our point of departure is to take a fresh look at M\"{u}ller's Theorem
2.2, and from it formulate a notion of a weighted tensor algebra whose
representation theory is parametrized by a ``disc'' that generalizes
the kind of domains that Poposecu studied. In fact, even when specialized
to Popescu's setting in \cite{Popescu2010}, our discs are more general
than his. Many of our arguments follow his in broad outline, but numerous
new approaches and details are necessitated by the increase in generality
we consider.

\section{Preliminaries\label{sec:Preliminaries}}

For the purposes we have in mind, it seems best to formulate our results
in the context of $W^{*}$-algebras and $W^{*}$-correspondences.
However, many of our results are seen easily to have reformulations
that are valid in the context of unital $C^{*}$-algebras and $C^{*}$-correspondences.
We will call attention to relevant modifications for the $C^{*}$
setting when it seems advisable to do so. It should be emphasized,
too, that at the current level of development of the theory, working
in the setting with \emph{finite dimensional} $W^{*}$-algebras and
correspondences leads to interesting new results. In this case, of
course, there is no distinction between $C^{*}$-algebras and $W^{*}$-algebras.
However, what happens with the infinite dimensional constructs we
build from them does depend on whether one takes norm limits or (ultra)
weak limits. 

We shall follow the theory developed initially in \cite{Muhly2004a},
but due to advances during the last 10 years, or so, it has become
necessary to adopt somewhat different notation. Throughout this paper,
unless specified otherwise, $M$ will denote a $W^{*}$-algebra. That
is, $M$ will denote a $C^{*}$-algebra that is also a dual space.
We want to emphasize a representation-free approach, and so we make
no \emph{a priori} assumptions about $M$ being represented on some
particular Hilbert space. The weak-$*$ topology on $M$ will be referred
to as the ultraweak topology on $M$. Homomorphisms between $W^{*}$-algebras
will always be assumed to be continuous with respect to their ultraweak
topologies. For emphasis, we shall call such homomorphisms ``normal''
or ``$W^{*}$-homomorphisms''. We will also assume that our homomorphisms
are unital unless explicit assertions to the contrary are made. 

Likewise, $E$ will denote a $W^{*}$-correspondence over $M$. This
means, first of all, that $E$ is a self-dual, right, Hilbert $C^{*}$-module
over $M$ in the sense of Paschke \cite{Paschke1973}. (We shall follow
\cite{MT2005} for details about $W^{*}$-correspondences.) By \cite[Proposition 3.3.3]{MT2005},
$E$ is a dual space and we shall call the weak-$*$ topology on $E$
the ultraweak topology, as well. The set of all endomorphisms of $E$
will be denoted by $\mathcal{L}(E)$. By \cite[Corollary 3.3.2 and Proposition 3.3.4 ]{MT2005},
$\mathcal{L}(E)$ is naturally a $W^{*}$-algebra. To make $E$ a
$W^{*}$-correspondence we must assume that $E$ is a left module
over $M$ and that the left action is given by a normal $*$-homomorphism,
usually denoted $\varphi$, of $M$ into $\mathcal{L}(E)$. 

To avoid digressions which contribute little to the main points we
want to make, we shall assume that the left action of $M$ on $E$
is faithful and that the right action is \emph{full} in the sense
that the ultraweakly closed linear span of the inner products $\langle\xi,\eta\rangle$,
$\xi,\eta\in E$, is all of $M$. These two assumptions guarantee
that all the tensor powers $E^{\otimes n}$, $n\geq0$ (which are
always assumed to be balanced over $M$) are nonzero. In fact, each
$E^{\otimes n}$ is full, and the left action is faithful. We shall
denote the left action of $M$ on $E^{\otimes n}$ by $\varphi_{n}$.
Recall that it is defined on elementary tensors by the formula $\varphi_{n}(a)(\xi_{1}\otimes\xi_{2}\otimes\cdots\otimes\xi_{n})=(\varphi(a)\xi_{1})\otimes\xi_{2}\otimes\cdots\otimes\xi_{n}$,
$a\in M$. Note that $E^{\otimes0}$ is defined to be $M$, with $\langle\xi,\eta\rangle:=\xi^{*}\eta$,
and $\varphi_{0}$ given by left multiplication in $M$. The ultraweak
direct sum of the $E^{\otimes n}$, $n\geq0$, is called the \emph{Fock
space} of $E$ and is denoted $\mathcal{F}(E)$. It is, in an obvious
way, a right $W^{*}$-Hilbert module over $M$ and there is a left
action, $\varphi_{\infty}$, of $M$ on $\mathcal{F}(E)$, given by
the formula $\varphi_{\infty}(a)=\sum_{n\geq0}\varphi_{n}(a)$, $a\in M$.

The commutant \emph{in} $\mathcal{L}(E^{\otimes k})$ of the image
of $M$ under $\varphi_{k}$ will be written $\varphi_{k}(M)^{c}$.
Note that $\varphi_{0}(M)^{c}$ is simply the center of $M$, $\mathfrak{Z}(M)$.
Also, the commutant in $\mathcal{L}(\mathcal{F}(E))$ of $\varphi_{\infty}(M)$
will be denoted by $\varphi_{\infty}(M)^{c}$.

Suppose, now, $F$ is another $W^{*}$-Hilbert module over a $W^{*}$-algebra,
say, $N$, and suppose $\sigma:M\to\mathcal{L}(F)$ is a normal $*$-homomorphism.
Then we we can form the $W^{*}$-Hilbert module over $N$, $E\otimes_{\sigma}F$,
where the subscript $\sigma$ is to indicate that the tensor product
is balanced over $M$ through $\sigma$. That is $\xi a\otimes\eta=\xi\otimes\sigma(a)\eta$,
$\xi\in E$, $\eta\in F$, and $a\in M$. We can also form Rieffel's
induced representation or homomorphism $\sigma^{E}$ which maps $\mathcal{L}(E)$
to $\mathcal{L}(E\otimes_{\sigma}F)$ via the formula $\sigma^{E}(A)=A\otimes I_{F}$.
(Rieffel developed the theory of induced representations of $C^{*}$-algebras
in \cite{R1974b}; additional facts about induced representations,
especially for $W^{*}$-algebras may be found in \cite{R1974a}.) 
\begin{defn}
\label{Def:_Intertwiner_space} The \emph{intertwiner space}, $\mathfrak{I}(\sigma^{E}\circ\varphi,\sigma)$,
is defined to be the collection of bounded operators $X$ from $E\otimes_{\sigma}F$
to $F$ that satisfy the equation 
\[
X\sigma^{E}\circ\varphi(a)=\sigma(a)X,\qquad a\in M.
\]

\end{defn}
At this level of generality, $\mathfrak{I}(\sigma^{E}\circ\varphi,\sigma)$
may reduce to the zero space, as we indicated in \cite[Remark 4.6]{Muhly1999}.
(The setting there was a $C^{*}$-setting, but with minor adjustments,
one can produce $W^{*}$-Hilbert modules in which the intertwiner
space is zero.) Nevertheless, $\mathfrak{I}(\sigma^{E}\circ\varphi,\sigma)$
usually carries a lot of information about $E$ and $\sigma$ as we
shall see. In particular, we note that $\mathfrak{I}(\sigma^{E}\circ\varphi,\sigma)^{*}:=\{X^{*}\mid X\in\mathfrak{I}(\sigma^{E}\circ\varphi,\sigma)\}=\mathfrak{I}(\sigma,\sigma^{E}\circ\varphi)$
is a $W^{*}$-correspondence over the commutant of $\sigma(M)$ in
$\mathcal{L}(F)$, i.e., $\mathfrak{I}(\sigma^{E}\circ\varphi,\sigma)^{*}$
is a $W^{*}$-correspondence over the relative commutant $\sigma(M)^{c}$
of $\sigma(M)$ in $\mathcal{L}(F)$. When $F$ is a Hilbert space,
we shall use the standard notation $\sigma(M)'$ instead of $\sigma(M)^{c}$.
The inner product of $X$ and $Y$ in $\mathfrak{I}(\sigma,\sigma^{E}\circ\varphi)$
is simply $\langle X,Y\rangle:=X^{*}Y$, and the left and right actions
of $\sigma(M)^{c}$ are given by the formula:
\[
a\cdot X\cdot b:=(I_{E}\otimes a)Xb,\qquad a,b\in\sigma(M)^{c},X\in\mathfrak{I}(\sigma,\sigma^{E}\circ\varphi).
\]

Note, too, that if $X$ and $Y$ are in $\mathfrak{I}(\sigma^{E}\circ\varphi,\sigma)$,
then while the composition of $X$ and $Y$ makes no sense, in general,
we do have the following important containment:
\begin{equation}
X(I_{E}\otimes Y)\in\mathfrak{I}(\sigma^{E^{\otimes2}}\circ\varphi_{2},\sigma).\label{eq:partial_products}
\end{equation}
If we apply this formula when $X=Y$, we arrive at a notion we call
tensorial powers.
\begin{defn}
\label{Def:powers} If $\sigma$ is a normal representation $M$ on
a Hilbert $W^{*}$-module $F$ and if $Z\in\mathfrak{I}(\sigma^{E}\circ\varphi,\sigma)$,
then the $k^{th}$ \emph{tensorial power} of $Z$, $Z^{(k)}$, is
the element of $\mathfrak{I}(\sigma^{E^{\otimes k}}\circ \varphi_k,\sigma)$ that
is defined inductively by setting $Z^{(1)}=Z$, and setting $Z^{(k+1)}:=Z(I_{E}\otimes Z^{(k)})$. 
\end{defn}
The formulas
\begin{equation}
Z^{(k)}=Z(I_{E}\otimes Z)(I_{E^{\otimes2}}\otimes Z)\cdots(I_{E^{\otimes(k-1)}}\otimes Z)\label{eq:powers_a}
\end{equation}
and

\begin{equation}
Z^{(k+l)}=Z^{(k)}(I_{E^{\otimes k}}\otimes Z^{(l)})\label{eq:powers_b}
\end{equation}
are immediate. They justify in part the use of the term ``tensorial
power''. Later, in Section \ref{Sec:Duality and the commutant}, we
will give further justification for the term by showing that in certain
settings, $\mathfrak{I}(\sigma^{E^{\otimes k}}\circ \varphi_k,\sigma)^{*}$ can be
identified with $(\mathfrak{I}(\sigma^{E}\circ \varphi,\sigma)^{*})^{\otimes k}$,
where the tensoring is balanced over $\sigma(M)^{c}$ in $\mathcal{L}(F)$,
and then $Z^{(k)*}$ may be viewed as $Z^{*}\otimes Z^{*}\otimes\cdots\otimes Z^{*}$. 

A companion to the notion of an intertwining operator is the notion
of an \emph{insertion operator:}
\begin{defn}
\label{Def:Insertion_operator}Let $\sigma:M\to\mathcal{L}(F)$ be
a representation of $M$ on a Hilbert $W^{*}$-module $F$. For $\xi\in E$,
we define the \emph{insertion operator} $L_{\xi}^{F}:F\to E\otimes_{\sigma}F$
by
\[
L_{\xi}^{F}\phi:=\xi\otimes\phi.
\]

\end{defn}
We omit the superscript $F$, if $F$ is clear from context. 
\begin{rem}
\label{Rmk:_Pimsner_calculation} As Pimsner notes \cite[p. 192]{Pi},
$L_{\xi}^{*}$ is given by the formula $L_{\xi}^{*}(\eta\otimes\zeta)=\sigma(\langle\xi,\eta\rangle)\zeta$,
$\eta\in E,\zeta\in F$. Consequently, $L_{\xi}^{*}L_{\xi}=\sigma(\langle\xi,\xi\rangle)$,
which shows that the norm of $L_{\xi}$ is dominated by $\Vert\xi\Vert$
in general, and equality holds if $\sigma$ is faithful. We note,
too, that $L_{\xi}L_{\eta}^{*}=(\xi\otimes\eta^{*})\otimes I_{F}$,
where $\xi\otimes\eta^{*}$ denotes the rank one element of $\mathcal{L}(E)$
defined by the formula $\xi\otimes\eta^{*}(\zeta):=\xi\langle\eta,\zeta\rangle$.
\end{rem}
Two special instances of $\mathfrak{I}(\sigma^{E}\circ\varphi,\sigma)$
play important rolls in this paper. For the first, we take $F$ to
be $\mathcal{F}(E)$ and $\sigma$ to be $\varphi_{\infty}$. Since
$E^{\otimes k+1}$ is \emph{defined} to be $E\otimes_{\varphi_{k}}E^{\otimes k}$,
$E\otimes_{\varphi_{\infty}}\mathcal{F}(E)$ may be identified with
$\sum_{k\geq1}E^{\otimes k}$, essentially by definition. Let $T$
be the map that embeds $E\otimes_{\varphi_{\infty}}\mathcal{F}(E)$
into $\mathcal{F}(E)$ as the direct sum $\sum_{k\geq1}E^{\otimes k}$.
It is manifest that $T$ belongs to $\mathfrak{I}(\varphi_{\infty}^{E}\circ\varphi,\varphi_{\infty})$.
Further, $T^{*}$, mapping $\mathcal{F}(E)$ onto $E\otimes_{\varphi_{\infty}}\mathcal{F}(E)$,
has kernel $M$, viewed as the summand of index zero in $\mathcal{F}(E)$,
and acts like the identity on $\sum_{k\geq1}E^{\otimes k}$, viewed
as $E\otimes_{\varphi_{\infty}}\mathcal{F}(E)$. Thus, in particular
$T$ is an isometry with $T^{*}T$ equal to the identity on $E\otimes_{\varphi_{\infty}}\mathcal{F}(E)$
and $TT^{*}=I_{\mathcal{F}(E)}-P_{0}$, where $P_{0}$ is the projection
of $\mathcal{F}(E)$ onto $E^{\otimes0}$.
\begin{defn}
\label{Def:_Tautological_isometry}The map $T\in\mathfrak{I}(\varphi_{\infty}^{E}\circ\varphi,\varphi_{\infty})$
just described is called the \emph{tautological isometry} in $\mathfrak{I}(\varphi_{\infty}^{E}\circ\varphi,\varphi_{\infty})$. 
\end{defn}
The second instance which is important for our discussion is the case
when $\sigma$ is an ordinary normal representation of $M$ on a Hilbert
space $H_{\sigma}$, which we view as a $W^{*}$-correspondence over
$\mathbb{C}.$ In this case, we typically write elements of $\mathfrak{I}(\sigma^{E}\circ\varphi,\sigma)$
with fraktur letters from the end of the alphabet. Also, in conformance
with earlier work of ours, we will at times abbreviate $\mathfrak{I}(\sigma^{E}\circ\varphi,\sigma)$
by $E^{\sigma*}$(see \cite[Definition 3.1]{Muhly2004a}). When $M=\mathbb{C}=E$
and when $\sigma$ is the one-dimensional representation of $M$ on
$\mathbb{C}$ viewed as a one-dimensional Hilbert space, $\mathfrak{I}(\sigma^{E}\circ\varphi,\sigma)$
is naturally identified with $\mathbb{C}$ in such a way that the
function theory we want to stress coincides with classical function
theory on the complex plane. Thus, we take the view that in general,
the intertwiner spaces should be viewed as noncommutative affine spaces.
\begin{defn}
\label{Def:Creation_operator}The \emph{creation operator} on $\mathcal{F}(E)$
determined by $\xi\in E$ and denoted by $T_{\xi}$, is defined to
be $TL_{\xi}^{\mathcal{F}(E)}$.
\end{defn}
Of course, this really is the same definition that we have given in
\cite[p. 394]{Muhly1998a}. The reason for the formulation here
is that we want to accentuate a feature of the products of creation
operators. For this purpose, we will write $L_{\xi}$ for $L_{\xi}^{\mathcal{F}(E)}$,
when $\xi\in E^{\otimes k}$. Then, if $\xi_{1},\xi_{2}\in E$, 
\begin{multline*}
T_{\xi_{1}}T_{\xi_{2}}\eta=TL_{\xi_{1}}TL_{\xi_{2}}\eta=T(\xi_{1}\otimes T(\xi_{2}\otimes\eta))\\
=T(I_{E}\otimes T)(\xi_{1}\otimes\xi_{2}\otimes\eta)\\
=T^{(2)}L_{\xi_{1}\otimes\xi_{2}}\eta.
\end{multline*}
 It follows easily that $T_{\xi}=T^{(k)}L_{\xi}$, for all $\xi\in E^{\otimes k}$,
and that for $\xi_{1}\in E^{\otimes k}$ and $\xi_{2}\in E^{\otimes l}$,
\[
T_{\xi_{1}}T_{\xi_{2}}=T^{(k)}(I_{E^{\otimes k}}\otimes T^{(l)})L_{\xi_{1}\otimes\xi_{2}}=T^{(k+l)}L_{\xi_{1}\otimes\xi_{2}}=T_{\xi_{1}\otimes\xi_{2}}.
\]
 
\begin{defn}
\label{Def:_Unweighted_algebras}The \emph{algebraic tensor algebra}
of $E$ is simply the subalgebra of $\mathcal{L}(\mathcal{F}(E))$
generated by $\varphi_{\infty}(M)$ and $\{T_{\xi}\mid\xi\in E\}$.
It is denoted $\mathcal{T}_{0+}(E)$. Its norm closure in $\mathcal{L}(\mathcal{F}(E))$
is called \emph{the tensor algebra} of $E$ and is denoted $\mathcal{T}_{+}(E)$.
The ultraweak closure of $\mathcal{T}_{0+}(E)$ is called the \emph{Hardy
algebra} of $E$ and is denoted $H^{\infty}(E)$. Finally, the \emph{Toeplitz
algebra} of $E$ is the $C^{*}$-subalgebra of $\mathcal{L}(\mathcal{F}(E))$
generated by $\mathcal{T}_{0+}(E)$.
\end{defn}
There is another algebra of importance for the theory, the Cuntz-Pimsner
algebra of $E$, but we shall discuss this later. Also, it is not
difficult to see that under our hypotheses on $E$ - that it is full
and $\varphi$ is faithful - the $W^{*}$-subalgebra of $\mathcal{L}(\mathcal{F}(E))$
generated by $H^{\infty}(E)$ is all of $\mathcal{L}(\mathcal{F}(E))$.

For the purposes of this paper, the principal result of \cite{Muhly1998a},
Theorem 3.10, as amended by \cite[Theorem 3.4]{Muhly2004a}, can be
formulated as follows.
\begin{thm}
\label{thm:Main_thm_1998a}Let $\rho:\mathcal{T}_{+}(E)\to B(H)$
be a completely contractive representation with the property that
$\sigma:=\rho\circ\varphi_{\infty}$ is an ultraweakly continuous,
unital representation of $M$. Then $\sigma$ is a normal $*$-representation
of $M$, and the operator $\mathfrak{z}:E\otimes_{\sigma}H\to H$,
defined by the formula 
\begin{equation}
\mathfrak{z}(\xi\otimes h):=\rho(T_{\xi})h,\qquad\xi\otimes h\in E\otimes_{\sigma}H,\label{eq:Def_of_mathfrak(z)}
\end{equation}
lies in $\mathfrak{I}(\sigma^{E}\circ\varphi,\sigma)$ and has norm
at most one. Conversely, given a pair $(\sigma,\mathfrak{z})$, where
$\sigma$ is a normal $*$-representation of $M$ on a Hilbert space
$H_{\sigma}$ and $\mathfrak{z}$ is an element of $\mathfrak{I}(\sigma^{E}\circ\varphi,\sigma)$
of norm at most one, then there is a unique completely contractive
representation $\rho$ of $\mathcal{T}_{+}(E)$ on $H_{\sigma}$ such
that $\rho\circ\varphi_{\infty}=\sigma$ and such that $\rho$ satisfies
\ref{eq:Def_of_mathfrak(z)}.
\end{thm}
Observe that equation \ref{eq:Def_of_mathfrak(z)} can be written
\[
\rho(T_{\xi})=\mathfrak{z}L_{\xi}
\]
where $L_{\xi}$ is the insetion operator defined in Definition \ref{Def:Insertion_operator}.
This ability to factor $\rho(T_{\xi})$ is key to the function theory
that we developed in \cite{Muhly2013a}. Observe, also, that when
$M=\mathbb{C}=E$ and $\sigma$ is the one-dimensional representation
of $M$, then Theorem \ref{thm:Main_thm_1998a} essentially proves
that $\mathcal{T}_{+}(E)$ is the space of Toeplitz operators with
symbols that are continuous on the closed disc and analytic on the
interior and that the maximal ideal space of this algebra is the closed
unit disc in the complex plane. Thus in general, we denote the open
unit ball in $\mathfrak{I}(\sigma^{E}\circ\varphi,\sigma)$ by $D(E,\sigma)$
and refer to it as the (open) \emph{unit disc} in $\mathfrak{I}(\sigma^{E}\circ\varphi,\sigma)$.
Then $\overline{D(E,\sigma)}$ parametrizes the completely contractive
representations $\rho$ of $\mathcal{T}_{+}(E)$ such that $\rho\circ\varphi_{\infty}=\sigma$.
Also, because of Theorem \ref{thm:Main_thm_1998a}, we adopt the notation
$\sigma\times\mathfrak{z}$ for the representation of $\mathcal{T}_{+}(E)$
determined by $\mathfrak{z}\in\overline{D(E,\sigma)}$. Then for $F\in\mathcal{T}_{+}(E)$,
we may define $\widehat{F}_{\sigma}$ on $\overline{D(E,\sigma)}$
by the formula
\begin{equation}
\widehat{F}_{\sigma}(\mathfrak{z}):=(\sigma\times\mathfrak{z})(F).\label{eq:F-hat}
\end{equation}
In \cite{Muhly2013} we called $\widehat{F}_{\sigma}$ the ($\sigma-)$
\emph{Berezin transform} and showed that it is continuous on $\overline{D(E,\sigma)}$
and Frechet analytic on $D(E,\sigma)$. In fact, we showed that each
$F\in H^{\infty}(E)$ can be represented as a ``power series'' that
converges throughout $D(E,\sigma)$. 

\section{The Setting\label{sec:The-Setting}}

We began with the question: ``How might the theory described above
be generalized?'' We observed that on the algebraic level, at least,
much of what we have done makes no special use of the special features
of $T$. We could just as well have chosen \emph{any} operator $W$,
say, in $\mathfrak{I}(\varphi_{\infty}^{\mathcal{F}(E)}\circ\varphi,\varphi_{\infty})$
and we could have formed new creation operators $W_{\xi}$, $\xi\in E$,
by the formula 
\[
W_{\xi}:=WL_{\xi}^{\mathcal{F}(E)}.
\]
We could then study the norm-closed subalgebra of $\mathcal{L}(\mathcal{F}(E))$
that is generated by $\varphi_{\infty}(M)$ and $\{W_{\xi}\}_{\xi\in E}$.
At this level of generality nothing specific can really be said. Indeed,
if $M=\mathbb{C}=E$, then $\mathcal{F}(E)$ is simply $\ell^{2}(\mathbb{Z}_{+})$.
Of course, then $\varphi_{\infty}$ just represents $\mathbb{C}$
as scalar multiples of the identity on $\ell^{2}(\mathbb{Z}_{+})$,
so after we identify $E\otimes_{\varphi_{\infty}}\mathcal{F}(E)=\mathbb{C}\otimes_{\mathbb{C}}\ell^{2}(\mathbb{Z}_{+})$
with $\ell^{2}(\mathbb{Z}_{+})$, the sets $\mathfrak{I}(\varphi_{\infty}^{\mathcal{F}(E)}\circ\varphi,\varphi_{\infty})$
and $B(\ell^{2}(\mathbb{Z}_{+}))$ are essentially the same. Further,
after making the identifications we have here, $L_{\xi}$ is identified
with $\xi U_{+}$, since $E=\mathbb{C}$, where $U_{+}$ is the unilateral
shift on $\ell^{2}(\mathbb{Z}_{+})$. It follows that $W_{\xi}=\xi WU_{+}$
for all $\xi\in E$. Since $WU_{+}$can be any operator on the Hilbert
space $\ell^{2}(\mathbb{Z}_{+})$ we are effectively asking for the
(completely contractive) representation theory of \emph{any} singly
generated norm-closed algebra of operators on Hilbert space - something
that is well beyond reach at this time. However, if we make special
choices for $W$, then interesting, tractable, extensions of the theory
outlined above emerge.

One appealing natural and ``familiar'' choice for $W\in\mathfrak{I}(\varphi_{\infty}^{E}\circ\varphi,\varphi_{\infty})$
is to begin with a diagonal matrix $D$ on $\mathcal{F}(E)$ that
lies in $\varphi_{\infty}(M)^{c}$, and take $W=DT$.
\begin{defn}
\label{Def:_Weight_seq_Weighted_intertwiner}A  sequence $Z=\{Z_{k}\}_{k\geq0}$ of elements $Z_{k}\in\varphi_{k}(M)^{c}$ will be called a \emph{weight sequence}
on $(M,E)$ in case:
\begin{enumerate}
\item $Z_{0}=I_{M}$, 
\item $\sup\Vert Z_{k}\Vert<\infty$, and
\item each $Z_{k}$ is invertible. 
\end{enumerate}
We write $D={\rm {diag}}[Z_{0},Z_{1},Z_{2},\cdots]$ for the diagonal
matrix on $\mathcal{F}(E)$ determined by $Z$ and we write $W=W^{Z}:=DT$
for the intertwiner in $\mathfrak{I}(\varphi_{\infty}^{E}\circ\varphi,\varphi_{\infty})$
determined by $Z$. We call $W$ the \emph{weighted intertwiner} associated
to $Z$.
\end{defn}
Because $\sup\Vert Z_{k}\Vert<\infty$, the diagonal operator $D$
and the weighted intertwiner are clearly bounded, with $\Vert D\Vert=\sup_{k\geq0}\Vert Z_{k}\Vert$
and $\Vert W\Vert=\sup_{k\geq1}\Vert Z_{k}\Vert$. Further, because
$D$ lies in $\varphi_{\infty}(M)^{c}$, it is clear that $W$ lies
in $\mathfrak{I}(\varphi_{\infty}^{E}\circ\varphi,\varphi_{\infty})$.
The assumption that each $Z_{k}$ is invertible is one of convenience.
It is not absolutely essential, but to relax it leads to digressions
which are tangential to our primary objective. We do not assume, however,
that $D$ is invertible - except in special circumstances, where we
introduce this assumption as an explicit hypothesis.
\begin{defn}
\label{Def:_Weighted_shifts}Let $Z$ be a weight sequence on $(M,E)$
and let $W=W^{Z}$ be the associated weighted intertwiner in $\mathfrak{I}(\varphi_{\infty}^{E}\circ\varphi,\varphi_{\infty})$.
Then for $\xi\in E$, the \emph{weighted shift} determined by $\xi$
is the operator in $\mathcal{L}(\mathcal{F}(E))$ defined by the formula
\[
W_{\xi}^{Z}=W_{\xi}:=WL_{\xi}^{\mathcal{F}(E)}.
\]

\end{defn}
Evidently, $W_{\xi}=DT_{\xi}$ where $D$ is the diagonal operator
associated to $Z$. Consequently, we have $W_{\xi}^{*}=T_{\xi}^{*}D^{*}$.
\begin{defn}
\label{Zalgebras} Let $Z=\{Z_{k}\}_{k\geq0}$ be a weight sequence
on $(M,E)$. The subalgebra of $\mathcal{L}(\mathcal{F}(E))$ generated
by $\varphi_{\infty}(M)\cup\{W_{\xi}^{Z}:\xi\in E\}$ is called the
\emph{weighted algebraic tensor }algebra of $E$ determined by $Z$
and will be denoted $\mathcal{T}_{0+}(E,Z)$. The norm-closure of
$\mathcal{T}_{0+}(E,Z)$ in $\mathcal{L}(\mathcal{F}(E))$ is called,
simply, the \emph{weighted tensor algebra of $E$ determined by} $Z$
and will be denoted $\mathcal{T}_{+}(E,Z)$. The $C^{*}$-algebra
generated by $\mathcal{T}_{+}(E,Z)$ will be denoted $\mathcal{T}(E,Z)$
and called the \emph{weighted Toeplitz algebra} of $E$ determined
by $Z$. The ultraweak closure of $\mathcal{T}_{+}(E,Z)$ in $\mathcal{L}(\mathcal{F}(E))$
is the \emph{weighted Hardy algebra} of $E$ determined by $Z$, and
will be written $H^{\infty}(E,Z)$.\end{defn}
\begin{rem}
\label{Rm:Products} It will be helpful to have formulas at hand for
products of weighted shifts. We have defined $W_{\xi}$ as $WL_{\xi}^{\mathcal{F}(E)}=DT_{\xi}.$
So for $\xi\in E$ and $\theta\in E^{\otimes k}$, we have find that
$W_{\xi}\theta=Z_{k+1}(\xi\otimes\theta)$. Therefore, if $\xi_{1},\ldots,\xi_{m}$
are in $E$, a straightforward calculation shows that 
\begin{multline}
W_{\xi_{1}}\cdots W_{\xi_{m}}\theta=Z_{m+k}(I_{E}\otimes Z_{m+k-1})\\\times\cdots(I_{E^{\otimes(m-1)}}\otimes Z_{k+1})(\xi_{1}\otimes\cdots\xi_{m}\otimes\theta).\label{eq:m-products_of_weighted_creators}
\end{multline}
In general and in a fashion similar to that which we have developed
for intertwiners, we shall write 
\[
Z^{(n)}:=Z_{n}(I_{E}\otimes Z_{n-1})(I_{E^{\otimes2}}\otimes Z_{n-2})\cdots(I_{E^{\otimes(n-1)}}\otimes Z_{1}).
\]
Observe that $Z^{(n)}\in\varphi_{n}(M)^{c}$ since $Z_{n}\in\varphi_{n}(M)^{c}$
and $(I_{E}\otimes Z_{n-1})(I_{E^{\otimes2}}\otimes Z_{n-2})\cdots(I_{E^{\otimes(n-1)}}\otimes Z_{1})\in\varphi_{n}(M)^{c}$
for any choice of elements $Z_{k}\in\mathcal{L}(E^{\otimes k})$,
$k=1,2,\cdots,n-1$. It follows that \eqref{eq:m-products_of_weighted_creators}
may be rewritten as 
\[
W_{\xi_{1}}\cdots W_{\xi_{m}}\theta=Z^{(m+k)}(I_{E^{\otimes m}}\otimes Z^{(k)})^{-1}(\xi_{1}\otimes\cdots\xi_{m}\otimes\theta).
\]
Consequently, we find that for $\xi\in E^{\otimes m}$, we may \emph{define}
$W_{\xi}$ via the formula 
\begin{equation}
W_{\xi}\theta=Z^{(m+k)}(I_{E^{\otimes m}}\otimes Z^{(k)})^{-1}(\xi\otimes\theta),\qquad\theta\in E^{\otimes k},\label{Wxi}
\end{equation}
and once this is done, a straightforward calculation shows that 
\begin{equation}
W_{\xi_{1}}W_{\xi_{2}}=W_{\xi_{1}\otimes\xi_{2}},\qquad\xi_{1}\in E^{\otimes n},\xi_{2}\in E^{\otimes m}.\label{eq:Product_formula}
\end{equation}
In particular, then, $W_{\xi}\in\mathcal{T}_{0+}(E,Z)$ for every
$\xi\in E^{\otimes k}$ and $k\geq1$. The formulas just established
show that the linear span of the $W_{\xi}$, as $\xi$ runs over $E^{\otimes k}$,
$k\geq1,$ together with all the $\varphi_{\infty}(a)$, $a\in M,$
is $\mathcal{T}_{0+}(E,Z)$.
\end{rem}
In the theory of single weighted shifts it is known that if the weighted
shift is injective, i.e. if the weights are all different from zero,
one can view the operator as an ordinary (unweighted) shift on a ``weighted
$\ell_{2}$ space\textquotedbl{}. An analogous statement holds here,
since we are assuming all the weights $\{Z_{k}\}$ are invertible.

Simply define $F(k)$, $k\geq1$, to be the $W^{*}$-correspondence
that algebraically is $E^{\otimes k}$ but has the $M$-valued inner
product defined by 
\begin{equation}
\langle\xi,\eta\rangle_{Z}=\langle(Z^{(k)})\xi,(Z^{(k)})\eta\rangle,\qquad\xi,\eta\in E^{\otimes k}.
\end{equation}
Since $Z^{(k)}\in\varphi_{k}(M)^{c}$ for each $k$, it is easy to
check that this equation defines a bonafide inner product on $F(k)$
with respect to which $F(k)$ is a $W^{*}$-correspondence over $M$.
For $k=0$, we set $F(0)=M$ and we form 
\[
\mathcal{F}_{Z}(E):=\sum_{k=0}^{\infty}\oplus F(k).
\]
On $\mathcal{F}_{Z}(E)$ we can define shifts by 
\[
T_{\xi}^{F}\eta=\xi\otimes\eta,\qquad\xi\in E^{\otimes n},\eta\in F(k).
\]

\begin{prop}
\label{prop: Weighted_Fock_space}The unique additive map $V:\mathcal{F}(E)\rightarrow\mathcal{F}_{Z}(E)$
that sends $\eta\in E^{\otimes k}$ to $(Z^{(k)})^{-1}\eta\in F(k)$
is a correspondence isomorphism and $V^{*}T_{\xi}^{F}V=W_{\xi}$ for
all $\xi\in E^{\otimes n}$. Further, $V^{*}\varphi_{\infty}^{F}(\cdot)V=\varphi_{\infty}(\cdot)$,
where $\varphi_{\infty}^{F}(a)=\sum_{k\geq0}\varphi_{k}(a)$, acting
on $\mathcal{F}_{Z}(E)$.
\end{prop}
The proof is a straightforward application of the formulas for multiplication
in $\mathcal{T}_{0+}(E,Z)$ established in Remark \ref{Rm:Products}.

{}
\begin{rem}
\label{Rm:_Algebraic_Extension}The formulas in Remark \ref{Rm:Products}
also reveal the sense in which we may assert that the Hilbert space
representations of $\mathcal{T}_{0+}(E,Z)$ are parametrized by the
pairs $(\sigma,\mathfrak{z})$ where $\sigma$ ranges over the normal
representations of $M$ on Hilbert space and for each $\sigma$, $\mathfrak{z}$
ranges over $\mathfrak{I}(\sigma^{E}\circ\varphi,\sigma)$. Indeed,
if $\sigma:M\to B(H)$ is a normal representation of $M$ on $H$
and if $\mathfrak{z}\in\mathfrak{I}(\sigma^{E}\circ\varphi,\sigma)$,
then the map $W_{\xi}\to\mathfrak{z}L_{\xi}^{H}:=\mathfrak{z}L_{\xi}$,
$\xi\in E$, is a bimodule map, which together with $\sigma$, extends
to a representation $\rho$ of $\mathcal{T}_{0+}(E,Z)$ on $H$. That
is, $W_{\varphi(a)\xi b}\to\mathfrak{z}L_{\varphi(a)\xi b}=\sigma(a)\mathfrak{z}L_{\xi}\sigma(b)$
and $\rho(W_{\xi_{1}\otimes\xi_{2}})=\rho(W_{\xi_{1}}W_{\xi_{2}})=\rho(W_{\xi_{1}})\rho(W_{\xi_{2}})=\mathfrak{z}L_{\xi_{1}}\mathfrak{z}L_{\xi_{2}}=\mathfrak{z}^{(2)}L_{\xi_{1}\otimes\xi_{2}}$
for all $\xi_{1},\xi_{2}\in E$. As a result, we see that for all
$\xi\in\mathfrak{I}(\sigma^{E^{\otimes k}}\circ\varphi_{k},\sigma)$,
$\rho(W_{\xi})=\mathfrak{z}^{(k)}L_{\xi}$. Conversely, given a representation
$\rho$ of $\mathcal{T}_{0+}(E,Z)$ with the property that $\rho\circ\varphi_{\infty}:=\sigma$
is a normal representation of $M$ and $\xi\to\rho(T_{\xi}^{F})$
is a completely bounded linear map on $E$, then $\xi\to\rho(T_{\xi}^{F})$
is a completely bounded bimodule map of $E$ on $H$, and so by \cite[Lemma 3.5]{Muhly1998a},
there is a $\mathfrak{z}\in\mathfrak{I}(\sigma^{E}\circ\varphi,\sigma)$
such that $\rho(T_{\xi}^{F})=\mathfrak{z}L_{\xi}$. Further, the cb-norm
of $\xi\to\rho(T_{\xi}^{F})$ is $\Vert\mathfrak{z}\Vert$. Our problem,
then, is to determine when a $\mathfrak{z}\in\mathfrak{I}(\sigma^{E}\circ\varphi,\sigma)$
is such that the map $W_{\xi}\to\mathfrak{z}L_{\xi}$ extends to a
completely contractive representation of $\mathcal{T}_{+}(E,Z)$.
\end{rem}
We are unable to determine when such an extension is possible without
additional assumptions on $Z$. However, we were led to our assumptions
asking somewhat more about possible extensions: We would like to identify
those $\mathfrak{z}\in\mathfrak{I}(\sigma^{E}\circ\varphi,\sigma)$
that give rise to representations that may be realized as parts of
induced representations. 
\begin{defn}
\label{Def:Induced_Rep}Let $\sigma$ be a normal representation of
$M$ on the Hilbert space $H_{\sigma}$ and let $\sigma^{\mathcal{F}(E)}$
be the representation of $\mathcal{L}(\mathcal{F}(E))$ on $\mathcal{F}(E)\otimes_{\sigma}H_{\sigma}$
that is induced by $\sigma$ in the sense of Rieffel \cite{R1974b}.
We call the restrictions of $\sigma^{\mathcal{F}(E)}$ to $\mathcal{T}_{+}(E,Z)$,
$\mathcal{T}(E,Z)$ and $H^{\infty}(E,Z)$, the representations of
$\mathcal{T}_{+}(E,Z)$, $\mathcal{T}(E,Z)$ and $H^{\infty}(E,Z)$,
respectively, that are \emph{induced by $\sigma$}, and we denote
all of them by $\sigma^{\mathcal{F}(E)}$.
\end{defn}
The multiple usage of the terminology and the notation $\sigma^{\mathcal{F}(E)}$
should cause no difficulty in context. 

Note that in the case and when $M=\mathbb{C}=E$, $\sigma$ is $1$-dimensional
and there are no weights, $\mathcal{F}(E)\otimes_{\sigma}H_{\sigma}$
is just $\ell^{2}(\mathbb{Z}_{+})$ and $\sigma^{\mathcal{F}(E)}$
in essence is the identity representation of $B(\ell^{2}(\mathbb{Z}_{+}))$.
Consequently, when $\sigma$ is the (essentially) unique representation
of $\mathbb{C}$ on $n$-dimensional Hilbert space $H$, $1\leq n\leq\infty$,
and when $\xi\in E=\mathbb{C}$ is the number $1$, then $\sigma^{\mathcal{F}(E)}(T_{\xi})$
is the unilateral shift of multiplicity $n$ acting on $\ell^{2}(\mathbb{Z}_{+})\otimes H$.
If, in this setting, $Z$ is a weight sequence of numbers, and if
$D$ is the diagonal matrix they determine, then $\sigma^{\mathcal{F}(E)}(W_{\xi})=\sigma^{\mathcal{F}(E)}(DT_{\xi})$
is simply the $n$-fold multiple of the weighted shift $W_{\xi}$
acting on $\ell^{2}(\mathbb{Z}_{+})\otimes H$. It was in this context
that M\"{u}ller investigated operators $T\in B(H)$ that can be modeled
as a part of $W_{\xi}$ acting on $\ell^{2}(\mathbb{Z}_{+})\otimes H$.
What he discovered was that when $T\in B(H)$ can be realized as part
of $W_{\xi}$ acting on $\ell^{2}(\mathbb{Z}_{+})\otimes H$, the
isometry embedding $H$ into $\ell^{2}(\mathbb{Z}_{+})\otimes H$
that effects the realization is, on close inspection, related to the
classical Poisson kernel. In a sense, this is not too much of a surprise.
Indeed, since Foia\c{s}'s modification of Rota's theorem \cite{Foias1963},
the work of de Branges and Rovnyak \cite{deBrangeRovnyak1964} and
the growth of model theory in the 1960's, it has been understood that
the isometry that realizes the canonical model of a so-called $C_{\cdot0}$
contraction is connected to the Poisson kernel. And of course, in
\cite{Popescu2010}, Popescu based his extension of M\"{u}ller's
work on a generalized Poisson kernel.

Now in \cite{Muhly2009}, we developed a Poisson kernel for representations
of Hardy algebras in the unweighted case. Although we did not highlight
this fact in \cite{Muhly2009}, what appears to be novel about our
discovery, and what is especially relevant for our current study,
is the explicit role played by certain completely positive maps that
lurk in the background and are not readily visible in situations studied
heretofore. To motivate choices we will make shortly, it will be helpful
to call attention to some salient features of our construction. 
\begin{defn}
\label{Def:Cauchy_and_Poisson_kernels} (See \cite[Definitions 6 and 8]{Muhly2009}.)
Let $\sigma:M\to B(H)$ be a normal representation of $M$, and let
$\mathfrak{z}\in\mathfrak{I}(\sigma^{E}\circ\varphi,\sigma)$ have
norm less than $1$. Then the \emph{Cauchy kernel }evaluated at $\mathfrak{z}$
is the map from $H$ to $\mathcal{F}(E)\otimes_{\sigma}H$ defined
by the matrix
\begin{equation}
C(\mathfrak{z}):=[\mathfrak{z}^{(0)*},\mathfrak{z}^{(1)*},\mathfrak{z}^{(2)*},\mathfrak{z}^{(3)*},\cdots]^{\intercal}\label{eq:unweighted Cauchy kernel}
\end{equation}
The \emph{Poisson kernel} evaluated at $\mathfrak{z}$ is the operator
\begin{equation}
K(\mathfrak{z}):=(I_{\mathcal{F}(E)}\otimes\Delta_{*}(\mathfrak{z}))C(\mathfrak{z}),\label{eq:unweighted_Poisson_kernel}
\end{equation}
mapping $H$ to $\mathcal{F}(E)\otimes_{\sigma}H$, where $\Delta_{*}(\mathfrak{z}):=(I_{H}-\mathfrak{z}\mathfrak{z}^{*})^{\frac{1}{2}}$. 
\end{defn}
Simple calculations show that $C(\mathfrak{z})$ intertwines $\sigma$
and the induced representation $\sigma^{\mathcal{F}(E)}\circ\varphi_{\infty}$,
i.e., $C(\mathfrak{z})\in\mbox{\ensuremath{\mathfrak{I}}(\ensuremath{\sigma},\ensuremath{\sigma^{\mathcal{F}(E)}\circ\varphi_{\infty}})}$.
But in fact, the proof of Lemma 12 of \cite{Muhly2009} shows that
\begin{equation}
C(\mathfrak{z})(\sigma\times\mathfrak{z}(F))=\sigma^{\mathcal{F}(E)}(F)C(\mathfrak{z})=(F\otimes I_{H})C(\mathfrak{z}),\qquad F\in H^{\infty}(E).\label{eq:Cauchy_transform}
\end{equation}
On the other hand, as we pointed out in \cite[Proposition 10]{Muhly2009},
because $I_{\mathcal{F}(E)}\otimes\Delta_{*}(\mathfrak{z})$ commutes
with $\sigma^{\mathcal{F}(E)}(\mathcal{L}(\mathcal{F}(E))$, $K(\mathfrak{z})$
is an \emph{isometry} from $H$ into $\mathcal{F}(E)\otimes_{\sigma}H$
such that
\begin{equation}
K(\mathfrak{z})(\sigma\times\mathfrak{z}(F))=(F\otimes I)K(\mathfrak{z}),\label{eq:Poisson_transform_1}
\end{equation}
for all $F\in H^{\infty}(E)$. If one takes adjoints in Equation \eqref{eq:Poisson_transform_1},
one sees that $\sigma\times\mathfrak{z}$ is obtained from the induced
representation $\sigma^{\mathcal{F}(E)}$ by restricting the latter
to a co-invariant subspace. Thus, that equation is somewhat stronger
than the equation

\begin{equation}
\sigma\times\mathfrak{z}(F)=K(\mathfrak{z})^{*}(F\otimes I)K(\mathfrak{z}),\label{eq:Poisson_transform_2}
\end{equation}
which simply says that $\sigma\times\mathfrak{z}$ is the compression
of $\sigma^{\mathcal{F}(E)}$ to a semi-invariant subspace. 

The point that we want to emphasize here is that \emph{the intertwiner
$\mathfrak{z}$ induces a completely positive map $\Theta_{\mathfrak{z}}$
on the commutant of $\sigma(M)$, $\sigma(M)'$; vis. $\Theta_{\mathfrak{z}}(a):=\mathfrak{z}(I_{E}\otimes a)\mathfrak{z}^{*}$,
$a\in\sigma(M)'$. }We have used this formula a number of places in
our work (see, e.g., \cite{Muhly2002a,Muhly2003,Muhly2007,Muhly2011a}).
Its importance lies in the fact that the powers of $\Theta_{\mathfrak{z}}$
can be expressed in terms of $\mathfrak{z}$ via the formula
\[
\Theta_{\mathfrak{z}}^{k}(a)=\mathfrak{z}^{(k)}(I_{E^{\otimes k}}\otimes a)\mathfrak{z}^{(k)*}. 
\]
Further, since $\Vert\Theta_{\mathfrak{z}}\Vert=\Vert\mathfrak{z}\Vert^{2}$,
the series 
\begin{equation} \label{eq:_Potential_Thetaz}
\Gamma:=\mbox{\ensuremath{\iota}+\ensuremath{\Theta}}_{\mathfrak{z}}+\Theta_{\mathfrak{z}}^{2}+\cdots,
\end{equation}where $\iota$ denotes the identity map on $\sigma(M)'$, converges
in the norm on the space of bounded linear maps on $\sigma(M)'$,
when $\Vert\mathfrak{z}\Vert<1$. Further, still, in this case, the
fact that $C(\mathfrak{z})\in\mathfrak{I}(\sigma,\sigma^{\mathcal{F}(E)}\circ\varphi_{\infty})$
enables one to factor $\Gamma$ as
\[
\Gamma(a)=C(\mathfrak{z})^{*}(I_{\mathcal{F}(E)}\otimes a)C(\mathfrak{z}),\qquad a\in\sigma(M)'.
\]
Of course, $\Gamma$ is the resolvent $(\iota-\Theta_{\mathfrak{z}})^{-1}$.
Since $\Theta_{\mathfrak{z}}$ and $\Gamma$ are completely positive
maps, it is perhaps more appropriate to view $\Theta_{\mathfrak{z}}$
as a noncommutative Markov kernel (or operator) and then $\Gamma$
is a noncommutative analogue of its potential kernel (or operator)
from probability theory \cite[Definition 2.1.3]{Revuz1984}. The analogy becomes especially apt when $\Vert \mathfrak{z} \Vert = 1$ because, like potentials in probability theory, the convergence of the series \eqref{eq:_Potential_Thetaz} becomes problematic when the kernel involved has norm $1$. The more immediate 
point for us is that in equation \eqref{eq:Poisson_transform_2} the
normalization factor, $I_{\mathcal{F}(E)}\otimes\Delta_{*}(\mathfrak{z})$,
may be expressed in terms of $\Theta_{\mathfrak{z}}$ - $\Delta_{*}(\mathfrak{z})^{2}=(\iota-\Theta_{\mathfrak{z}})(I)$
- and this underwrites the validity of equation \eqref{eq:Poisson_transform_2}.
In fact, a careful reading of M\"{u}ller's \cite[Theorem 2.2]{Muller1988}
leads to the following theorem, which is the starting point of our
analysis.
\begin{thm}
\label{thm:Muller's_Thm2.2}Assume the weight sequence $Z=\{Z_{k}\}_{k\geq0}$
has the property that $Z^{(n)}=Z_{n}(I_{E}\otimes Z_{n-1})\cdots(I_{E^{\otimes n-1}}\otimes Z_{1})$
is positive for each $n$, and let $R_{n}:=(Z^{(n)})^{-1}$. Let $\sigma$
be a normal representation of $M$ on a Hilbert space $H$ and assume
that $\mathfrak{z}\in\mathfrak{I}(\sigma^{E}\circ\varphi,\sigma)$
has the property that the series 
\begin{equation}
\sum_{k=0}^{\infty}\mathfrak{z}^{(k)}(R_{k}^{2}\otimes I_{H})\mathfrak{z}^{(k)*}\label{eq:theta^R}
\end{equation}
converges ultraweakly in $B(H)$. Then: 
\begin{enumerate}
\item For every $a\in\sigma(M)'$, the series $\sum_{k=0}^{\infty}\mathfrak{z}^{(k)}(R_{k}^{2}\otimes a)\mathfrak{z}^{(k)*}$converges
ultrastrongly in $B(H)$ and defines a completely positive map $\Theta_{\mathfrak{z}}^{R}$
on $\sigma(M)'$; and
\item If $\Theta_{\mathfrak{z}}^{R}$ is the potential of a completely positive
map $\phi$ on $\sigma(M)'$, with norm $\Vert\phi\Vert\leq1$, i.e., if $\Theta_{\mathfrak{z}}^{R} = \iota + \phi + \phi^{2} + \cdots$, as in \eqref{eq:_Potential_Thetaz}, then
the operator $V:H\to\mathcal{F}(E)\otimes_{\sigma}H$ defined by the
formula, 
\[
V=(I_{\mathcal{F}(E)}\otimes(I-\phi)(I)^{\frac{1}{2}}){\rm {diag}}[1,R_{1},R_{2},\cdots]C(\mathfrak{z}),
\]
is an isometry with the property that 
\begin{equation}
V\sigma(a)=(\varphi_{\infty}(a)\otimes I_{H})V,\qquad a\in M,\label{eq:sigma_intertwine}
\end{equation}
and
\begin{equation}
V(\mathfrak{z}L_{\xi}^{H})=(W_{\xi}\otimes I_{H})V,\qquad\xi\in E.\label{eq: W_xi_intertwine}
\end{equation}

\end{enumerate}
In particular, $\sigma\times\mathfrak{z}$, which is defined initially
on $\mathcal{T}_{0+}(E,Z)$ extends to an ultraweakly continuous representation
of $H^{\infty}(E,Z)$ that is unitarily equivalent, via $V$, to the
compression of $\sigma^{\mathcal{F}(E)}(H^{\infty}(E,Z))$ to a co-invariant
subspace.
\end{thm}
The assumption that the $Z^{(n)}$ are all positive is not necessary.
However, it simplifies the argument and for the purposes we have in
mind for the theorem, the extra generality will not be used.
\begin{proof}
The series \eqref{eq:theta^R} converges ultraweakly if and only if
its partial sums are uniformly bounded. If the partial sums are bounded
by $N$, say, then for all $a\in\sigma(M)'$, the partial sums of
$\sum_{k=0}^{\infty}\mathfrak{z}^{(k)}(R_{k}^{2}\otimes aa^{*})\mathfrak{z}^{(k)*}$
are bounded by $N\Vert a\Vert^{2}$, and so this series converges
ultraweakly. But then $\sum_{k=0}^{\infty}\mathfrak{z}^{(k)}(R_{k}^{2}\otimes a)\mathfrak{z}^{(k)*}$
converges ultrastrongly. Observe that if $\Vert \mathfrak{z} \Vert < 1$, then the series converges \emph{in norm} and
then $\Theta_{\mathfrak{z}}^{R}$ may be written 
\[
\Theta_{\mathfrak{z}}^{R}(a)=C(\mathfrak{z})^{*}\mathcal{R}(I_{\mathcal{F}(E)}\otimes a)\mathcal{R}C(\mathfrak{z}),\qquad a\in\sigma(M)',
\]
where $\mathcal{R}$ is the diagonal matrix with respect to the direct
sum decomposition of $\mathcal{F}(E)$, ${\rm {diag}}[I,R_{1},R_{2},R_{3},\cdots]$. If $\Vert \mathfrak{z} \Vert = 1$, then the series may still converge in some sense even if the series representing $C(\mathfrak{z})$ fails to converge.  We encountered this kind of problem in \cite{Muhly2009}.  The Poisson kernel may converge when the Cauchy kernel doesn't.  In any case, when the series \eqref{eq:theta^R} converges, we will write $\Theta_{\mathfrak{z}}^{R}(a)$ as $C(\mathfrak{z})^{*}\mathcal{R}(I_{\mathcal{F}(E)}\otimes a)\mathcal{R}C(\mathfrak{z})$ without regard to the convergence of the Cauchy kernel.   If $\Theta_{\mathfrak{z}}^{R}$ is the potential of $\phi$, i.e.,
if $\Theta_{\mathfrak{z}}^{R}(a)=(I-\phi)^{-1}(a)$ for all $a\in\sigma(M)'$,
where $\phi$ is a completely positive map on $\sigma(M)'$ of norm
at most one, then $(I-\phi)(I)$ is a positive element of $\sigma(M)'$
and
\begin{eqnarray*}
V^{*}V & = & C(\mathfrak{z})^{*}\mathcal{R}(I_{\mathcal{F}(E)}\otimes(I-\phi)(I))\mathcal{R}C(\mathfrak{z})\\
 & = & \Theta_{\mathfrak{z}}^{R}\circ(I-\phi)(I)\\
 & = & I.
\end{eqnarray*}
Thus $V$ is an isometry.

The equation \eqref{eq:sigma_intertwine} is immediate because $I_{\mathcal{F}(E)}\otimes(I-\phi(I))^{\frac{1}{2}}$
commutes with $\sigma^{\mathcal{F}(E)}(\mathcal{L}(\mathcal{F}(E))$
and, because $R_{n}$ lies in $\varphi_{n}(M)^{c}$, $\mathfrak{z}^{(n)}R_{n}\in\mathfrak{I}(\sigma^{E^{\otimes n}}\circ\varphi_{n},\sigma)$.

The equation \eqref{eq: W_xi_intertwine} requires a bit more calculation.
Observe first that $Z_{k}=R_{k}^{-1}(I_{E}\otimes R_{k-1})$ for all
$k$ . To prove \eqref{eq: W_xi_intertwine} we need to show that
for every $k$ we have 
\begin{multline*}
(I_{E^{\otimes k}}\otimes(I-\phi)(I)^{\frac{1}{2}})(R_{k}\otimes I_{H})(\mathfrak{z}^{(k)})^{*}L_{\xi}^{*}\mathfrak{z}^{*}\\
=(W_{\xi}^{*}\otimes I_{H})(I_{E^{\otimes(k+1)}}\otimes(I-\phi)(I)^{\frac{1}{2}})(R_{k+1}\otimes I_{H})(\mathfrak{z}^{(k+1)})^{*}.
\end{multline*}
To do this we shall show that, for every $k$, 
\begin{multline*}
L_{\xi}\mathfrak{z}^{(k)}(R_{k}\otimes I_{H})(I_{E^{\otimes k}}\otimes(I-\phi)(I)^{\frac{1}{2}})\\
=(I_{E}\otimes\mathfrak{z}^{(k)})(R_{k+1}\otimes I_{H})(I_{E^{\otimes(k+1)}}\otimes(I-\phi)(I)^{\frac{1}{2}})(W_{\xi}\otimes I_{H})
\end{multline*}
as operators on $E^{\otimes k}\otimes H$. So, fix $\theta\in E^{\otimes k}$
and $h\in H$ and apply the right-hand-side of this equation to $\theta\otimes h$
to get 
\begin{multline}
(I_{E}\otimes\mathfrak{z}^{(k)})(R_{k+1}\otimes I_{H})(I_{E^{\otimes(k+1)}}\otimes(I-\phi)(I)^{\frac{1}{2}})(W_{\xi}\otimes I_{H})(\theta\otimes h)\\
=(I_{E}\otimes\mathfrak{z}^{(k)})(R_{k+1}\otimes I_{H})(I_{E^{\otimes(k+1)}}\otimes(I-\phi)(I)^{\frac{1}{2}})(Z_{k+1}(\xi\otimes\theta)\otimes h)\\
=(I_{E}\otimes\mathfrak{z}^{(k)})(R_{k+1}Z_{k+1}(\xi\otimes\theta)\otimes(I-\phi)(I)^{\frac{1}{2}}h)\\
=(I_{E}\otimes\mathfrak{z}^{(k)})(\xi\otimes R_{k}\theta\otimes(I-\phi)(I)^{\frac{1}{2}}h)
\end{multline}
where, in the last equality, we use the fact that $Z_{k+1}=R_{k+1}^{-1}(I_{E}\otimes R_{k})$.
Since $(I_{E}\otimes\mathfrak{z}^{(k)})(\xi\otimes R_{k}\theta\otimes(I-\phi)(I)^{\frac{1}{2}}h)$
equals $L_{\xi}\mathfrak{z}^{(k)}(R_{k}\otimes I_{H})(I_{E^{\otimes k}}\otimes(I-\phi)(I)^{\frac{1}{2}})(\theta\otimes h)$,
the proof of \eqref{eq: W_xi_intertwine} is complete.

For the final assertion of the theorem, simply note that equations
\eqref{eq:sigma_intertwine} and \eqref{eq: W_xi_intertwine} express
the fact that $V$ intertwines the generators of $\sigma\times\mathfrak{z(\mathcal{T}}_{0+}(E,Z))$
and the generators of $\sigma^{\mathcal{F}(E)}(\mathcal{T}_{0+}(E,Z))$.
From this, the assertion is an immediate consequence.
\end{proof}

\section{Special Potentials\label{sec:Special Potentials} }

Under natural minimal hypotheses, a weight sequence $Z$ always gives
rise to a weighted shift and a completely positive operator $\Theta_{\mathfrak{z}}^{R}$
for suitable $\mathfrak{z}$'s. It appears to be a difficult matter,
however, to determine if $\Theta_{\mathfrak{z}}^{R}$ is the potential
operator for a completely positive map on $\sigma(M)'$. Consequently,
it seems best to reorient our focus and seek completely positive maps
on $\sigma(M)'$ whose potentials may be expressed in terms of weighted
shifts. For the rest of the paper, we focus on one attractive source
of completely positive maps with this property. Our choice was inspired
by the paper of M\"{u}ller, particularly \cite[Theorem 2.2]{Muller1988},
the memoir of Popescu \cite{Popescu2009a}, and especially the following
simple calculation: If $\{x_{k}\}_{k\geq1}$ is a sequence of nonnegative
numbers, with $x_{1}>0$, such that the power series $\sum x_{k}z^{k}$
has a positive radius of convergence, then for $|z|$ sufficiently
small, we may write 
\[
\frac{1}{1-\sum_{k\geq1}x_{k}z^{k}}=\sum_{k\geq0}r_{k}^{2}z^{k}
\]
where the $r_{k}$ are strictly positive numbers. In fact, the $r_{k}$'s
may be computed easily in terms of the $x_{k}$'s. This calculation
will be contained in the proof of Theorem \ref{thm:The-potential-of-Phiz}.
\begin{defn}
\label{X}A sequence $X=\{X_{k}\}_{k=1}^{\infty}$ of operators will
be called \emph{admissible} in case it satisfies the following three
conditions:
\begin{enumerate}
\item $X_{k}\in\varphi_{k}(M)^{c}$.
\item $X_{k}\geq0$ for all $k\geq1$ and $X_{1}$ is invertible. 
\item $\limsup\Vert X_{k}\Vert^{\frac{1}{k}}<\infty$ .
\end{enumerate}
\end{defn}
{}
\begin{lem}
\label{lem:Estimate}If $X$ is an admissible sequence, with $\rho:=\limsup\Vert X_{k}\Vert^{\frac{1}{k}}$,
and if $\sigma:M\to B(H_{\sigma})$ is a normal representation, then
for all $\mathfrak{z}\in\mathfrak{I}(\sigma^{E}\circ\varphi,\sigma)$
with $\Vert\mathfrak{z}\Vert<\rho^{-\frac{1}{2}}$, 
\begin{equation}
\sum_{k=1}^{\infty}\Vert\mathfrak{z}^{(k)}(X_{k}\otimes I_{H_{\sigma}})\mathfrak{z}^{(k)*}\Vert<\infty.\label{eq:dominating_series}
\end{equation}
\end{lem}
\begin{proof}
The term $\Vert\mathfrak{z}^{(k)}(X_{k}\otimes I_{H_{\sigma}})\mathfrak{z}^{(k)*}\Vert$
is dominated by $\Vert\mathfrak{z}\Vert^{2k}\Vert X_{k}\Vert$. If
$\Vert\mathfrak{z}\Vert<\rho^{-\frac{1}{2}}$, then $\limsup\Vert\mathfrak{z}^{(k)}(X_{k}\otimes I_{H_{\sigma}})\mathfrak{z}^{(k)*}\Vert^{\frac{1}{k}}<1$
and so \eqref{eq:dominating_series} converges.\end{proof}
\begin{defn}
\label{Domain}If $X=\{X_{k}\}_{k=1}^{\infty}$ is an admissible sequence
and if $\sigma:M\to B(H_{\sigma})$ is a normal representation of
$M$, then $D(X,\sigma)$ is defined to be 
\[
\{\mathfrak{z}\in\mathfrak{I}(\sigma^{E}\circ\varphi,\sigma)\mid\Vert\sum_{k=1}^{\infty}\mathfrak{z}^{(k)}(X_{k}\otimes I_{H_{\sigma}})\mathfrak{z}^{(k)*}\Vert<1\},
\]
and is called the (open) \emph{disc associated to $X$ and $\sigma$}.
\end{defn}
Observe that by Lemma \ref{lem:Estimate}, $D(X,\sigma)$ contains
a non-empty, norm-open neighborhood of the origin in $\mathfrak{I}(\sigma^{E}\circ\varphi,\sigma)$.
\begin{lem}
\label{lem:Basic_CP_maps}Let $X=\{X_{k}\}_{k=1}^{\infty}$ be an
admissible sequence and let $\sigma:M\to B(H_{\sigma})$ be a normal
representation. Then every $\mathfrak{z}\in\overline{D(X,\sigma)}$
determines a completely positive map $\Phi_{\mathfrak{z}}$ on $\sigma(M)'$,
of norm at most one, via the formula
\begin{equation}
\Phi_{\mathfrak{z}}(a):=\sum_{k=1}^{\infty}\mathfrak{z}^{(k)}(X_{k}\otimes a)\mathfrak{z}^{(k)*},\qquad a\in\sigma(M)',\label{eq:Phi-sub-z}
\end{equation}

\end{lem}
called the \emph{prepotential} determined by ($X$ and) $\mathfrak{z}$. 
\begin{proof}
The assumption that $X_{k}$ is a positive element in $\varphi_{k}(M)^{c}$
guarantees that $\mathfrak{z}^{(k)}X_{k}^{\frac{1}{2}}\in\mathfrak{I}(\sigma^{E^{\otimes k}}\circ\varphi_{k},\sigma)$.
Consequently $a\to\mathfrak{z}^{(k)}X_{k}^{\frac{1}{2}}(I_{E^{\otimes k}}\otimes a)(\mathfrak{z}^{(k)}X_{k}^{\frac{1}{2}})^{*}=\mathfrak{z}^{(k)}(X_{k}\otimes a)\mathfrak{z}^{(k)*}$
is a completely positive map on $\sigma(M)'$. Then so are the partial
sums, $a\to\sum_{k=1}^{N}\mathfrak{z}^{(k)}(X_{k}\otimes a)\mathfrak{z}^{(k)*}$,
of the series defining $\Phi_{\mathfrak{z}}$. Their norms are achieved
at $a=I_{H_{\sigma}}$. The hypothesis guarantees that these are all
bounded by $1$. Therefore the sum is a completely positive map of
norm at most 1.\end{proof}
\begin{thm}
\label{thm:The-potential-of-Phiz}Let $X=\{X_{k}\}_{k\geq1}$be an
admissible sequence. Then there is a unique weight sequence $Z=\{Z_{k}\}_{k\geq0}$
with these four properties:
\begin{enumerate}
\item $Z_{k}\in\varphi_{k}(M)^{c}$ , $k\geq0$,
\item $Z^{(m)}\geq0$ for all $m\geq1$, 
\item If $R_{m}$ is defined to be $(Z^{(m)})^{-1}$, then for each normal
representation $\sigma$ of $M$ on a Hilbert space $H$ and for each
$\mathfrak{z}\in D(X,\sigma)$, the series $\sum_{k=0}^{\infty}\mathfrak{z}^{(k)}(R_{k}^{2}\otimes I_{H})\mathfrak{z}^{(k)*}$
converges in norm, and
\item If $\Theta_{\mathfrak{z}}^{R}$ is the completely positive map on
$\sigma(M)'$ defined by $R$ and $\mathfrak{z}$ as in Theorem \ref{thm:Muller's_Thm2.2},
then $\Theta_{\mathfrak{z}}^{R}$ is the potential of $\Phi_{\mathfrak{z}}$. 
\end{enumerate}
\end{thm}
\begin{proof}
We produce the $R_{m}$'s first, and define the $Z_{k}$'s in terms
of them. Their uniqueness will then be established. 

Let $\mathfrak{z}$ be an element of the disc, $D(X,\sigma)$, and
form the completely positive map on $\sigma(M)'$ defined by equation
\ref{eq:Phi-sub-z}. We want to compute $(\iota-\Phi_{\mathfrak{z}})^{-1}$.
Note that the norm of $\Phi_{\mathfrak{z}}$ is $\Vert\Phi_{\mathfrak{z}}(I)\Vert<1$,
by hypothesis. Therefore, $(\iota-\Phi_{\mathfrak{z}})^{-1}=\sum_{k\geq0}\Phi_{\mathfrak{z}}^{k}$,
which converges in norm. To compute $\Phi_{\mathfrak{z}}^{2}$, let
$a\in\sigma(M)'$. Then 
\begin{multline*}
\Phi_{\mathfrak{z}}^{2}(a)=\sum_{k}\Phi_{\mathfrak{z}}(\mathfrak{z}^{(k)}(X_{k}\otimes a)\mathfrak{z}^{(k)*})\\
=\sum_{k,p}\mathfrak{z}^{(p)}(X_{p}\otimes\mathfrak{z}^{(k)}(X_{k}\otimes a)\mathfrak{z}^{(k)*})\mathfrak{z}^{(p)*}\\
=\sum_{k,p\geq1}\mathfrak{z}^{(k+p)}(X_{p}\otimes X_{k}\otimes a)\mathfrak{z}^{(k+p)*}.
\end{multline*}
Relabel the last sum to get
\[
\Phi_{\mathfrak{z}}^{2}(a)=\sum_{j=1}^{\infty}\mathfrak{z}^{(j)}(\sum_{p+k=j}X_{p}\otimes X_{k}\otimes a)\mathfrak{z}^{(j)*},
\]
and then iterate the process to find that all $l\geq1$, 
\[
\Phi_{\mathfrak{z}}^{l}(a)=\sum_{k\geq1}\mathfrak{z}^{(k)}\{(\sum_{i_{1}+\cdots+i_{l}=k}X_{i_{1}}\otimes\cdots\otimes X_{i_{l}})\otimes a\}\mathfrak{z}^{(k)*}.
\]
Consequently, 
\begin{multline*}
(\iota-\Phi_{\mathfrak{z}})^{-1}(a)=a+\sum_{l\geq0}\sum_{k\geq1}\mathfrak{z}^{(k)}\{(\sum_{i_{1}+\cdots+i_{l}=k}X_{i_{1}}\otimes\cdots\otimes X_{i_{l}})\otimes a\}\mathfrak{z}^{(k)*}\\
=a+\sum_{k\geq1}\mathfrak{z}^{(k)}\{(\sum_{l\geq0}\sum_{i_{1}+\cdots+i_{l}=k}X_{i_{1}}\otimes\cdots\otimes X_{i_{l}})\otimes a\}\mathfrak{z}^{(k)*},
\end{multline*}
where empty sums are treated as zero. Introducing notation suggested
to us by Jennifer Good, which clarifies the last sum, we write
\begin{equation}
F(k,l):=\{\alpha:\{1,2,\cdots,l\}\to\mathbb{N}\mid\sum_{i=1}^{l}\alpha(i)=k\},\label{eq:Phi_k_l}
\end{equation}
for positive integers $k$ and $l$, with $l\leq k$. Then one finds
that 
\[
(\iota-\Phi_{\mathfrak{z}})^{-1}(a)=a+\sum_{k=1}^{\infty}\mathfrak{z}^{(k)}\{(\sum_{l=1}^{k}\sum_{\alpha\in F(k,l)}\bigotimes_{i=1}^{l}X_{\alpha(i)})\otimes a\}\mathfrak{z}^{(k)*}.
\]
Set 
\begin{equation}
R_{k}:=\left(\sum_{l=1}^{k}\sum_{\alpha\in F(k,l)}\bigotimes_{i=1}^{l}X_{\alpha(i)}\right)^{\frac{1}{2}},\qquad k\geq1,\label{eq:Def_Rk}
\end{equation}
and note that each $R_{k}$ is invertible by virtue of the assumptions
that $X_{1}$ is invertible and that all the $X_{i}$ are nonnegative.
Also, set 
\begin{equation}
Z_{k}:=R_{k}^{-1}(I_{E}\otimes R_{k-1}),\qquad k\geq1,\label{eq:Def of Zk}
\end{equation}
and set $Z_{0}=I$. It is clear that the route from $\{R_{k}\}_{k\geq1}$
to $Z=\{Z_{k}\}_{k\geq0}$ may be traversed in reverse and therefore
$Z$ is uniquely determined and satisfies $Z^{(m)}\geq0$ for all
$m\geq1$. The only thing that really remains to prove is that $Z$
is a bounded sequence.

For this purpose, first note the following two equations, whose verifications
are immediate from the definitions:
\begin{equation}
Z^{(k)}R_{k}^{2}(Z^{(k)})^{*}=I_{E^{\otimes k}},\qquad k\geq1,\label{eq:Z-R_relation}
\end{equation}
and 
\begin{equation}
\sum_{l=1}^{k}X_{l}\otimes R_{k-l}^{2}=R_{k}^{2},\qquad k\geq1.\label{eq:X-R_relation}
\end{equation}
These equations combine to yield
\begin{multline}
\sum_{l=1}^{k}Z_{k}(I_{E}\otimes Z_{k-1})\cdots(I_{E^{\otimes l-1}}\otimes Z_{k+1-l})(X_{l}\otimes I_{E^{\otimes k-l}})\\
(I_{E^{\otimes l-1}}\otimes Z_{k+1-l})^{*}\cdots(I_{E}\otimes Z_{k-1})^{*}Z_{k}^{*}=I_{E^{\otimes k}}.\label{eq:X-Z_relation}
\end{multline}
Indeed, 
\begin{multline*}
\sum_{l=1}^{k}Z_{k}(I_{E}\otimes Z_{k-1})\cdots(I_{E^{\otimes l-1}}\otimes Z_{k+1-l})(X_{l}\otimes I_{E^{\otimes k-l}})\\
\times(I_{E^{\otimes l-1}}\otimes Z_{k+1-l})^{*}\cdots(I_{E}\otimes Z_{k-1})^{*}Z_{k}^{*}\\
\shoveleft=\sum_{l=1}^{k}Z_{k}(I_{E}\otimes Z_{k-1})\cdots(I_{E^{\otimes l-1}}\otimes Z_{k+1-l})\\
\times(X_{l}\otimes Z^{(k-l)}R_{k-l}^{2}(Z^{(k-l)})^{*})(I_{E^{\otimes l-1}}\otimes Z_{k+1-l})^{*}\cdots(I_{E}\otimes Z_{k-1})^{*}Z_{k}^{*}\\
=\sum Z^{(k)}(X_{l}\otimes R_{k-l}^{2})(Z^{(k)})^{*}=Z^{(k)}R_{k}^{2}(Z^{(k)})^{*}=I_{E^{\otimes k}}.
\end{multline*}
But equation \eqref{eq:X-Z_relation} contains the term $Z_{k}(X_{1}\otimes I_{E^{\otimes k-1}})Z_{k}^{*}$
and all the summands are nonnegative. Therefore, $Z_{k}(X_{1}\otimes I_{E^{\otimes k-1}})Z_{k}^{*}\leq I_{E^{\otimes k}}$,
for all $k\geq1$. Since $X_{1}$ is invertible, there is an $a>0$
such that $X_{1}\geq aI_{E}$ and it then follows that $aZ_{k}Z_{k}^{*}\leq I_{E^{\otimes k}}$
and, thus, $||Z_{k}||\leq1/\sqrt{a}$ for every $k\geq1$, i.e., $Z$
is a bounded sequence.
\end{proof}
Theorem \eqref{thm:The-potential-of-Phiz} raises two questions which
need further investigation that we shall not pursue here. First, ``Are
all weighted shifts covered by this theorem?'' That is, can every
weighted shift with positive weights be realized through a potential
of a suitable admissible sequence? The answer is ``No''. In her thesis
\cite{Good2015}, among many things, Jennifer Good has shown that
the Bergman shift does not arise from an admissible sequence in the
fashion described in Theorem \ref{thm:The-potential-of-Phiz}. The
second question is: ``Are the prepotentials $\Phi_{\mathfrak{z}}$
built from admissible sequences the only prepotentials that give rise
to weighted shifts?'' The answer, again, is ``No''. In his thesis,
Itzik Martziano studies the following generalization of our setting.
Let $X=\{X_{k}\}_{k\geq1}$ be an admissible sequence, let $m$ be
a positive integer and let $\overline{D(X,\sigma,m)}:=\{\mathfrak{z}\mid(id-\Phi_{\mathfrak{z}})^{k}(I)\geq0,\;1\leq k\leq m\}$,
where $\Phi_{\mathfrak{z}}$ is the series of completely positive
maps on $\sigma(M)'$ determined by $\mathfrak{z}$ and $X$ through
\ref{eq:Phi-sub-z} and where $id$ denotes the identity map on $\sigma(M)'$.
He shows that in such a situation, it is possible to build a weight
sequence $Z$ such that every point in $\overline{D(X,\sigma,m)}$
gives a completely contractive representation of $\mathcal{T}_{+}(E,Z)$.
Further, with additional technical hypotheses, every completely contractive
representation of $\mathcal{T}_{+}(E,Z)$ comes from a point in $\overline{D(X,\sigma,m)}$.
\begin{defn}
\label{Def:Associated} Let $X$ be an admissible sequence. A sequence
$G=\{G_{k}\}_{k\geq0}$ of invertible operators, with $G_{k}\in\varphi_{k}(M)^{c}$,
for all $k$, is said to be \emph{associated} to $X$ in case $G^{(m)*}G^{(m)}=R_{m}^{-2}$,
where $R_{m}$ is defined in terms of $X$ in \eqref{eq:Def_Rk}.
The sequence $Z=\{Z_{k}\}_{k\geq0}$ constructed in Theorem \ref{thm:The-potential-of-Phiz}
from $X$ is called the \emph{canonical} weight sequence associated
with $X$.
\end{defn}
The canonical weight sequence need not consist of positive operators,
as is evident from the definition. Nevertheless, as we shall see,
$Z$ is equivalent, in the sense of the following definition to a
sequence with positive terms. 
\begin{defn}
\label{equivalent_weights} Suppose $A=\{A_{k}\}$ and $B=\{B_{k}\}$
are both sequences of weights associated with admissible sequences.
Then $A$ and $B$ are called \emph{equivalent} if, for every $k\geq0$,
there is a unitary operator $U_{k}\in\varphi_{k}(M)^{c}$ such that
\[
B_{k+1}=U_{k+1}A_{k+1}(I_{E}\otimes U_{k}^{*}),\qquad k\geq1.
\]
 \end{defn}
\begin{lem}
\label{equivalent} If $A=\{A_{k}\}$ and $B=\{B_{k}\}$ are both
sequences of invertible weights, with $A_{k},B_{k}\in\varphi_{k}(M)^{c}$,
then $A$ and $B$ are equivalent if and only if they are associated
with the same admissible sequence. \end{lem}
\begin{proof}
Suppose the two sequences are equivalent and $\{U_{k}\}$ are as in
Definition~\ref{equivalent_weights} then $B^{(k)}=U_{k}A^{(k)}$
for all $k\geq1$ and, thus, $A^{(k)*}A^{(k)}=B^{(k)*}B^{(k)}$. Since
the $X_{k}$'s can be calculated in terms of the $R_{k}$'s and since
the $R_{k}'s$ can be calculated in terms of the weight sequence by
Definition~\ref{Def:Associated}, we see that the equation $A^{(k)*}A^{(k)}=B^{(k)*}B^{(k)}$
implies that the two sequences $A$ and $B$ come from the same admissible
sequence. For the converse, assume that the two sequences are associated
with the same admissible sequence. Then $A^{(k)*}A^{(k)}=B^{(k)*}B^{(k)}$
for all $k$. We construct the sequence $\{U_{k}\}$ in Definition~\ref{equivalent_weights}
inductively. We set $U_{0}=I$ and $U_{1}=B_{1}A_{1}^{-1}$. Then
$U_{1}$ is a unitary operator since it is invertible and $A_{1}^{*}A_{1}=B_{1}^{*}B_{1}$.
Suppose $U_{1},U_{2},\ldots,U_{k}$ have been chosen such that $B_{l}=U_{l}A_{l}(I_{E}\otimes U_{l-1}^{*})$
for $l\leq k$. It follows that, for every such $l$, $B^{(l)}=U_{l}A^{(l)}$.
Set $U_{k+1}=B^{(k+1)}(A^{(k+1)})^{-1}$. Then $U_{k+1}$ is invertible
and 
\begin{multline*}
U_{k}^{*}U_{k}=((A^{(k+1)})^{-1})^{*}(B^{(k+1)})^{*}B^{(k+1)}(A^{(k+1)})^{-1}\\
=((A^{(k+1)})^{-1})^{*}(A^{(k+1)})^{*}A^{(k+1)}(A^{(k+1)})^{-1}=I.
\end{multline*}
Thus $U_{k+1}$ is unitary and $B^{(k+1)}=U_{k+1}A^{(k+1)}$. Therefore
\begin{multline*}
B_{k+1}=B^{(k+1)}(I_{E}\otimes B^{(k)})^{-1}=U_{k+1}A^{(k+1)}(I_{E}\otimes A^{(k)})^{-1}(I_{E}\otimes U_{k})^{*}\\
=U_{k+1}A_{k+1}(I_{E}\otimes U_{k})^{*}.
\end{multline*}
 \end{proof}
\begin{lem}
\label{weights} If $X=\{X_{k}\}_{k\geq1}$ is an admissible sequence,
then there is a unique sequence $Z=\{Z_{k}\}_{k\geq0}$ of weights
associated with $X$ such that $Z_{k}\geq0$ for every $k\geq1$.\end{lem}
\begin{proof}
The proof is by induction. For $k=0$, $Z_{0}=I$, by convention.
For $k=1$, $Z_{1}$ is defined by the equation $Z_{1}^{2}=R_{1}^{-2}$
so that $Z_{1}$ is uniquely determined and clearly $Z_{1}$ invertible.
Suppose that $Z_{1},Z_{2},\ldots Z_{m}$ are positive and invertible.
Then $Z^{(1)},\ldots,Z^{(m)}$ are invertible, and as we want $R_{m+1}^{-2}=Z^{(m+1)*}Z^{(m+1)}=(I_{E}\otimes Z^{(m)})^{*}Z_{m}^{2}(I_{E}\otimes Z^{(m)})$,
we simply define $Z_{m+1}^{2}:=((I_{E}\otimes Z^{(m)})^{*})^{-1}R_{m+1}^{-2}(I_{E}\otimes Z^{(m)})^{-1}$.
Then $Z_{m+1}$ is positive and invertible, and its uniqueness is
manifest.
\end{proof}
The proof of the following lemma is straightforward and is omitted.
\begin{lem}
\label{uequivalent} Suppose $A=\{A_{k}\}$ and $B=\{B_{k}\}$ are
two equivalent sequences of weights associated with the same admissible
sequence $X$ and let $\{U_{k}\}_{k=0}^{\infty}$ be a sequence of
unitary operators implementing the equivalence with $U_{0}=I$. Write
$U$ for the diagonal operator on $\mathcal{F}(E)$ defined by $\{U_{k}\}_{k=0}^{\infty}$
, i.e., let $U\theta=U_{k}\theta$, for $\theta\in E^{\otimes k}$.
Then $U$ is a unitary operator in $\mathcal{L}(\mathcal{F}(E))$
that commutes with $\varphi_{\infty}(M)$ and satisfies the equation
$W_{\xi}^{B}U=UW_{\xi}^{A}$ for all $\xi\in E$. Consequently, $U\mathcal{T}_{+}(E,A)U^{*}=\mathcal{T}_{+}(E,B)$,
$U\mathcal{T}(E,A)U^{*}=\mathcal{T}(E,B)$, and $UH^{\infty}(E,A)U^{*}=H^{\infty}(E,B).$
\end{lem}
Thus, the algebras we are studying are all artifacts of the admissible
weight sequence $X$ and, up to unitary equivalence, the representation
theory of each algebra depends only upon $X$.

\section{Representations and the Poisson Kernel\label{sec:Representations-and-Poisson}}

Throughout this section $X=\{X_{k}\}_{k\geq1}$ will be an admissible
weight sequence and $Z=\{Z_{k}\}_{k\geq0}$ will be a sequence of
weights associated to $X$ . Our goal is to prove
\begin{thm}
\label{Dparametrizesrepresentations} For each normal {*}-representation
$\sigma$ of $M$, the completely contractive representations of the
tensor algebra $\mathcal{T}_{+}(E,Z)$ whose restrictions to $\varphi_{\infty}(M)$
is $\sigma$ are parameterized by the points in $\overline{D(X,\sigma)}$.
That is, a representation $\sigma\times\mathfrak{z}$ of the algebraic
tensor algebra $\mathcal{T}_{0+}(E,Z)$ extends to a completely contractive
representation of the full tensor algebra $\mathcal{T}_{+}(E,Z)$
if and only if $\mathfrak{z}$ belongs to $\overline{D(X,\sigma)}$.
Further, if $\mathfrak{z}\in D(X,\sigma)$, then $\sigma\times\mathfrak{z}$
extends to an ultraweakly continuous completely contractive representation
of $H^{\infty}(E,Z)$.
\end{thm}
The key ingredient to be introduced and used in this effort is the
Poisson kernel built from $Z$. 

We write $W=W^{Z}$ for the weighted intertwiner in $\mathfrak{I}(\varphi_{\infty}^{\mathcal{F}(E)}\circ\varphi,\varphi)$
determined by $Z$ and we write $P_{k}$ for the projection in $\mathcal{L}(\mathcal{F}(E))$
onto the $k^{th}$ summand of $\mathcal{F}(E)$, $E^{\otimes k}$,
$k\geq0$. Note that $W$ may be viewed matricially as 
\[
W=\begin{bmatrix}0 & 0 & 0 & \cdots\\
Z_{1} & 0 & 0 & \cdots\\
0 & Z_{2} & 0 & \cdots\\
\vdots & \ddots & \ddots & \ddots
\end{bmatrix}
\]
acting from $\sum_{k\geq1}^{\oplus}E^{\otimes k}$ to $\sum_{k\geq0}^{\oplus}E^{\otimes k}$.
Also $X$ may be promoted to a sequence of operators $\widetilde{X}=\{\widetilde{X}_{k}\}_{k\geq0}$
on $\mathcal{F}(E)$, where $\widetilde{X}_{0}=0$ and for $k\geq1$,
$\widetilde{X}_{k}$ is the diagonal operator whose $l^{th}$ diagonal
entry $\widetilde{X}_{k}(l)$ is given by
\begin{equation}
\widetilde{X}_{k}(l):=\begin{cases}
\begin{array}{c}
0,\\
X_{k}\otimes I_{E^{\otimes(l-k)}},
\end{array} & \begin{array}{c}
l\leq k\\
l>k
\end{array}\end{cases}.\label{eq:Xk-tilde}
\end{equation}

The proof of the following lemma depends heavily on equation \ref{eq:X-Z_relation}.
In that equation, the sequence $Z$ is the canonical sequence associated
with $X$. However, the only fact about $Z$ that was used is that
it is associated with $X$, i.e., that $Z^{(m)*}Z^{(m)}=R_{m}^{-2}$,
where $R_{m}$ is defined in terms of $X$ in \eqref{eq:Def_Rk}.
So in particular, that equation is valid for the choice of $Z$ made
at the outset of this section. 
\begin{lem}
\label{W} With the notation established, 
\begin{equation}
\sum_{k=1}^{\infty}W^{(k)}\widetilde{X}_{k}W^{(k)*}=I_{\mathcal{F}(E)}-P_{0}.\label{eq:sumW}
\end{equation}
\end{lem}
\begin{proof}
Write $L$ for the left-hand side of (\ref{eq:sumW}). Then it is
easy to see that $L$ is a diagonal matrix and so we need only attend
to its diagonal entries. For $l=0$, $P_{0}LP_{0}=0$, by definition.
For $l=1$, 
\begin{multline*}
P_{1}LP_{1}=P_{1}W\widetilde{X}_{1}W{}^{*}P_{1}\\
=P_{1}WP_{1}\widetilde{X}_{1}P_{1}W{}^{*}P_{1}=P_{1}(Z_{1}X_{1}Z_{1}^{*})P_{1}\\
=P_{1},
\end{multline*}
by equation \eqref{eq:X-Z_relation}. To compute $P_{l}LP_{l}$ for
$l>1$, note first that for $k\leq l$, 
\begin{multline*}
P_{l}W^{(k)}P_{l}=P_{l}W(I_{E}\otimes W)\cdots(I_{E^{\otimes(k-1)}}\otimes W)P_{l}\\
=P_{l}WP_{l}(I_{E}\otimes W)P_{l}\cdots P_{l}(I_{E^{\otimes(k-1)}}\otimes W)P_{l}\\
=P_{l}Z_{l}(I_{E}\otimes P_{l-1}Z_{l-1}P_{l-1})\cdots(I_{E^{\otimes(k-1)}}\otimes P_{l-k+1}Z_{l-k+1})P_{l}\\
=P_{l}(Z_{l}(I_{E}\otimes Z_{l-1})\cdots(I_{E^{\otimes(k-1)}}\otimes Z_{l-k+1})P_{l}.
\end{multline*}
It follows easily from equation \eqref{eq:X-Z_relation} that $P_{l}LP_{l}=P_{l}$
for every $l>0$.
\end{proof}
Suppose, now, $\sigma$ is a normal representation of $M$ on a Hilbert
space $H_{\sigma}$, and let $\mathfrak{z}\in\mathfrak{I}(\sigma^{E}\circ\varphi,\sigma)$.
Then, as we noted in Remark \ref{Rm:_Algebraic_Extension}, the pair
$(\sigma,\mathfrak{z})$ determines a purely algebraic homomorphism
or representation of $\mathcal{T}_{0+}(E,Z)$ in $B(H_{\sigma})$,
denoted $\sigma\times\mathfrak{z}$, by the formula 
\[
\sigma\times\mathfrak{z}(\varphi_{\infty}(a)+\sum_{k\geq1}W_{\xi_{k}})=\sigma(a)+\sum_{k\geq1}\mathfrak{z}^{(k)}L_{\xi_{k}},
\]
where $\xi_{k}\in E^{\otimes k}$ and, recall, $L_{\xi_{k}}$ is the
map from $H_{\sigma}$ to $E^{\otimes k}\otimes_{\sigma}H_{\sigma}$
defined by the formula $L_{\xi_{k}}h:=\xi_{k}\otimes h$. We want
to determine when $\sigma\times\mathfrak{z}$ extends to a completely
contractive representation of $\mathcal{T}_{+}(E,Z)$. We first show
that each point in the \emph{open} disc $D(X,\sigma)$ defines a completely
contractive representation of $\mathcal{T}_{+}(E,Z)$. In fact, we
show that such representations automatically extend to ultraweakly
continuous, completely contractive representations of $H^{\infty}(X,\sigma)$.
We then focus on points on the boundary of $D(X,\sigma)$. Finally,
we show that each completely contractive representation of $\mathcal{T}_{+}(E,Z$)
is determined by a point in $\overline{D(X,\sigma)}$. 
\begin{defn}
\label{Poisson_Kernel} Let $X$ be an admissible sequence and let
$R_{k}$ be defined as in equation \eqref{eq:Def_Rk}. Given a normal
representation $\sigma$ on $H$ and an element $\mathfrak{z}\in\overline{D(X,\sigma)}$,
we set $\Delta_{*}(\mathfrak{z})=(I-\sum_{k=1}^{\infty}\mathfrak{z}^{(k)}(X_{k}\otimes I_{H_{\sigma}})\mathfrak{z}^{(k)*})^{1/2}$.
The \emph{Poisson kernel}, denoted $K(\mathfrak{z})$, is the operator
$K(\mathfrak{z}):H\rightarrow\mathcal{F}(E)\otimes_{\sigma}H$ defined
by 
\[
K(\mathfrak{z})=(I_{\mathcal{F}(E)}\otimes\Delta_{*}(\mathfrak{z}))[I_{H}\quad(R_{1}\otimes I_{H})\mathfrak{z}^{*}\quad(R_{2}\otimes I_{H})(\mathfrak{z}^{(2)})^{*}\quad\ldots]^{T}
\]
That is, 
\[
K(\mathfrak{z})h=(I_{\mathcal{F}(E)}\otimes\Delta_{*}(\mathfrak{z}))(h\oplus(R_{1}\otimes I_{H})\mathfrak{z}^{*}h\oplus(R_{2}\otimes I_{H})(\mathfrak{z}^{(2)})^{*}h\oplus\cdots)
\]
for $h\in H$. 
\end{defn}
It should be emphasized that \emph{a priori} the sum representing $K(\mathfrak{z})h$ converges only for $\mathfrak{z}\in D(X,\sigma)$. The following lemma shows that the series representing $K(\mathfrak{z})^*K(\mathfrak{z})$ converges in the strong operator topology and identifies the sum.  From this one concludes that the sum representing $K(\mathfrak{z})h$ converges strongly for all $\mathfrak{z}\in \overline{ D(X,\sigma)}$.
\begin{lem}
\label{PoissonIsometry} Assume $\mathfrak{z}\in\overline{D(X,\sigma)}$
and let $\Phi_{\mathfrak{z}}$ be its prepotential as defined in Lemma
\ref{lem:Basic_CP_maps}. Then 
\begin{enumerate}
\item The sequence $\{\Phi_{\mathfrak{z}}^{m}(I)\}_{m=0}^{\infty}$ is a
decreasing sequence of positive operators on $H_{\sigma}$, with strong
limit denoted $Q_{\mathfrak{z}}$. 
\item $\Phi_{\mathfrak{z}}(Q_{\mathfrak{z}})=Q_{\mathfrak{z}}$. 
\item $K(\mathfrak{z})^{*}K(\mathfrak{z})=I_{H}-Q_{\mathfrak{z}}$. 
\item If $\mathfrak{z}\in D(X,\sigma)$, i.e., if $||\sum_{k=1}^{\infty}\mathfrak{z}^{(k)}(X_{k}\otimes I_{H_{\sigma}})\mathfrak{z}^{(k)*}||<1$,
then $Q_{\mathfrak{z}}=0$ and, consequently, $K(\mathfrak{z})^{*}K(\mathfrak{z})=I_{H}$
so that $K(\mathfrak{z})$ is an isometry. 
\end{enumerate}
\end{lem}
\begin{proof}
For the first assertion, note that $\Phi_{\mathfrak{z}}(I)\leq I$
and $\Phi_{\mathfrak{z}}$ is completely positive. Thus the sequence
is decreasing. The second assertion is immediate. For the third, we
compute.  But first note that all the terms in the sums involved are non-negative and so sums may be interchanged.  
\[
K(\mathfrak{z})^{*}K(\mathfrak{z})=\Delta_{*}(\mathfrak{z})^{2}+\sum_{k=1}^{\infty}\mathfrak{z}^{(k)}(R_{k}^{2}\otimes\Delta_{*}(\mathfrak{z})^{2})\mathfrak{z}^{(k)*}=
\]
\begin{multline*}
=\Delta_{*}(\mathfrak{z})^{2}+\sum_{k=1}^{\infty}\mathfrak{z}^{(k)}\left(\left(\sum_{l=1}^{k}\sum_{\alpha\in F(k,l)}\bigotimes_{i=1}^{l}X_{\alpha(i)}\right)\otimes\Delta_{*}(\mathfrak{z})^{2}\right)\mathfrak{z}^{(k)*}\\
=\Delta_{*}(\mathfrak{z})^{2}+\sum_{k=1}^{\infty}\sum_{l\leq k}\mathfrak{z}^{(k)}(\sum_{i_{1}+i_{2}+\cdots+i_{l}=k}X_{i_{1}}\otimes X_{i_{2}}\otimes\cdots\otimes X_{i_{l}}\otimes\Delta_{*}(\mathfrak{z})^{2})\mathfrak{z}^{(k)*}.
\end{multline*}
Now we compute again: 
\begin{multline*}
\Phi_{\mathfrak{z}}(\mathfrak{z}^{(k)}(X_{k}\otimes\Delta_{*}(\mathfrak{z})^{2})\mathfrak{z}^{(k)*})\\=\sum_{m=1}^{\infty}\mathfrak{z}^{(m)}(X_{j}\otimes\mathfrak{z}^{(k)}(X_{k}\otimes\Delta_{*}(\mathfrak{z})^{2})\mathfrak{z}^{(k)*})\mathfrak{z}^{(m)*}\\=\sum_{m=k+1}^{\infty}\mathfrak{z}^{(m)}(X_{m-k}\otimes X_{k}\otimes\Delta_{*}(\mathfrak{z})^{2})\mathfrak{z}^{(m)*},
\end{multline*}
and conclude 
\[
\Phi_{\mathfrak{z}}^{2}(\Delta_{*}(\mathfrak{z})^{2})=\sum_{m=2}^{\infty}\mathfrak{z}^{(m)}(\sum_{i_{1}+i_{2}=m}X_{i_{1}}\otimes X_{i_{2}}\otimes\Delta_{*}(\mathfrak{z})^{2})\mathfrak{z}^{(m)*}.
\]
Similarly, 
\[
\Phi_{\mathfrak{z}}^{l}(\Delta_{*}(\mathfrak{z})^{2})=\sum_{m=l}^{\infty}\mathfrak{z}^{(m)}(\sum_{i_{1}+i_{2}+\cdots+i_{l}=m}X_{i_{1}}\otimes X_{i_{2}}\otimes\cdots\otimes X_{i_{l}}\otimes\Delta_{*}(\mathfrak{z})^{2})\mathfrak{z}^{(m)*}.
\]
Combining these computations, we find that 
\begin{multline*}
K(\mathfrak{z})^{*}K(\mathfrak{z})=\Delta_{*}(\mathfrak{z})^{2}+\sum_{l=1}^{\infty}\Phi_{\mathfrak{z}}^{l}(\Delta_{*}(\mathfrak{z})^{2})=I-\Phi_{\mathfrak{z}}(I)+\sum_{l=1}^{\infty}\Phi_{\mathfrak{z}}^{l}(I-\Phi_{\mathfrak{z}}(I))\\
=I-\lim\Phi_{\mathfrak{z}}^{l}(I)=I-Q_{\mathfrak{z}}.
\end{multline*}
The fourth, and final, assertion is immediate.\end{proof}
\begin{lem}
\label{Kintertwines} Let $\sigma$ be a normal representation of
$M$, let $X$ be an admissible sequence and let $Z=\{Z_{k}\}$ be
an associated sequence of weights. 
\begin{enumerate}
\item If $Z$ is the canonical sequence of weights associated with $X$,
then, for every $\xi\in E$ and every $\mathfrak{z}\in\overline{D(X,\sigma)}$,
\begin{equation}
K(\mathfrak{z})L_{\xi}^{*}\mathfrak{z}^{*}=(W_{\xi}^{*}\otimes I_{H})K(\mathfrak{z})\label{eq:intertwine}
\end{equation}
 where $L_{\xi}h=\xi\otimes h$ for $\xi\in E$ and $h\in H$. 
\item For a general sequence of weights $Z$, there exists a unitary operator,
$U\in\mathcal{L}(\mathcal{F}(E))\cap\varphi_{\infty}(M)'$, that maps
each $E^{\otimes k}$ onto itself, such that 
\begin{equation}
(U\otimes I_{H})K(\mathfrak{z})L_{\xi}^{*}\mathfrak{z}^{*}=(W_{\xi}^{*}\otimes I_{H})(U\otimes I_{H})K(\mathfrak{z}).\label{eq:intertwine_a}
\end{equation}

\item For $a\in M$ we have 
\begin{equation}
\sigma^{\mathcal{F}(E)}\circ\varphi_{\infty}(a)K(\mathfrak{z})=(\varphi_{\infty}(a)\otimes I_{H})K(\mathfrak{z})=K(\mathfrak{z})\sigma(a).\label{eq:intertwine_b}
\end{equation}

\end{enumerate}
\end{lem}
\begin{proof}
To prove \eqref{eq:intertwine} we need to show that for every $k$
we have 
\begin{multline*}
(I_{E^{\otimes k}}\otimes\Delta_{*}(\mathfrak{z}))(R_{k}\otimes I_{H})(\mathfrak{z}^{(k)})^{*}L_{\xi}^{*}\mathfrak{z}^{*}\\=(W_{\xi}^{*}\otimes I_{H})(I_{E^{\otimes(k+1)}}\otimes\Delta_{*}(\mathfrak{z}))(R_{k+1}\otimes I_{H})(\mathfrak{z}^{(k+1)})^{*}.
\end{multline*}
To do this we shall show that, for every $k$, 
\begin{multline*}
L_{\xi}\mathfrak{z}^{(k)}(R_{k}\otimes I_{H})(I_{E^{\otimes k}}\otimes\Delta_{*}(\mathfrak{z}))\\=(I_{E}\otimes\mathfrak{z}^{(k)})(R_{k+1}\otimes I_{H})(I_{E^{\otimes(k+1)}}\otimes\Delta_{*}(\mathfrak{z}))(W_{\xi}\otimes I_{H})
\end{multline*}
as operators on $E^{\otimes k}\otimes H$. So, fix $\theta\in E$
and $h\in H$ and apply the right-hand-side of this equation to $\theta\otimes h$
to get 
\begin{multline*}
(I_{E}\otimes\mathfrak{z}^{(k)})(R_{k+1}\otimes I_{H})(I_{E^{\otimes(k+1)}}\otimes\Delta_{*}(\mathfrak{z}))(W_{\xi}\otimes I_{H})(\theta\otimes h)\\
=(I_{E}\otimes\mathfrak{z}^{(k)})(R_{k+1}\otimes I_{H})(I_{E^{\otimes(k+1)}}\otimes\Delta_{*}(\mathfrak{z}))(Z_{k+1}(\xi\otimes\theta)\otimes h)\\
=(I_{E}\otimes\mathfrak{z}^{(k)})(R_{k+1}Z_{k+1}(\xi\otimes\theta)\otimes\Delta_{*}(\mathfrak{z})h)\\
=(I_{E}\otimes\mathfrak{z}^{(k)})(\xi\otimes R_{k}\theta\otimes\Delta_{*}(\mathfrak{z})h),
\end{multline*}
where, in the last equality, we used the fact that $Z_{k+1}=R_{k+1}^{-1}(I_{E}\otimes R_{k})$. Since $(I_{E}\otimes\mathfrak{z}^{(k)})(\xi\otimes R_{k}\theta\otimes\Delta_{*}(\mathfrak{z})h)$ is equal to $L_{\xi}\mathfrak{z}^{(k)}(R_{k}\otimes I_{H})(I_{E^{\otimes k}}\otimes\Delta_{*}(\mathfrak{z}))(\theta\otimes h)$,
this proves \ref{eq:intertwine}.

For part (2), use part (1) and Lemma~\ref{uequivalent}. Part (3)
follows from the definition of $K(\mathfrak{z})$ since $R_{k}\in\varphi_{k}(M)'$
for all $k$ and $\mathfrak{z}^{(k)}$ intertwines $\sigma$ and $\varphi_{k}\otimes I_{H}$.\end{proof}
\begin{defn}
\label{ZPoisson} If $Z=\{Z_{k}\}$ is a sequence of weights associated
with an admissible sequence $X$, then the $Z$-Poisson kernel, written
$K_{Z}$ , is defined by the equation 
\[
K_{Z}(\mathfrak{z}):=(U\otimes I_{H})K(\mathfrak{z})
\]
where $U$ is the unitary operator referred to in Lemma~\ref{Kintertwines}
(2) and $\mathfrak{z}\in\overline{D(X,\sigma)}$. 
\end{defn}
Note that if $Z$ is the canonical sequence associated with $X$,
then $K_{Z}=K$.

{}
\begin{defn}
\label{Def: Evaluation Map} Suppose $\sigma$ is a normal representation
of $M$ on $H$ and that $\mathfrak{z}\in\overline{D(X,\sigma)}$.
Suppose also that $Z=\{Z_{k}\}_{k\geq0}$ is an invertible sequence
of weights associated to $X$. The formula
\[
T\mapsto K_{Z}(\mathfrak{z})^{*}(T\otimes I_{H})K_{Z}(\mathfrak{z}),\qquad T\in\mathcal{L}(\mathcal{F}(E)),
\]
defines a map from $\mathcal{L}(\mathcal{F}(E))$ to $B(H)$ called
the \emph{evaluation map} determined by $\mathfrak{z}$ and is denoted
$\Psi_{\mathfrak{z}}$, $\Psi_{Z,\mathfrak{z}}$, or $\Psi_{X,Z,\mathfrak{z}}$,
depending on context. \end{defn}
\begin{rem}
\label{Rm:Eval_Map}Assertion (2) of Lemma \ref{PoissonIsometry}
shows that $K(\mathfrak{z})$ is a contraction for each $\mathfrak{z}\in\overline{D(X,\sigma)}$
and therefore so is $K_{Z}(\mathfrak{z})=(U\otimes I)K(\mathfrak{z})$.
Consequently, $\Psi_{Z,\mathfrak{z}}$ is completely positive and
contractive; it is unital if and only if $Q_{\mathfrak{z}}=0$.\end{rem}
\begin{cor}
\label{representations} Suppose $\mathfrak{z}\in\overline{D(X,\sigma)}$
and $Q_{\mathfrak{z}}=0$ (as is the case if $||\sum_{k=1}^{\infty}\mathfrak{z}^{(k)}(X_{k}\otimes I_{H_{\sigma}})\mathfrak{z}^{(k)*}||<1$).
Then the evaluation map $\Psi_{Z,\mathfrak{z}}$ restricted to $H^{\infty}(E,Z)$
is an ultraweakly continuous, completely contractive representation
of $H^{\infty}(E,Z)$ on $H_{\sigma}$ that extends $\sigma\times\mathfrak{z}$
on $\mathcal{T}_{0+}(E,Z)$. Specifically,

\begin{equation}
\Psi_{Z,\mathfrak{z}}(\varphi_{\infty}(a))=\sigma(a),\qquad a\in M,\label{eq:Covariant-1}
\end{equation}
and
\begin{equation}
\Psi_{Z,\mathfrak{z}}(W_{\xi})=\mathfrak{z}^{(n)}L_{\xi},\qquad\xi\in E^{\otimes n}.\label{eq:Covariant-2}
\end{equation}
\end{cor}
\begin{proof}
The assumption that $Q_{\mathfrak{z}}=0$ implies that $K(\mathfrak{z})$
and $K_{Z}(\mathfrak{z})$ are isometries. If $\sigma_{K}$ is the
induced representation $\sigma_{K}=\sigma^{\mathcal{F}(E)}$ restricted
to $H^{\infty}(E,Z)$, then the range of $K_{Z}(\mathfrak{z})$ is
coinvariant for $\sigma_{K}(H^{\infty}(E,Z))$ by equations \ref{eq:intertwine}
and \ref{eq:intertwine_b}. This shows that the restriction of $\Psi_{Z,\mathfrak{z}}$
to $H^{\infty}(E,Z)$ is an ultraweakly continuous, completely contractive
representation of $H^{\infty}(E,Z)$. Equation \ref{eq:Covariant-1}
follows from equation \ref{eq:intertwine_b} and the fact that $U\otimes I_{H_{\sigma}}$commutes
with $\sigma^{\mathcal{F}(E)}\circ\varphi_{\infty}$. As for equation
\ref{eq:Covariant-2}, observe that equation \ref{eq:intertwine_a}
shows that 
\begin{multline*}
K_{Z}(\mathfrak{z})L_{\xi}^{*}\mathfrak{z}^{*}=(U\otimes I_{H})K(\mathfrak{z})L_{\xi}^{*}\mathfrak{z}^{*}=(W_{\xi}^{*}\otimes I_{H})(U\otimes I_{H})K(\mathfrak{z})\\
=(W_{\xi}^{*}\otimes I_{H})K_{Z}(\mathfrak{z}).
\end{multline*}
Since $K_{Z}(\mathfrak{z)}$ is an isometry, we may multiply both
sides of this equation by $K_{Z}(\mathfrak{z})^{*}$ to obtain $K_{Z}(\mathfrak{z})^{*}(W_{\xi}^{*}\otimes I_{H})K_{Z}(\mathfrak{z})=L_{\xi}^{*}\mathfrak{z}^{*}$.
If we take adjoints in this equation, then we have equation \ref{eq:Covariant-2}
for the case $\xi\in E$. The general result now follows from this
case, Remark \ref{Rm:Products}, and the fact that $\Psi_{Z,\mathfrak{z}}$
restricted to $H^{\infty}(E,Z)$ is a homomorphism.
\end{proof}
We have shown that each $\mathfrak{z}\in\overline{D(X,\sigma)}$ such
that $Q_{\mathfrak{z}}=0$ determines a unital completely positive
map on all of $\mathcal{L}(\mathcal{F}(E))$ with the property that
the restriction to $\mathcal{T}_{+}(E,Z)$ is a completely contractive
representation. We now want to show that independent of whether or
not $Q_{\mathfrak{z}}=0$, $\mathfrak{z}$ determines a completely
positive contractive map $\Psi_{Z,\mathfrak{z}}$ on the \emph{operator
system} $\overline{span}\{TS^{*}\mid T,S\in\mathcal{T}_{+}(E,Z)\}$,
which contains $\mathcal{T}_{+}(E,Z)$, with the property that its
restriction to $\mathcal{T}_{+}(E,Z)$ is a completely contractive
representation such that on $M$, it gives rise to a normal $W^{*}$-representation. 
\begin{prop}
\label{Psi} Given $\mathfrak{z}\in\overline{D(X,\sigma)}$, there
is a unital, completely positive linear map 
\[
\Psi_{Z,\mathfrak{z}}:\overline{span}\{TS^{*}\mid T,S\in\mathcal{T}_{+}(E,Z)\}\rightarrow B(H_{\sigma})
\]
that restricts to a completely contractive representation of $\mathcal{T}_{+}(E,Z)$
and is such that, for $\xi\in E^{\otimes k}$ and $\eta\in E^{\otimes m}$,
\[
\Psi_{Z,\mathfrak{z}}(W_{\xi}W_{\eta}^{*})=\mathfrak{z}^{(k)}L_{\xi}L_{\eta}^{*}\mathfrak{z}^{(m)*},
\]
\[
\Psi_{Z,\mathfrak{z}}(W_{\xi})=\mathfrak{z}^{(k)}L_{\xi}
\]
and 
\[
\Psi_{Z,\mathfrak{z}}(\varphi_{\infty}(a))=\sigma(a)
\]
for $a\in M$.\end{prop}
\begin{proof}
For every $0<r<1$, $r\mathfrak{z}$ has the property that $Q_{r\mathfrak{z}}=0$.
Therefore, the evaluation map $\Psi_{Z,r\mathfrak{z}}$ is a unital
completely positive linear map on all of $\mathcal{L}(\mathcal{F}(E))$
mapping to $B(H_{\sigma})$, by Remark \ref{Rm:Eval_Map}. In fact,
$\Psi_{Z,r\mathfrak{z}}\circ\varphi_{\infty}=\sigma$ for each $r$,
$0\leq r<1$. Further, each $\Psi_{Z,r\mathfrak{z}}$ restricts to
a completely contractive representation on $\mathcal{T}_{+}(E,Z)$
by Corollary~\ref{representations}. Moreover, by Corollary \ref{representations},
we also have, for $\xi\in E^{\otimes k}$ and $\eta\in E^{\otimes m}$,
\begin{equation}
\Psi_{Z,r\mathfrak{z}}(W_{\xi}W_{\eta}^{*})=r^{k+m}\mathfrak{z}^{(k)}L_{\xi}L_{\eta}^{*}\mathfrak{z}^{(m)*}.\label{eq:Psi_Z_rz}
\end{equation}
Hence the restriction of  $\Psi_{Z,r\mathfrak{z}}$ to $\mathfrak{S}:=\overline{span}\{TS^{*}\mid T,S\in\mathcal{T}_{+}(E,Z)\}$
has all the desired properties. As in \cite[p. 146 ff]{Arv1969a},
we shall write $\mathcal{B}(\mathfrak{S},H_{\sigma})$ for the Banach
space of all bounded linear maps from $\mathfrak{S}$ to $B(H_{\sigma})$
with the operator norm and the point-ultraweak topology. Remark 1.1.2
of \cite{Arv1969a} coupled with \cite[Lemma 1.2.4]{Arv1969a} shows
that the collection of all unital, completely positive maps in $\mathcal{B}(\mathfrak{S},H_{\sigma})$
is a compact subset of the unit ball, $\mathcal{B}_{1}(\mathfrak{S},H_{\sigma})$
in this topology. If $\Psi_{Z,\mathfrak{z}}$ denotes any limit point
of $\{\Psi_{Z,r\mathfrak{z}}\}_{0\leq r<1}$, then it is easy to see
that $\Psi_{Z,\mathfrak{z}}$ has all the desired properties. For
instance, on the one hand, $\Psi_{Z,\mathfrak{z}}(W_{\xi}W_{\eta}^{*})$
is the ultraweak limit $\{\Psi_{Z,r\mathfrak{z}}(W_{\xi}W_{\eta}^{*})\}_{0\leq r<1}$.
On the other, $\mathfrak{z}^{(k)}L_{\xi}L_{\eta}^{*}\mathfrak{z}^{(m)*}$
is the ultraweak limit of $\{r^{k+m}\mathfrak{z}^{(k)}L_{\xi}L_{\eta}^{*}\mathfrak{z}^{(m)*}\}_{0\leq r<1}$.
Therefore, $\Psi_{Z,\mathfrak{z}}(W_{\xi}W_{\eta}^{*})=\mathfrak{z}^{(k)}L_{\xi}L_{\eta}^{*}\mathfrak{z}^{(m)*}$,
by \eqref{eq:Psi_Z_rz}. The other requirements of $\Psi_{Z,\mathfrak{z}}$
are verified equally easily. 
\end{proof}
We now want to show that every completely contractive representation
$\rho$ of $\mathcal{T}_{+}(E,Z)$ with the property that $\sigma:=\rho\circ\varphi_{\infty}$
is a normal representation of $M$ is determined by an element $\mathfrak{z}\in\overline{D(X,\sigma)}$.
For this purpose, we require the following lemma, which may appear
in the literature. However, we supply a proof since we do not know
a specific reference.
\begin{lem}
\label{lem:Positive_Expansion}Let $E$ be a $W^{*}$-Hilbert module
over a $W^{*}$-algebra $M$, and let $X$ be a positive element of
$\mathcal{L}(E)$. Then there is a family of vectors in $E$, $\{\xi_{\alpha}\}_{\alpha\in A}$,
such that 
\[
X=\sum_{\alpha\in A}\xi_{\alpha}\otimes\xi_{\alpha}^{*},
\]
in the sense that if $\mathcal{F}$ is the set of all finite subsets
of $A$ directed by inclusion then the net of all finite partial sums
of the series, $\{\sum_{\alpha\in F}\xi_{\alpha}\otimes\xi_{\alpha}^{*}\}_{F\in\mathcal{F}}$,
converges to $X$ in the ultrastrong topology on $\mathcal{L}(E)$.\end{lem}
\begin{proof}
The proof of Theorem 3.12 of \cite{Paschke1973} shows that there
is a set of vectors $\{\eta_{\alpha}\}_{\alpha\in A}$ in $E$ which
is maximal in the collection of all sets of vectors with the properties
that $\langle\eta_{\alpha},\eta_{\beta}\rangle=0$, if $\alpha\neq\beta$,
while $\langle\eta_{\alpha},\eta_{\alpha}\rangle$ is a projection
for all $\alpha.$ Such a set has the property that $I_{E}=\sum_{\alpha\in A}\eta_{\alpha}\otimes\eta_{\alpha}^{*}$
in the sense that the net of finite partial sums converges to $I_{E}$
in the ultraweak topology. Indeed, each of the operators $\eta_{\alpha}\otimes\eta_{\alpha}^{*}$
is a projection and any two of them are orthogonal. Therefore the
sum converges to a projection $P\in\mathcal{L}(E)$, say, in the ultraweak
topology. If the sum were not $I_{E}$, then a nonzero vector $\zeta$
in the range of $I-P$ can be found so that $\langle\zeta,\zeta\rangle$
is a projection, by \cite[Proposition 3.11]{Paschke1973}. Adding
$\zeta$ to $\{\eta_{\alpha}\}_{\alpha\in A}$ would contradict the
maximality of $\{\eta_{\alpha}\}_{\alpha\in A}$. With such a family
$\{\eta_{\alpha}\}_{\alpha\in A}$ in hand, let $\xi_{\alpha}:=X^{\frac{1}{2}}\eta_{\alpha}$.
Then a straightforward calculation shows that 
\[
\sum_{\alpha\in A}\xi_{\alpha}\otimes\xi_{\alpha}^{*}=\sum_{\alpha\in A}(X^{\frac{1}{2}}\eta_{\alpha})\otimes(X^{\frac{1}{2}}\eta_{\alpha})^{*}=X^{\frac{1}{2}}(\sum_{\alpha\in A}\eta_{\alpha}\otimes\eta_{\alpha}^{*})X^{\frac{1}{2}}=X.
\]
The conclusion that the net converges ultrastrongly to $X$ follows
from the fact that the partial sums are nonnegative elements of $\mathcal{L}(E)$
that increase to $X$.\end{proof}
\begin{rem}
\label{Rm:Cstar-algebras} If the $W^{*}$-algebra $M$ is countably
decomposable and if $E$ is countably generated as a $W^{*}$-Hilbert
module over $M$, then $\mathcal{L}(E)$ is countably decomposable
and the set $A$ in Lemma \ref{lem:Positive_Expansion} may be chosen
to be countable.

Also, if $M$ is a unital $C^{*}$-algebra and if $E$ is a countably
generated Hilbert $C^{*}$-module over $M$, then by Kasparov's stabilization
theorem, we can find a sequence of vectors in $E$, $\{\eta_{n}\}_{n\geq1}$,
with the property that $\sum_{n\geq1}\eta_{n}\otimes\eta_{n}^{*}=I_{E}$,
where the series converges in the strict topology on $\mathcal{L}(E)$,
viewed as the multiplier algebra of $\mathcal{K}(E)$. It follows
that if $X$ is a positive element of $\mathcal{L}(E)$ and if $\xi_{n}:=X^{\frac{1}{2}}\eta_{n}$,
then $\sum_{n\geq1}\xi_{n}\otimes\xi_{n}^{*}=X$, with convergence
in the strict topology.
\end{rem}
Recall that in Remark \ref{Rm:_Algebraic_Extension}, we noted that
if $\rho$ is a completely contractive representation of $\mathcal{T}_{+}(E,Z)$
such that $\sigma:=\rho\circ\varphi_{\infty}$ is a normal $*$-representation
of $M$, then the map $\xi\to\rho(W_{\xi})$ is a completely bounded
bimodule map on $E$, and so there is an element $\mathfrak{z}\in\mathfrak{I}(\sigma^{E}\circ\varphi,\sigma)$
such that $\rho(W_{\xi})=\mathfrak{z}L_{\xi}.$ The following lemma
completes the proof of Theorem \ref{Dparametrizesrepresentations}.
\begin{lem}
\label{ccrepn} If $\mathfrak{z}\in\mathfrak{I}(\sigma^{E}\circ\varphi,\sigma)$
is such that the representation $\rho$ defined by the equation $\rho(W_{\xi})=\mathfrak{z}L_{\xi}$ is completely contractive, then $\mathfrak{z}$
lies in $\overline{D(X,\sigma)}$. \end{lem}
\begin{proof}
As we noted in Remark \ref{Rm:_Algebraic_Extension}, $\rho(W_{\xi})=\mathfrak{z}^{(n)}L_{\xi}$
for all $\xi\in E^{\otimes n}$. By lemma \ref{lem:Positive_Expansion},
we may write each $X_{k}=\sum_{\alpha\in A(k)}\xi_{\alpha}\otimes\xi_{\alpha}^{*},$
for a suitable indexing set $A(\alpha)$, where $\xi_{\alpha}\in E^{\otimes k}$.
Fix $k$ and let $F(k)$ be a finite subset of $A(k)$. Then 
\begin{multline}
\sum_{\alpha\in F(k)}\mathfrak{z}^{(k)}(\xi_{\alpha}\otimes\xi_{\alpha}^{*}\otimes I_{H})\mathfrak{z}^{(k)*}=\sum_{\alpha\in F(k)}\mathfrak{z}^{(k)}L_{\xi_{\alpha}}L_{\xi_{\alpha}}^{*}\mathfrak{z}^{(k)*}\\
=\sum_{\alpha\in F(k)}\rho(W_{\xi_{\alpha}})\rho(W_{\xi_{\alpha}})^{*}.\label{eq:z-1}
\end{multline}
Since $\rho$ is completely contractive and unital, we may apply Arveson's
dilation theorem \cite[Theorem 1.2.9]{Arv1969a} to produce a $C^{*}$-representation
$\pi$ of the weighted Toeplitz algebra $\mathcal{T}(E,Z)$ on a Hilbert
space $H_{\pi}$ and an isometry $V:H\to H_{\pi}$ such that $\rho(\cdot)=V^{*}\pi(\cdot)V$.
We may then continue equation \ref{eq:z-1} with 
\begin{multline}
\sum_{\alpha\in F(k)}\rho(W_{\xi_{\alpha}})\rho(W_{\xi_{\alpha}})^{*}=\sum_{\alpha\in F(k)}V^{*}\pi(W_{\xi_{\alpha}})VV^{*}\pi(W_{\xi_{\alpha}})^{*}V\\
\leq\sum_{\alpha\in F(k)}V^{*}\pi(W_{\xi_{\alpha}})\pi(W_{\xi_{\alpha}})^{*}V=V^{*}\sum_{\alpha\in F(k)}\pi(W_{\xi_{\alpha}})\pi(W_{\xi_{\alpha}})^{*}V\\
=V^{*}\pi\left(\sum_{\alpha\in F(k)}W_{\xi_{\alpha}}W_{\xi_{\alpha}}^{*}\right)V=V^{*}\pi\left(W^{(k)}(\sum_{\alpha\in F(k)}T_{\xi_{\alpha}}T_{\xi_{\alpha}}^{*})W^{(k)*}\right)V.\label{eq:z-2}
\end{multline}
Combining \eqref{eq:z-1} and \eqref{eq:z-2} we conclude that for
each positive integer $n$ and for each choice of finite subsets $F(k)\subseteq A(k)$,
\begin{multline}
\sum_{k=1}^{n}\left(\sum_{\alpha\in F(k)}\mathfrak{z}^{(k)}(\xi_{\alpha}\otimes\xi_{\alpha}^{*}\otimes I_{H})\mathfrak{z}^{(k)*}\right)\\
\leq V^{*}\pi\left(\sum_{k=1}^{n}W^{(k)}(\sum_{\alpha\in F(k)}T_{\xi_{\alpha}}T_{\xi_{\alpha}}^{*})W^{(k)*}\right)V\\
=V^{*}\pi\left(\sum_{k=1}^{n}W^{(k)}(\sum_{\alpha\in F(k)}(\xi_{\alpha}\otimes\xi_{\alpha}^{*})\otimes I_{\mathcal{F}(E)})W^{(k)*}\right)V.\label{eq:z-3}
\end{multline}
The expression $\sum_{k=1}^{n}W^{(k)}(\sum_{\alpha\in F(k)}(\xi_{\alpha}\otimes\xi_{\alpha}^{*})\otimes I_{\mathcal{F}(E)})W^{(k)*}$
is dominated by $\sum_{k=1}^{n}W^{(k)}\widetilde{X}_{k}W^{(k)*}$,
for each $n$. Each of these, in turn, is dominated by $I_{\mathcal{F}(E)}$,
by virtue of Lemma \ref{W}. Therefore, even though we do not know
$\pi$ is normal, we still may conclude that the supremum of the operators on the left-hand side of equation  \eqref{eq:z-3} , which is $\sum_{k=1}^{\infty}\mathfrak{z}^{(k)}(X_{k}\otimes I_{H})\mathfrak{z}^{(k)*}$,
is dominated by $I_{H}$.
\end{proof}

\section{Dilations and Coextensions\label{sec:Dilations}}

In this section we study dilations and coextensions of completely
contractive representations of $\mathcal{T}_{+}(E,Z)$. Recall that
in general, if $\rho_{i}:A\to B(H_{i})$ is a representation of an
algebra $A$ on a Hilbert space $H_{i}$, $i=1,2$, then $\rho_{2}$
is called a \emph{dilation} of $\rho_{1}$ in case there is a Hilbert
space isometry $V:H_{1}\to H_{2}$ such that $\rho_{1}(a)=V^{*}\rho_{2}(a)V$
for all $a\in A$. If $\rho_{2}$ is a dilation of $\rho_{1}$, then
the operators $\rho_{2}(a)$ admit a block matricial decomposition
which may be described as follows. Let $\mathcal{M}_{1}$ be the smallest
subspace of $H_{2}$ that contains the range of $V$ and is invariant
under $\rho_{2}(a)$ for all $a\in A$ and let $\mathcal{M}_{2}$
be the orthogonal complement of $VH_{1}$ in $\mathcal{M}_{1}$. Then,
as Sarason showed in \cite{Sarason1965}, corresponding to the family
of subspaces $\{0\}\subseteq\mathcal{\mathcal{M}}_{2}\subseteq\mathcal{M}_{1}\subseteq H_{2}$,
every $\rho_{2}(a)$ has an upper triangular block matrix decomposition:
\[
\rho_{2}(a)=\begin{bmatrix}\sigma_{1}(a) & * & *\\
0 & \sigma_{2}(a) & *\\
0 & 0 & \sigma_{3}(a)
\end{bmatrix},
\]
where each $\sigma_{i}$ is a representation of $A$ and $\sigma_{2}(\cdot)= VV^{*}\rho_{2}(\cdot)V|_{VH_{1}}$
is unitarily equivalent to $\rho_{1}$. If $\mathcal{M}_{1}= H_{2}$,
so $\mathcal{\rho}_{2}(\cdot)=\begin{bmatrix}\sigma_{1}(\cdot) & *\\
0 & \sigma_{2}(\cdot)
\end{bmatrix}$, then not only is $\rho_{2}$ a dilation of $\rho_{1}$, it is also
called a \emph{coextension} of $\rho_{1}$. Thus, a dilation $\rho_{2}$
of $\rho_{1}(\cdot)=V^{*}\rho_{2}(\cdot)V$ is a coextension of $\rho_{1}$
if and only if the range of $V$ is invariant under the operators
$\{\rho_{2}(a)^{*}\mid a\in A\}$. 

In our setting, if $\sigma\times\mathfrak{z}$ is a completely contractive
representation of $\mathcal{T}_{+}(E,Z)$ on a Hilbert space $H$,
then the role of the Poisson kernel $K(\mathfrak{z})$ (Definition
\ref{Poisson_Kernel}) is to be an isometry that realizes the induced
representation $\sigma^{\mathcal{F}(E)}$ (restricted to $\mathcal{T}_{+}(E)$)
as a coextension of $\sigma\times\mathfrak{z}$. Of course, $K(\mathfrak{z})$
is an isometry if and only if $Q_{\mathfrak{z}}=0$ in Lemma \ref{PoissonIsometry},
but when it is, Corollary \ref{representations} shows that it realizes
$\sigma^{\mathcal{F}(E)}$ as a coextension of $\sigma\times\mathfrak{z}$.
In Theorem \ref{dilation1}, we will show how to compensate for
the possibility that $K(\mathfrak{z})$ may not be an isometry. 

Recall that the $Z$-Toeplitz algebra, $\mathcal{T}(E,Z)$, is the
$C^{*}$-subalgebra of $\mathcal{L}(\mathcal{F}(E))$ generated by
$\mathcal{T}_{+}(E,Z)$. Our first objective is to identify the representations
$\sigma\times\mathfrak{v}$ of $\mathcal{T}_{+}(E,Z)$ that are the
restrictions to $\mathcal{T}_{+}(E,Z)$ of $C^{*}$-representations
of $\mathcal{T}(E,Z)$. Our results may seem somewhat provisional
because their proofs are based on hypotheses that are not universally
satisfied. Nevertheless, they do have considerable applicability and
their inclusion here is waranted because they help to illuminate the
representation theory of $\mathcal{T}_{+}(E,Z)$. 

Throughout this section, $M$, $E,$ $X$ and $Z$ will be fixed.
There are three special hypotheses that play a role in our analysis. 
\begin{quote}
\textbf{Hypothesis A. }$\mathcal{K}(E^{\otimes k})=\mathcal{L}(E^{\otimes k})$
for all $k\geq0$.
\end{quote}
This hypothesis is the assumption that $E$ is finitely generated
in a strong sense. It is satisfied in many examples. We know of no
finitely generated correspondence $E$ such that $\mathcal{K}(E^{\otimes k})\neq\mathcal{L}(E^{\otimes k})$
for some $k\geq2$.
\begin{quote}
\textbf{Hypothesis B.} The $Z_{k}$'s are uniformly bounded below,
i.e., there is an $\epsilon>0$ such that $Z_{k}\geq\epsilon I_{E^{\otimes k}}$
for all $k\geq1$.
\end{quote}
Hypotheses $B$ guarentees that $Z$ is invertible in $\mathcal{L}(\mathcal{F}(E))$,
of course. Hypotheses $A$ and $B$, together, guarentee that the
unweighted tensor algebra$,$ $\mathcal{T}_{+}(E)$, is contained
in $\mathcal{T}_{+}(E,Z)$, as we shall see.
\begin{quote}
\textbf{Hypothesis C.} There is an $a$ such that \[(\limsup\Vert X_{k}\Vert^{\frac{1}{k}})I_{E}<aI_{E}\leq X_{1}.\]
\end{quote}
Note that Hypothesis C is satisfied if there are only finitely many
nonzero weights $X_{k}$. Hypotheses A, B, and C, together guarentee
that a variant of our generalization of Wold's decomposition theorem
\cite[Theorem 2.9]{Muhly1999} is valid for represenations of $\mathcal{T}(E,Z)$
as we shall see in Theorem \ref{thm:Wold}. 

By virtue of Arveson's dilation theory \cite{Arv1969a}, every completely
contractive representation of $\mathcal{T}_{+}(E,Z)$ can be dilated
to a $C^{*}$-representation of $\mathcal{T}(E,Z)$. In general, it
does not seem that one can say a lot about the structure of the dilation.
However, under our hypotheses there are dilations that are also coextensions.
We show also that without assuming Hypotheses A, B, and C, it is possible
to build a coextension for any completely contractive representation
of $\mathcal{T}_{+}(E,Z)$ via an explicit construction. We show that
under Hypotheses A, B, and C, the induced part of the explicit coextension
agrees with the induced part of the coextension that arises in our
dilation process.

In the setting of unweighted tensor and Toeplitz algebras, our Wold
decomposition theorem asserts that every $C^{*}$-representation of
the Toe\-plitz algebra of a $C^{*}$-correspondence decomposes into
the direct sum of an induced representation of the algebra and a fully
coisometric representation in the sense of \cite[Definition 5.3]{Muhly1998a}.
The formally same result is valid in the weighted setting provided
the notion of ``fully coisometric'' is modified according to
\begin{defn}
\label{Def:Fully_Coisometric} Let $\sigma$ normal $W^{*}$-representation
of $M$ on a Hilbert space $H$. An element $\mathfrak{v}\in\overline{D(X,\sigma)}$
is called \emph{fully coisometric} in case 
\[
\sum_{k\geq1}\mathfrak{v}^{(k)}(X_{k}\otimes I_{H})\mathfrak{v}^{(k)*}=I_{H}.
\]
In this case, we shall also call $\sigma\times\mathfrak{v}$ a \emph{fully
coisometric }representation.
\end{defn}
Our generalization of Theorem 2.9 of \cite{Muhly1999} is the following
theorem, which generalizes also Theorem 3.8 of Popescu's Memoir \cite{Popescu2010}.
\begin{thm}
\label{thm:Wold}Suppose Hypotheses A, B, and C are satisfied and
suppose $\pi$ is a $C^{*}$-representation of $\mathcal{T}(E,Z)$
on a Hilbert space $H$. Then $\pi$ decomposes as $\pi=\pi_{ind}\oplus\pi_{full}$
acting on $H=H_{ind}\oplus H_{full}$, where $\pi_{ind}$ is (unitarily
equivalent to) an induced representation and $\pi_{full}$ is fully
coisometric. 
\end{thm}
We will break the proof into a series of lemmas that put into evidence
the rolls of the hypotheses used. The initial focus will be to show
that the hypotheses guarantee that the (unweighted) Toeplitz algebra
$\mathcal{T}(E)$ is contained in $\mathcal{T}(E,Z)$. That uses Hypotheses
$A$ and $B$ only. Hypothesis $C$ is used near the end of the proof
of Theorem \ref{thm:Wold} to secure the convergence of the series
that shows that $\pi_{full}$ is fully coisometric.
\begin{lem}
\label{comp1} For $k,m\in\mathbb{N}$ and $\xi,\eta\in E^{\otimes m}$,
the restriction of $W_{\xi}W_{\eta}^{*}$ to the $m+k$-summand of
the Fock space $\mathcal{F}(E)$ is given by the formula 
\begin{multline*}
W_{\xi}W_{\eta}^{*}|E^{\otimes(m+k)}\\=Z^{(m+k)}(I_{m}\otimes Z^{(k)})^{-1}((\xi\otimes\eta^{*})\otimes I_{k})(I_{m}\otimes Z^{(k)*})^{-1}Z^{(m+k)*},
\end{multline*}
where $\xi\otimes\eta^{*}$ is the operator in $\mathcal{K}(E^{\otimes m})$
defined by $\xi\otimes\eta^{*}(\zeta)=\xi\langle\eta,\zeta\rangle$
for $\zeta\in E^{\otimes m}$. \end{lem}
\begin{proof}
This immediate from \eqref{Wxi}.
\end{proof}
In the following lemma, we shall write $\widetilde{A}$, when $A\in\mathcal{L}(E^{\otimes k})$,
for the operator $\rm{diag}[A(i)]$ in $\mathcal{L}(\mathcal{F}(E))$
that is diagonal with respect to the direct sum decomposition of $\mathcal{F}(E)$
as $\sum_{n\geq0}E^{\otimes n}$, where $A(i)=0$, $0\leq i<k$, and
$A(i)=A\otimes I_{E^{\otimes(i-k)}}$, $i\geq k$. This notation is
simply an extension of equation \eqref{eq:Xk-tilde}.
\begin{lem}
\label{toeplitz} If Hypothesis A is satisfied, then 
\[
W^{(m)}\widetilde{X}_{m}W^{(m)*}\in\mathcal{T}(E,Z)
\]
for all $m\geq1$. If, in addition, Hypothesis C is satisfied, then
the series $\sum_{k=1}^{\infty}W^{(k)}\widetilde{X_{k}}W^{(k)*}$
converges in norm to $I-P_{0}$, where $P_{0}$ is the projection
of the Fock correspondence onto the zero$^{th}$ summand.\end{lem}
\begin{proof}
Lemma~\ref{comp1} implies that $W^{(k)}\widetilde{D}W^{(k)*}\in\mathcal{T}(E,Z)$
for every $D\in\mathcal{K}(E^{\otimes k})$. Hypothesis A, then implies
that $X_{m}\in\mathcal{K}(E^{\otimes m})$ for every $m$, so $W^{(m)}\widetilde{X}_{m}W^{(m)*}\in\mathcal{T}(E,Z)$
for every $m\geq1$. To prove the second assertion, note first that
the series $\sum_{k=1}^{\infty}W^{(k)}\widetilde{X_{k}}W^{(k)*}$
converges to $I-P_{0}$ ultraweakly by Lemma~\ref{W}. So all we
need to show is that the series $\sum_{k=1}^{\infty}W^{(k)}\widetilde{X_{k}}W^{(k)*}$
converges in norm. If $a$ is a positive number satisfying Hypothesis
C, then the last paragraph of the proof of Theorem \ref{thm:The-potential-of-Phiz}
shows that $\Vert Z_{k}\Vert\leq1/\sqrt{a}$ for all $k$. Consequently,
$\Vert W^{(k)}\widetilde{X}_{k}W^{(k)*}\Vert\leq\frac{1}{a^{k}}\Vert X_{k}\Vert$.
Since $(\overline{\lim}||X_{k}||^{1/k})<a$, by Hypothesis C, $\overline{\lim}\Vert W^{(k)}\widetilde{X}_{k}W^{(k)*}\Vert^{\frac{1}{k}}<1$.
Consequently, the series $\sum_{k=1}^{\infty}W^{(k)}\widetilde{X_{k}}W^{(k)*}$
converges in $\mathcal{T}(E,Z)$, by the root test.
\end{proof}
We shall write $D_{k}$, $k\geq1$, for the diagonal operator on $\mathcal{F}(E)$
defined by 
\[
D_{k}:=\rm{diag}[0,0,\ldots,0,Z^{(k)},Z^{(k+1)}(I_{k}\otimes Z)^{-1},Z^{(k+2)}(I_{k}\otimes Z^{(2)})^{-1},\ldots]
\]
(with $k$ zeros). So that, for $\xi\in E^{\otimes k}$, we may write
\[
W_{\xi}=D_{k}T_{\xi}.
\]
We shall also write 
\[
D^{(-k)}=(0,0,\ldots,0,(Z^{(k+1)}(I_{k}\otimes Z)^{-1})^{-1},(Z^{(k+2)}(I_{k}\otimes Z^{(2)})^{-1})^{-1},\ldots)
\]
(again, $k$ zeros). Consequently, 
\[
D_{k}D^{(-k)}=D^{(-k)}D_{k}=\sum_{j=k}^{\infty}P_{j},
\]
where $P_{j}$ is the projection of $\mathcal{F}(E)$ onto $E^{\otimes j}$,
and for $\xi\in E^{\otimes k}$, 
\[
T_{\xi}=D^{(-k)}W_{\xi}.
\]

\begin{prop}
\label{TE} If Hypotheses A and B are satisfied, then
\begin{enumerate}
\item $\mathcal{T}(E)\subseteq\mathcal{T}(E,Z)$. In particular, for each
$j\geq0$, $P_{j}\in\mathcal{T}(E,Z)$.
\item For every $k\geq1$, $D_{k}$ and $D^{(-k)}$ lie in $\mathcal{T}(E,Z)\cap(\varphi_{\infty}(M))'$.
\end{enumerate}
\end{prop}
\begin{proof}
Recall that $W_{\xi}=D_{1}T_{\xi}$, where $D_{1}=\rm{diag}[0,Z_{1},Z_{2},\ldots]$
and $\xi\in E$. Thus, for $\xi,\eta\in E$, $D_{1}T_{\xi}T_{\eta}^{*}D_{1}^{*}=W_{\xi}W_{\eta}^{*}\in\mathcal{T}(E,Z)$.
Since we can write $I_{E}$ as a finite sum, $I_{E}=\sum_{k}\xi_{k}\otimes\eta_{k}^{*}$,
by Hypothesis A, $\sum_{k}T_{\xi_{k}}T_{\eta_{k}}^{*}=I-P_{0}$. Thus
$D_{1}^{2}=D_{1}(I-P_{0})D_{1}^{*}\in\mathcal{T}(E,Z)$. The spectrum
of $D_{1}^{2}$ is contained in $\{0\}\cup[\epsilon^{2},\infty)$
and, applying the function $f(t)$ defined to be $t^{-1/2}$ on $[\epsilon^{2},\infty)$
and $f(0)=0$ (which is continuous on the spectrum of $D_{1}^{2}$),
we see that $D^{(-1)}:=\rm{diag}[0,Z_{1}^{-1},Z_{2}^{-1},\ldots]$ lies
in $\mathcal{T}(E,Z)$. Multiplying $D^{(-1)}$ by $D_{1}T_{\xi}$,
we find that $T_{\xi}\in\mathcal{T}(E,Z)$, completing the proof of
the first assertion. For the second, observe that for $k\geq1$, we
may write $I_{E^{\otimes k}}=\sum_{j}\xi_{j}\otimes\eta_{j}^{*}$,
another finite sum. Consequently, $D_{k}=\sum_{j}D_{k}T_{\xi_{j}}T_{\eta_{j}}^{*}=\sum_{j}W_{\xi_{j}}T_{\eta_{j}}^{*}\in\mathcal{T}(E,Z)$.
Also, $\sum_{j=0}^{k-1}P_{j}=I-\sum_{j}T_{\xi_{j}}T_{\eta_{j}}^{*}\in\mathcal{T}(E)\subseteq\mathcal{T}(E,Z)$
and, so $D^{(-k)}=(D_{k}+\sum_{j=0}^{k-1}P_{j})^{-1}-\sum_{j=0}^{k-1}P_{j}\in\mathcal{T}(E,Z)$.
\end{proof}
We next want to focus on a general completely contractive representations
of $\mathcal{T}_{+}(E,Z)$, $\tau\times\mathfrak{v},$ that arises
as the restriction of a $C^{*}$-representation $\pi$ of $\mathcal{T}(E,Z)$.
We shall write the common Hilbert space of $\tau\times\mathfrak{v}$
and $\pi$ as $K$ and we shall assume that Hypotheses $A$ and $B$
are in force so that we may invoke Proposition~\ref{TE} and use
its notation.

Since $D_{k},D^{-k}$ and $P_{0}$, $k\geq1$, all lie in $\mathcal{T}(E,Z)$,
we may form $P_{\mathcal{M}}:=\pi(P_{0})$, $d_{k}:=\pi(D_{k})$ and
$d^{(-k)}:=\pi(D^{-k})$. It results that $d_{k}d^{(-k)}=d^{(-k)}d_{k}=I-\pi(\sum_{j=0}^{k-1}P_{j})$
and, in particular, $d_{1}d^{(-1)}=d^{(-1)}d_{1}=I-P_{\mathcal{M}}$.
Since $\pi(\varphi_{\infty}(M))=\tau(M)$ and $D_{k},D^{(-k)},P_{0}\in(\varphi_{\infty}(M))'$,
on $K$ $d_{k},d^{(-k)},P_{\mathcal{M}}\in\tau(M)'$, and $\mathcal{M}$
is a subspace of $K$ that reduces $\tau$.

Further, for $\xi\in E$, $\pi(W_{\xi})=(\tau\times\mathfrak{v})(W_{\xi})=\mathfrak{v}L_{\xi}$
and $\pi(T_{\xi})=\pi(D^{(-1)}W_{\xi})=d^{(-1)}\mathfrak{v}L_{\xi}$.
Note that $d^{(-1)}\mathfrak{v}\in E^{\tau*}$. Since $\mathcal{T}(E)\subseteq\mathcal{T}(E,Z)$,
we may restrict $\pi$ to $\mathcal{T}(E)$ and $\mathcal{T}_{+}(E)$,
getting a $C^{*}$-represen\-ta\-tion of $\mathcal{T}(E)$ whose
restriction to $\mathcal{T}_{+}(E)$ is $\tau\times d^{(-1)}\mathfrak{v}$.
Consequently, $d^{(-1)}\mathfrak{v}$ is an isometry.

It is immediate that if $\xi,\theta\in E$, then $\pi(T_{\xi}T_{\theta})=d^{(-1)}\mathfrak{v}L_{\xi}d^{(-1)}\mathfrak{v}L_{\theta}=(d^{(-1)}\mathfrak{v})^{(2)}L_{\xi\otimes\theta}$.
It follows easily that for every $\eta\in E^{\otimes k}$ (where $k\geq1$),
$\pi(T_{\eta})=(d^{(-1)}\mathfrak{v})^{(k)}L_{\eta}$.

Also, for every $k\geq1$, $\eta\in E^{\otimes(k-1)}$ and $\xi\in E$,
we have 
\[
T_{Z_{k}(\xi\otimes\eta)}P_{0}=W_{\xi}T_{\eta}P_{0}.
\]
So, applying $\pi$, we find that 
\[
(d^{(-1)}\mathfrak{v})^{(k)}L_{Z_{k}(\xi\otimes\eta)}P_{\mathcal{M}}=\mathfrak{v}L_{\xi}(d^{(-1)}\mathfrak{v})^{(k-1)}L_{\eta}P_{\mathcal{M}}
\]
and for every $g\in\mathcal{M}$,

\begin{equation}
(d^{(-1)}\mathfrak{v})^{(k)}(Z_{k}(\xi\otimes\eta)\otimes g)=\mathfrak{v}(\xi\otimes(d^{(-1)}\mathfrak{v})^{(k-1)}(\eta\otimes g)).\label{dv}
\end{equation}

Finally note that, for $k\geq1$, 
\[
(d^{(-1)}\mathfrak{v})^{(k)}=d^{(-k)}\mathfrak{v}^{(k)}
\]
because 
\begin{multline*}D^{(-1)}W_{\xi_{1}}D^{(-1)}W_{\xi_{2}}\cdots D^{(-1)}W_{\xi_{k}}\\=T_{\xi_{1}}T_{\xi_{2}}\cdots T_{\xi_{k}}=T_{\xi_{1}\otimes\xi_{2}\cdots\otimes\xi_{k}}=D^{(-k)}W_{\xi_{1}\otimes\xi_{2}\cdots\otimes\xi_{k}}.
\end{multline*}
\begin{lem}
\label{induced} With the notation established, we have
\begin{enumerate}
\item The spaces $(d_{1}\mathfrak{v})^{(k)}(E^{\otimes k}\otimes_{\tau}\mathcal{M}$
\; and \; $(d_{1}\mathfrak{v})^{(k)}(E^{\otimes m}\otimes_{\tau}\mathcal{M})$
are orthogonal for  $k\ne m$. 
\item If $K_{0}$ is the subspace of $K$ defined by 
\[K_{0}:=\sum_{k=0}^{\oplus}(d_{1}\mathfrak{v})^{(k)}(E^{\otimes k}\otimes_{\tau}\mathcal{M}),\]
and if $V:\mathcal{F}(E)\otimes_{\tau}\mathcal{M}\rightarrow K_{0}$
is the operator that maps $\eta\otimes g$ to $(d_{1}\mathfrak{v})^{(k)}(\eta\otimes g)$,
$\eta\in E^{\otimes k}$, $g\in\mathcal{M}$, then $V$ is Hilbert
space isomorphism and, for $\xi\in E$, 
\begin{equation}
V(W_{\xi}\otimes I_{\mathcal{M}})=\pi(W_{\xi})V.\label{Vintertwine}
\end{equation}

\end{enumerate}
\end{lem}
\begin{proof}
For $k\ne m$ and $\eta_{k}\in E^{\otimes k},\eta_{m}\in E^{\otimes m}$,
we have $P_{0}T_{\eta_{k}}^{*}T_{\eta_{m}}P_{0}=0$. Apply $\pi$
to get $P_{\mathcal{M}}L_{\eta_{k}}^{*}(d_{1}\mathfrak{v})^{(k)*}(d_{1}\mathfrak{v})^{(m)}L_{\eta_{m}}P_{\mathcal{M}}=0$,
which proves $1$. Since, as was mentioned above, $d_{1}\mathfrak{v}$
is a isometry (and thus so is $(d_{1}\mathfrak{v})^{(k)}$ for every
$k\geq1$), $V$ is a Hilbert space isomorphism. To prove Equation
(\ref{Vintertwine}) it suffices to apply both sides to $\eta\otimes g\in E^{\otimes k}\otimes\mathcal{M}$.
But this is just Equation (\ref{dv}).
\end{proof}
{}
\begin{proof}
\emph{(of Theorem \ref{thm:Wold})} We continue to use the notation
just established and recall that $\pi(P_{0})=P_{\mathcal{M}}$ and
$\pi(T_{\xi})=d_{1}\mathfrak{v}L_{\xi}$ for $\xi\in E$. Since Hypothesis
$A$ allows us to write $I_{E}$ as a finite sum, $I_{E}=\sum\xi_{j}\otimes\xi_{j}^{*}$,
we find that $\sum T_{\xi_{j}}P_{0}T_{\xi_{j}}^{*}=P_{1}$. Apply
$\pi$ to conclude that \[\sum d_{1}\mathfrak{v}L_{\xi_{j}}P_{\mathcal{M}}L_{\xi_{j}}^{*}(d_{1}\mathfrak{v})^{*}=\pi(P_{1}).\]
Thus $\pi(P_{1})$ is the projection onto $d_{1}\mathfrak{v}(E\otimes_{\tau}\mathcal{M})$.
Similarly, $\pi(P_{k})$ is the projection onto $(d_{1}\mathfrak{v})^{(k)}(E^{\otimes k}\otimes_{\tau}\mathcal{M})$.
Write $Q_{n}=\sum_{k=n}^{\infty}P_{k}$. Then we see that $\{\pi(Q_{n})\}$
is a decreasing sequence of projections that converges in the strong
operator topology to the projection onto $K_{0}^{\perp}$ where $K_{0}=\sum_{k=0}^{\oplus}(d_{1}\mathfrak{v})^{(k)}(E^{\otimes k}\otimes_{\tau}\mathcal{M})\subseteq K$,
by Lemma~\ref{induced}.

Since, for every $\xi\in E$ and $n\geq0$, $W_{\xi}Q_{n}=Q_{n}W_{\xi}Q_{n}$,
we see that $\pi(W_{\xi})$ leaves each $\pi(Q_{n})K$ invariant and
so, $\pi(W_{\xi})$ leaves $K_{0}^{\perp}$ invariant. On the other
hand, equation \ref{Vintertwine} shows that $K_{0}$ is invariant
for each $\pi(W_{\xi})$. Thus $K_{0}$ reduces $\pi$ and by Lemma~\ref{induced}
the restriction of $\pi$ to $K_{0}$ is unitarily equivalent to an
induced representation. Therefore, we shall rename $K_{0}$ as $K_{ind}$,
write $\pi_{ind}:=\pi\vert K_{ind}$, and in anticipation of what
we shall show next, we shall write $K_{full}:=K_{0}^{\perp}$ and
$\pi_{full}:=\pi\vert K_{full}$. 

To see that $\pi_{full}$ is in fact fully coisometric in the sense
of Definition \ref{Def:Fully_Coisometric}, note that $\pi_{full}(P_{0})=0$
by construction. Note also that by Lemma \ref{toeplitz}, the series
$\sum_{k=0}^{\infty}W^{(k)}\widetilde{X}_{k}W^{(k)*}$
converges in norm to $I-P_{0}$. Hence $\pi_{full}(\sum_{k=0}^{\infty}W^{(k)}\widetilde{X}_{k}W^{(k)*})=\pi_{full}(I)=I_{K_{full}}$.
It suffices, therefore, to prove that 
\[
\pi_{full}(W^{(k)}\widetilde{X}_{k}W^{(k)*})=\mathfrak{v}_{0}^{(k)}(X_{k}\otimes I_{K_{full}})\mathfrak{v}_{0}^{(k)*},
\]
where $\mathfrak{v}_{0}$ is the restriction of $\mathfrak{v}$ to
$K_{full}$. To this end, observe first that $\pi_{full}=\tau_{0}\otimes\mathfrak{v}_{0}$,
where $\tau_{0}$ is the restriction of $\tau$ to $K_{full}$ and
observe that since $X_{k}$ is invertible, Hypothesis $A$ implies
that we may write $X_{k}$ as a finite sum $X_{k}=\sum_{i}\xi_{i}\otimes\xi_{i}^{*}$,
where $\xi_{i}\in E^{\otimes k}$. Consequently, $W^{(k)}\widetilde{X}_{k}W^{(k)*}=\sum W^{(k)}\widetilde{L}_{\xi_{i}}\widetilde{L}_{\xi_{i}}^{*}W^{(k)*}=\sum W_{\xi_{i}}^{(k)}W_{\xi_{i}}^{(k)*}$,
where $\widetilde{L}_{\xi_{i}}:\mathcal{F}(E)\to E^{\otimes k}\otimes\mathcal{F}(E)$
is the insertion operator determined by $\xi_{i}$. Therefore, 
\[
\pi_{full}(W^{(k)}\widetilde{X}_{k}W^{(k)*})=\sum\pi_{full}(W_{\xi_{i}}^{(k)})\pi_{full}(W_{\xi_{i}}^{(k)})^{*}.
\]
Since $\pi_{full}=\tau_{0}\otimes\mathfrak{v}_{0}$, $\pi_{full}(W_{\xi_{i}}^{(k)})=\mathfrak{v}_{0}^{(k)}L_{\xi_{i}}$,
where $L_{\xi_{i}}:K_{full}\to E^{\otimes k}\otimes K_{full}$ is
the insertion operator as in Remark \ref{Rm:_Algebraic_Extension}.
Finally, then, 
\begin{multline*}
\pi_{full}(W^{(k)}\widetilde{X}_{k}W^{(k)*})=\sum\pi_{full}(W_{\xi_{i}}^{(k)})\pi_{full}(W_{\xi_{i}}^{(k)})^{*}\\
=\sum\mathfrak{v}_{0}^{(k)}L_{\xi_{i}}L_{\xi_{i}}^{*}\mathfrak{v}_{0}^{(k)*}=\mathfrak{v}_{0}^{(k)}(\sum\xi_{i}\otimes\xi_{i}^{*}\otimes I_{K_{full}})\mathfrak{v}_{0}^{(k)*}\\
=\mathfrak{v}_{0}^{(k)}(X_{k}\otimes I_{K_{full}})\mathfrak{v}_{0}^{(k)*},
\end{multline*}
as was to be proved.
\end{proof}
The following theorem generalizes Theorem 3.12 of \cite{Popescu2010}.

\begin{thm}
\label{dilation2} Suppose that Hypotheses $A$, $B$, and $C$ are
satisfied. Suppose also that $\sigma$ is a normal representation
of $M$ on a Hilbert space $H$ and that $\mathfrak{z}\in\overline{D(X,\sigma)}$.
Then $\sigma\times\mathfrak{z}$ admits a $C^{*}$-dilation that is
also a coextension.\end{thm}
\begin{proof}
Consider the map $\Psi_{Z,\mathfrak{z}}$ from Proposition \ref{Psi}.
It is a unital completely positive map on the operator system $\overline{span}\{TS^{*}\mid T,S\in\mathcal{T}_{+}(E,Z)\}$
with values in $B(H_{\sigma})$ that restricts to the completely contractive
representation $\sigma\times\mathfrak{z}$ on $\mathcal{T}_{+}(E,Z)$.
We may apply Arveson's extension of the Hahn-Banach theorem \cite[Theorem 1.2.3]{Arv1969a}
to extend $\Psi_{Z,\mathfrak{z}}$ to a completely positive map from
$\mathcal{T}(E,Z)$ to $B(H_{\sigma})$ (we keep the same symbol for
the extension) and then employ Stinespring's theorem to dilate $\Psi_{Z,\mathfrak{z}}$
to a $C^{*}$-representation $\pi$ of $\mathcal{T}(E,Z)$ on a Hilbert
space $K$ that contains $H$. Thus we may write $\Psi_{Z,\mathfrak{z}}(X)=P_{H}\pi(X)P_{H}$
for all $X\in\mathcal{T}(E,Z)$. We show that $\pi$ is a coextension.
To this end, write $I_{K}=I=P_{H}+P_{H}^{\perp}$ and observe that
\begin{multline*}
\Psi_{Z,\mathfrak{z}}(W_{\xi}W_{\eta}^{*})=P_{H}\pi(W_{\xi}W_{\eta}^{*})P_{H}=P_{H}\pi(W_{\xi})(P_{H}+P_{H}^{\perp})\pi(W_{\eta})^{*}P_{H}\\
=P_{H}\pi(W_{\xi})P_{H}\pi(W_{\eta})^{*}P_{H}+P_{H}\pi(W_{\xi})P_{H}^{\perp}\pi(W_{\eta})^{*}P_{H}.
\end{multline*}
By Proposition \ref{Psi}, 
\[\Psi_{Z,\mathfrak{z}}(W_{\xi}W_{\eta}^{*})=\mathfrak{z}^{(k)}L_{\xi}L_{\eta}^{*}\mathfrak{z}^{(m)*}\\=\sigma\times\mathfrak{z}(W_{\xi})\sigma\times\mathfrak{z}(W_{\eta})^{*}.
\]
On the other hand, since $\pi$ dilates $\Psi_{Z,\mathfrak{z}}$,
$P_{H}\pi(W_{\xi})P_{H}\pi(W_{\eta})^{*}P_{H}=\sigma\times\mathfrak{z}(W_{\xi})\sigma\times\mathfrak{z}(W_{\eta})^{*}$
also. Thus we conclude that $P_{H}\pi(W_{\xi})P_{H}^{\perp}\pi(W_{\eta})^{*}P_{H}=0$
for all $\xi$ and $\eta$. Taking $\xi=\eta$, then shows that $P_{H}\pi(W_{\xi})P_{H}^{\perp}=0$
for all $\xi$. So $\pi$ is a coextension as required.\end{proof}
\begin{rem}
\label{rem:noncoextension}We want to emphasize that we are not claiming
in Theorem \ref{dilation2} that all $C^{*}$-dilations of a representation
$\sigma\times\mathfrak{z}$ are coextensions. Rather, we showed that
$C^{*}$-dilations that are coextensions exist by constructing them
from maps of the form $\Psi_{Z,\mathfrak{z}}$ defined in Proposition
\ref{Psi}. Although we know of no particular examples, there may
well be dilations that do not come from such maps. We want to emphasize,
too, that the existence of maps of the form $\Psi_{Z,\mathfrak{z}}$
reflects features of $\mathcal{T}_{+}(E,Z)$ that may not be shared
by other operator algebras.
\end{rem}
We turn now to another, somewhat more constructive approach to producing
a coextension of a representation $\sigma\times\mathfrak{z}.$ Our arguments are inspired in part by those of Viselter in \cite[Theorem 2.22]{Viselter2011}).To
present it, we require additional notation. Suppose $\sigma$ is a
normal representation of $M$ on $H$ and that $\mathfrak{z}\in\overline{D(X,\sigma)}$.
Then to build the induced representation that goes with $\mathfrak{z}$,
we write $\Delta_{*}(\mathfrak{z})$ for $(I_{H}-\sum_{k=1}^{\infty}\mathfrak{z}^{(k)}(X_{k}\otimes I_{H})\mathfrak{z}^{(k)*})^{1/2}$
as in Definition \ref{Poisson_Kernel}. We also write $\mathcal{D}_{*}$
for the subspace $\overline{\Delta_{*}(\mathfrak{z})H}$ and note
that it reduces $\sigma(M)$. Further, we write $K_{\mathcal{D}_{*}}$
for $\mathcal{F}(E)\otimes_{\sigma}\mathcal{D}_{*}$, where we abuse
notation slightly by writing $\sigma$ under the tensor sign instead
of its restriction to $\mathcal{D}_{*}$, $\sigma\vert\mathcal{D}_{*}$.
Likewise, we shall write $\sigma^{\mathcal{F}(E)}$ for the representation
of $\mathcal{L}(\mathcal{F}(E))$ on $K_{\mathcal{D}_{*}}$ that is
induced by the restriction of $\sigma$ to $\mathcal{D}_{*}$. Finally,
we write $\mathfrak{w}_{\mathcal{D}_{*}}$ for $W\otimes I_{\mathcal{D}_{*}}$,
which is an element of $\overline{D(X,\sigma{}^{\mathcal{F}(E)}\circ\varphi_{\infty})}$.
To build the fully coisometric representation that goes with $\mathfrak{z}$,
let $\Phi_{\mathfrak{z}}$ be the completely positive map on $\sigma(M)'$
defined in Definition \ref{Poisson_Kernel} and let $Q_{\mathfrak{z}}$
be the limit of $\{\Phi_{\mathfrak{z}}^{m}(I)\}_{m\geq1}$ as in Lemma
\ref{PoissonIsometry}. To simplify the notation, we write $Q$ for
$Q_{\mathfrak{z}}$. Then $Q\in\sigma(M)'$ and we let $\mathcal{U}$
be the closure of the range of $Q^{\frac{1}{2}}$. So $\mathcal{U}$
reduces $\sigma$. 
\begin{thm}
\label{dilation1} Suppose that $\sigma$ is a normal representation
of $M$ on $H$ and that $\mathfrak{z}\in\overline{D(X,\sigma)}$.
Then there is a Hilbert space $\mathcal{U}$ with a normal representation
$\tau$ of $M$ on $\mathcal{U}$ and a fully coisometric element
$\mathfrak{v}\in\overline{\mathcal{D}(X,\tau)}$ such that $(\sigma^{\mathcal{F}(E)}\otimes I_{\mathcal{D}_{*}}\oplus\tau)\times(\mathfrak{w}_{\mathcal{D}_{*}}\oplus\mathfrak{v})$
acting on $\mathcal{F}(E)\otimes_{\sigma}\mathcal{D}_{*}\oplus\mathcal{U}$
is a coextension of $\sigma\times\mathfrak{z}$. If $Q=0,$ then $\mathcal{U}=\{0\}$.
\end{thm}
We do not claim that $(\sigma^{\mathcal{F}(E)}\otimes I_{\mathcal{D}_{*}}\oplus\tau)\times(\mathfrak{w}_{\mathcal{D}_{*}}\oplus\mathfrak{v})$
extends to a $C^{*}$-representation of $\mathcal{T}(E,Z)$ acting
on $\mathcal{F}(E)\otimes_{\sigma}\mathcal{D}_{*}\oplus\mathcal{U}$,
but it is a coextension of $\sigma\times\mathfrak{z}$ in the sense
described at the beginning of this section.
\begin{proof}
Recall that $Q$ is the strong operator limit of $\{\Phi_{\mathfrak{z}}^{m}(I)\}$
and so it lies in $\sigma(M)'$. Let $\mathcal{U}$ be the subspace
$\overline{Q^{1/2}(H)}$, set $\tau:=\sigma|\mathcal{U}$ and let
$Y:H\rightarrow\mathcal{U}$ be defined by $Yh=Q^{1/2}h$. If we define
$\Lambda:Q^{1/2}H\rightarrow E\otimes\mathcal{U}$ by $\Lambda Yh=(I_{E}\otimes Y)\mathfrak{z}^{*}h$,
then for every $h\in H$, we have 
\[
||(I_{E}\otimes Y)\mathfrak{z}^{*}h||^{2}=\langle\mathfrak{z}(I\otimes Q)\mathfrak{z}^{*}h,h\rangle\leq||X_{1}^{-1}||\langle\mathfrak{z}(X_{1}\otimes Q)\mathfrak{z}^{*}h,h\rangle
\]
\[
\leq||X_{1}^{-1}||\langle\Phi_{\mathfrak{z}}(Q)h,h\rangle=||X_{1}^{-1}||\langle Qh,h\rangle=||X_{1}^{-1}||||Yh||^{2}
\]
where we use the fact that $\Phi_{\mathfrak{z}}(Q)=Q$. Thus $\Lambda$
is a well defined bounded linear operator and, therefore, may be extended
to all of $\mathcal{U}$, mapping to $E\otimes_{\tau}\mathcal{U}$.
Further, for $a\in M$ and $h\in H$, 
\begin{multline*}
\Lambda(\tau(a)Yh)=\Lambda(Y\sigma(a)h)=(I\otimes Y)\mathfrak{z}^{*}(\sigma(a))\\
=(I\otimes Y)(\varphi(a)\otimes I_{H})\mathfrak{z}^{*}h=(\varphi(a)\otimes I)\Lambda(Yh).
\end{multline*}
Consequently, if we set $\mathfrak{v}:=\Lambda^{*}$, we see that
$\mathfrak{v}\in E^{\tau}$. It follows from the definition that $Y^{*}\mathfrak{v}=\mathfrak{z}(I_{E}\otimes Y^{*})$.
We now claim that, for every $k\geq1$, $Y^{*}\mathfrak{v}^{(k)}=\mathfrak{z}^{(k)}(I_{E^{\otimes k}}\otimes Y^{*})$.
Since $Y^{*}\mathfrak{v}^{(k+1)}=Y^{*}\mathfrak{v}(I_{E}\otimes\mathfrak{v}^{(k)})=\mathfrak{z}(I_{E}\otimes Y^{*})(I_{E}\otimes\mathfrak{v}^{(k)})=\mathfrak{z}(I_{E}\otimes\mathfrak{z}^{(k)}(I_{k}\otimes Y^{*}))=\mathfrak{z}^{(k+1)}(I_{k+1}\otimes Y^{*})$for
every $k$ for which the formula holds, induction proves that the
formula is valid for all $k$. Consequently, $Y^{*}(\sum_{k}\mathfrak{v}^{(k)}(X_{k}\otimes I_{\mathcal{U}})\mathfrak{v}^{(k)*})Y=\sum_{k}\mathfrak{z}^{(k)}(X_{k}\otimes Q)\mathfrak{z}^{(k)}=\Phi_{\mathfrak{z}}(Q)=Q$.
It follows that for every $u_{1},u_{2}\in\mathcal{U}$, with $u_{i}=Q^{1/2}h_{i}=Yh_{i}$,
$h_{i}\in H$, $\langle(\sum_{k=1}^{\infty}\mathfrak{v}^{(k)}(X_{k}\otimes I_{\mathcal{U}})\mathfrak{v}^{(k)*})u_{1},u_{2}\rangle=\langle u_{1},u_{2}\rangle$,
which proves that $\mathfrak{v}$ is fully coisometric. 

Define $V:H\rightarrow K_{\mathcal{D}_{*}}\oplus\mathcal{U}$ by $Vh=K_{Z}(\mathfrak{z})h\oplus Yh$.
Since $K_{Z}(\mathfrak{z})^{*}K_{Z}(\mathfrak{z})+Q_{\mathfrak{}}=I_{H}$,
$V$ is an isometry. We use $V$ to identify $H$ with a subspace
of $K_{\mathcal{D}_{*}}\oplus\mathcal{U}$. Recall from Lemma~\ref{Kintertwines}
and Definition~\ref{ZPoisson} that, for $\xi\in E$, $K_{Z}(\mathfrak{z})L_{\xi}^{*}\mathfrak{z}^{*}=(W_{\xi}^{*}\otimes I_{\mathcal{D}_{*}})K_{Z}(\mathfrak{z})$.
Thus, for every $\xi\in E$ and $h\in H$, 
\begin{multline}
L_{\xi}^{*}\mathfrak{w}_{\mathcal{D}_{*}}^{*}K_{Z}(\mathfrak{z})h=(W_{\xi}^{*}\otimes I_{\mathcal{D}_{*}})K_{Z}(\mathfrak{z})h\\=K_{Z}(\mathfrak{z})L_{\xi}^{*}\mathfrak{z}^{*}h=L_{\xi}^{*}(I_{E}\otimes K_{Z}(\mathfrak{z}))\mathfrak{z}^{*}h.
\end{multline}
It follows that, for $h\in H$, 
\begin{equation}
\mathfrak{w}_{\mathcal{D}_{*}}^{*}K_{Z}(\mathfrak{z})h=(I_{E}\otimes K_{Z}(\mathfrak{z}))\mathfrak{z}^{*}h.
\end{equation}
We also have 
\begin{equation}
\mathfrak{v}^{*}Y=(I_{E}\otimes Y)\mathfrak{z}^{*}
\end{equation}
and we conclude that, 
\begin{equation}
V\mathfrak{z}^{*}=(\mathfrak{w}_{\mathcal{D}_{*}}^{*}\oplus\mathfrak{v}^{*})(I_{E}\otimes V).
\end{equation}

\end{proof}
Suppose $\sigma\times\mathfrak{z}$ is a completely contractive representation
of $\mathcal{T}_{+}(E,Z)$ on a Hilbert space $H$. In Theorem \ref{dilation1}
we showed how to construct an explicit coextension of $\sigma\times\mathfrak{z}$,
while in Theorem \ref{dilation2}, under special hypotheses, we showed
how to construct a coextension that also extendeds to a $C^{*}$-representation
of $\mathcal{T}(E,Z)$. A priori, these two coextensions have nothing
to do with one another. However, as we now show, the induced pieces
turn out to be unitarily equivalent. For this purpose we begin with
\begin{lem}
\label{idealP0} Write $P_{0}$ for the projection onto the zero$^{th}$
summand of $\mathcal{F}(E)$. Then: 
\begin{enumerate}
\item If $m\geq k$ and $\eta\in E^{\otimes k}$, while $\xi\in E^{\otimes m}$,
then there is $\theta\in E^{\otimes(m-k)}$ such that 
\[
W_{\eta}^{*}W_{\xi}P_{0}=W_{\theta}P_{0}.
\]

\item Assume that Hypotheses $A$, $B$, and $C$ are satisfied and let
$\mathcal{I}$ be the ideal in $\mathcal{T}(E,Z)$ generated by $P_{0}$.
Then 

\begin{enumerate}
\item 
\[
\mathcal{I}=\overline{span}\{W_{\xi}P_{0}W_{\theta}^{*}\;:\;\xi,\theta\in\cup_{n\geq0}E^{\otimes n}\}.
\]

\item If $K$ and $K_{0}$ are the Hilbert spaces of Theorem \ref{thm:Wold},
then 
\[
K_{0}=\overline{span}\{\pi(b)k:\; b\in\mathcal{I},\; k\in K\}.
\]

\end{enumerate}
\end{enumerate}
\end{lem}
\begin{proof}
A straightforward computation shows that, for every $a\in M=P_{0}(\mathcal{F}(E))$,
$W_{\eta}^{*}W_{\xi}a=W_{\theta}a$ where $\theta=R_{m-k}^{2}L_{\eta}^{*}R_{m}^{-2}\xi$.
This proves part (1).

Now write $\mathcal{C}$ for the subspace on the right hand side of
(2a). To prove (2a) it suffices to show that $\mathcal{C}$ is an
ideal in $\mathcal{T}(E,Z)$. Since it is a closed subspace and since
$\mathcal{T}(E,Z)$ is generated by $\{W_{\eta},W_{\eta}^{*},\varphi_{\infty}(a):\; a\in M,\eta\in\cup_{n>0}E^{\otimes n}\}$,
it suffices to show that $\mathcal{C}$ is closed under left and right
multiplication by the generators of $\mathcal{T}(E,Z)$. It will be
enough to consider left multiplication. But $\mathcal{C}$ is clearly
closed under left multiplication by $\varphi_{\infty}(a)$, $a\in M$
and by $W_{\eta}$, $\eta\in\cup_{n>0}E^{\otimes n}$. The fact that
it is closed under multiplication by $W_{\eta}^{*}$ for $\eta\in\cup_{n>0}E^{\otimes n}$
follows from part (1).

Since $\mathcal{T}(E)\subseteq\mathcal{T}(E,Z)$, it follows from
the proof of Theorem \ref{thm:Wold} that $P_{k}\in\mathcal{I}$ for
all $k\geq0$. Thus the space on the right hand side of (2b) contains
$\pi(P_{k})K$ for all $k\geq0$. Since $K_{0}$ is generated by these
ranges, we have $K_{0}\subseteq\overline{span}\{\pi(b)k:\; b\in\mathcal{I},\; k\in K\}$.
Since $K_{0}=\sum_{k=0}^{\infty}(d_{1}\mathfrak{v})^{(k)}(E^{\otimes k}\otimes\mathcal{M})$,
to show the reverse inclusion, it suffices to show that, for $\theta\in E^{\otimes k}$
and $\xi\in E^{\otimes m}$, the range of $\pi(W_{\xi}P_{0}W_{\theta}^{*})$
is contained in $K_{0}$ (using part (2a)). So we need to show that
$\mathfrak{v}^{(k)}(E^{\otimes k}\otimes\mathcal{M})\subseteq K_{0}$.
But $\mathfrak{v}^{(k)}(E^{\otimes k}\otimes\mathcal{M})\subseteq K_{0}=d_{k}d^{(-k)}\mathfrak{v}^{(k)}(E^{\otimes k}\otimes\mathcal{M})\subseteq K_{0}=d_{k}(d^{(-1)}\mathfrak{v})^{(k)}(E^{\otimes k}\otimes\mathcal{M})\subseteq d_{k}K_{0}$.
Since $d_{k}\in\pi(\mathcal{T}(E,Z))$ and $K_{0}$ is $\pi(\mathcal{T}(E,Z))$-reducing,
we are done.
\end{proof}
In preparation for the next proposition, fix a completely contractive
representation $\sigma\times\mathfrak{z}$ of $\mathcal{T}_{+}(E,Z)$
acting on the Hilbert space $H$. In the preparations for Theorem
\ref{dilation1} we formed the subspace $\mathcal{D}_{*}$ and noted
that it reduces $\sigma$. Consequently, we may and shall view $\mathcal{D}_{*}$
as a module over $M$, via $\sigma$. On the other hand, for Theorem
\ref{dilation2}, we formed the completely positive map $\Psi_{Z,\mathfrak{z}}$
on the operator system that we denoted $\mathfrak{S}$. We then took
an Arveson-Hahn-Banach extension of $\Psi_{Z,\mathfrak{z}}$ to all
of $\mathcal{T}(E,Z)$. We shall continue to denote the chosen extension
by $\Psi_{Z,\mathfrak{z}}$. We referred to the Stinespring dilation
$\pi$ of $\Psi_{Z,\mathfrak{z}}$, but we did not provide notation
for one of its constituent parts. Here we need this. So $K$ will
be the Hilbert space of $\pi$ as in Theorem \ref{dilation2}, and
$V:H\to K$ will denote the isometry that embeds $H$ in $K$ so that
$\Psi_{Z,\mathfrak{z}}(\cdot)=V^{*}\pi(\cdot)V$. Of course, we are
assuming the dilation is minimal. So $K$, $V$, and $\pi$ are uniquely
determined up to unitary equivalence. Finally, we let $\mathcal{M}$
be the range of $\pi(P_{0})$, where, as before, $P_{0}$ denotes
the projection is the projection of $\mathcal{F}(E)$ onto its zero$^{th}$
summand. As we have shown in Proposition \ref{TE}, $P_{0}\in\mathcal{T}(E,Z)$.
Of course, through $\pi\circ\varphi_{\infty}$, $\mathcal{M}$ becomes
a module over $M$, i.e., $\pi\circ\varphi_{\infty}(M)$ is reduced
by $\mathcal{M}$.
\begin{prop}
\label{induced parts} With the notation just established, we have:
The $M$ modules $\mathcal{D}_{*}$ and $\mathcal{M}$ are isomorphic,
i.e., there is a Hilbert space isomorphism from $\mathcal{M}$ to
$\mathcal{D}_{*}$ that intertwines the restriction of $\pi\circ\varphi_{\infty}(M)$
to $\mathcal{M}$ and the restriction of $\sigma(M)$ to $\mathcal{D}_{*}$.
As a consequence, the representations of $\mathcal{L}(\mathcal{F}(E))$
these two modules induce are unitarily equivalent, as well. \end{prop}
\begin{proof}
Since $K_{0}=\overline{\pi(\mathcal{I})K}$, for every approximate
identity $\{u_{\lambda}\}$ of $\mathcal{I}$ and every $k\in K$,
we have $\pi(u_{\lambda})k\rightarrow P_{K_{0}}k$. For every $h\in H$
and $a\in\mathcal{T}(E,Z)$, $P_{K_{0}}\pi(a)Vh=\lim_{\lambda}\pi(u_{\lambda})\pi(a)Vh=\lim_{\lambda}\pi(u_{\lambda}a)Vh\in\overline{\pi(\mathcal{I})H}$.
Thus $\overline{\pi(\mathcal{I})VH}=P_{K_{0}}\overline{\pi(\mathcal{T}(E,Z))VH}=K_{0}$
where the last equality follows from the minimality of the Stinespring
dilation. Thus $\mathcal{M}=\pi(P_{0})\overline{\pi(\mathcal{I})VH}$.
By Lemma\ref{idealP0}, we now conclude that 
\begin{multline}\mathcal{M}=\overline{span}\{\pi(W_{\xi})\pi(P_{0})\pi(W_{\theta})^{*}Vh\mid\xi,\theta\in\cup_{n\geq0}E^{\otimes n},h\in H\}\\=\overline{span}\{\pi(P_{0}W_{\xi}P_{0})\pi(W_{\theta})^{*}Vh\mid \xi,\theta\in\cup_{n\geq0}E^{\otimes n},h\in H\}\\=\overline{\pi(P_{0})\pi(\mathcal{T}(E,Z)^{*})VH}.\end{multline}
Note also that $\mathcal{D}_{*}=\overline{\Delta_{*}(\mathfrak{z})H}=\overline{\Delta_{*}(\mathfrak{z})V^{*}\pi(\mathcal{T}(E,Z)^{*})VH}$.
It follows from these expressions for $\mathcal{M}$ and $\mathcal{D}_{*}$
that the proof will be complete once we show that the map from a dense
set of $\mathcal{M}$ onto a dense subset of $\mathcal{D}$, defined
on generators by 
\begin{equation}
\pi(P_{0})\pi(T^{*})Vh\mapsto\Delta_{*}(\mathfrak{z})V^{*}\pi(T^{*})Vh,\label{modmap}
\end{equation}
for $T\in\mathcal{T}(E,Z)$ and $h\in H$, is a well-defined, isometric,
$M$-module map. For this purpose, we start with the claim that for every
$k\geq1$ and every $S\in\mathcal{K}(E^{\otimes k})$, we have 
\[
V^{*}\pi(W^{k}SW^{k*})V=\mathfrak{z}^{(k)}(S\otimes I)\mathfrak{z}^{(k)*}.
\]
It is enough to prove the claim for $S=L_{\xi}L_{\eta}^{*}$ for some
$\xi,\eta\in E^{\otimes k}$. But then 
\begin{multline*}V^{*}\pi(W^{k}SW^{k*})V=V^{*}\pi(W^{k}L_{\xi}L_{\eta}^{*}W^{k*})V\\=V^{*}\pi(W^{k}L_{\xi})\pi(W^{k}L_{\xi})^{*}V=\mathfrak{z}^{(k)}(S\otimes I)\mathfrak{z}^{(k)*}, \end{multline*}
since $\pi$ is a coextension of $\sigma\times\mathfrak{z}$. This proves the
claim. Using Lemma\ref{W}, we have $P_{0}=I_{\mathcal{F}(E)}-\sum_{k=1}^{\infty}W^{k}X_{k}W^{k*}$.
Thus $\pi(P_{0})=I_{K}-\sum_{k=1}^{\infty}\pi(W^{k}X_{k}W^{k*})$
and 
\begin{multline*} 
V^{*}\pi(P_{0})V=I_{H}-\sum_{k=1}^{\infty}V^{*}\pi(W^{k}X_{k}W^{k*})V\\=I_{H}-\sum_{k=1}^{\infty}\mathfrak{z}^{(k)}(X_{k}\otimes I)\mathfrak{z}^{(k)*}=\Delta_{*}(\mathfrak{z})^{2}. 
\end{multline*}  
We now compute in order to show that for every $T,S\in\mathcal{T}(E,Z)$,
\begin{multline*}
V^{*}\pi(T)\pi(P_{0})\pi(S)^{*}V=V^{*}\pi(T)VV^{*}\pi(P_{0})VV^{*}\pi(S)^{*}V\\=V^{*}\pi(T)V\Delta_{*}(\mathfrak{z})^{2}V^{*}\pi(S)^{*}V.
\end{multline*}
This shows that (\ref{modmap}) does indeed define an isometric map
from $\mathcal{M}$ onto $\mathcal{D}$. It is left to show that this
map is a module map but this follows because, when $c\in M$, $\pi(\varphi_{\infty}(c))\pi(P_{0})\pi(T^{*})Vh=\pi(P_{0})\pi((T\varphi_{\infty}(c^{*}))^{*})Vh$
and this gets mapped to 
\begin{multline*}
\Delta_{*}(\mathfrak{z})V^{*}\pi((T\varphi_{\infty}(c^{*}))^{*})Vh\\=\Delta_{*}(\mathfrak{z})V^{*}\pi(\varphi_{\infty}(c))\pi((T)^{*})Vh\\
=\Delta_{*}(\mathfrak{z})V^{*}\pi(\varphi_{\infty}(c))VV^{*}\pi(T{}^{*})Vh\\=\Delta_{*}(\mathfrak{z})\sigma(c)\pi(T{}^{*})Vh=\sigma(c)\Delta_{*}(\mathfrak{z})\pi(T{}^{*})Vh.
\end{multline*}

\end{proof}
We conclude this section with a proposition that shows that Hypothises
B and C of Theorem~\ref{dilation2} are satisfied in many examples.
In fact, we show that under a condition (which is satisfied in the
case where $E$ is the correspondence associated with an automorphism
of $M$), whenever we fix a sequence $\{X_{k}\}$ that satisfies Hypothesis
C, we can find an associate sequence of weights that satisfies Hypothesis B.
\begin{prop}
\label{conditionb} Suppose a weight sequence satisfies Hypothesis
$C$, so that there is an $a$ such that $(\overline{\lim}||X_{k}||^{1/k})I_{E}<aI_{E}\leq X_{1}$.
Suppose also that the $R_{k}$'s satisfy the \emph{commutivity condition}:
$R_{k}$ commutes with $I_{E}\otimes R_{k-1}$ for every $k\geq2$.
Then there is a sequence $\{Z_{k}\}$ of weights associated with $X$
that satisfies Hypothesis $B$. 
\end{prop}
In particular the commutivity condition is satisfied if $\varphi_{k}(M)^{c}\subseteq\mathcal{L}(E)\otimes I_{k-1}$
, and it is satisfied when $E=\,_{\alpha}M$, for an automorphism
$\alpha$ of $M$.
\begin{proof}
Note, first, that it follows from the definition of $R_{k}$ that
\begin{equation}
R_{k}^{2}=\sum_{m=1}^{k}X_{m}\otimes R_{k-m}^{2}
\end{equation}
for $k\geq1$ (where $R_{0}:=I$). Thus 
\begin{equation}
R_{k}^{2}(I_{1}\otimes R_{k-1}^{-2})=X_{1}\otimes I_{k-1}+\sum_{m=2}^{k}(X_{m}\otimes I_{k-m})(I_{m}\otimes R_{k-m}^{2})(I_{1}\otimes R_{k-1}^{-2}).\label{rk2}
\end{equation}
Now, set $A_{k}:=(I_{E}\otimes R_{k-1}^{2})R_{k}^{-2}\in\varphi_{k}(M)^{c}$
and compute, for $1<m\leq k$, $(I_{m-1}\otimes R_{k-m}^{2})R_{k-1}^{-2}=(I_{m-1}\otimes R_{k-m}^{2})(I_{m-2}\otimes R_{k-m+1}^{-2})(I_{m-2}\otimes R_{k-m+1}^{2})(I_{m-3}\otimes R_{k-m+2}^{-2})(I_{m-3}\otimes R_{k-m+2}^{2})\cdots(I_{1}\otimes R_{k-2}^{2})R_{k-1}^{-2}=(I_{m-2}\otimes A_{k-m+1})(I_{m-3}\otimes A_{k-m+2})\cdots A_{k-1}$.
Thus it follows from (\ref{rk2}) that 
\begin{multline}
A_{k}^{-1}\\=X_{1}\otimes I_{k-1}+\sum_{m=2}^{k}(X_{m}\otimes I_{k-m})(I_{1}\otimes(I_{m-2}\otimes A_{k-m+1})\\{\times}(I_{m-3}\otimes A_{k-m+2})\cdots A_{k-1}).\label{ak}
\end{multline}
By Lemma~\ref{weights} we can choose $Z_{k}=R_{k}^{-1}(I_{E}\otimes R_{k-1})$.
Since $R_{k}$ and $I_{E}\otimes R_{k-1}$ are commuting positive
operators by assumption, we find that $Z_{k}\geq0$, and $A_{k}=(I_{E}\otimes R_{k-1}^{2})R_{k}^{-2}=Z_{k}^{2}$.
If $a>0$ is such that $(\overline{\lim}||X_{k}||^{1/k})I_{E}<aI_{E}\leq X_{1}$,
as in Hypothesis $C$, then the argument at the end of proof of Theorem
\ref{thm:The-potential-of-Phiz} shows that $Z_{k}^{2}\leq1/aI_{k}$
for all $k\geq1$. Thus $A_{k}\leq1/aI_{k}$. It now follows from
(\ref{ak}) that 
\[
Z_{k}^{-2}=A_{k}^{-1}\leq(\sum_{m=1}^{k}||X_{m}||\frac{1}{a^{m-1}})I_{k}.
\]
But Hypothesis $C$ implies that $\sum_{m=1}^{\infty}||X_{m}||\frac{1}{a^{m-1}}$
converges to a positive number, say $b$. The choice $\epsilon=1/\sqrt{b}$
yields $Z_{k}\geq\epsilon I$ for all $k$.\end{proof}
\begin{rem}
\label{condcommute} \end{rem}
\begin{enumerate}
\item As mentioned after the statement of the Proposition, the commutivity
condition is satisfied when $E=\,_{\alpha}M$. Correspondences of
this form will be studied further in the final section of the paper. 
\item When $M=\mathbb{C}$ and $E=\mathbb{C}^{d}$, the commutivity condition
need not hold but if all the matrices $X_{k}$ are diagonal (as is
in the case studied by Popescu in \cite{Popescu2010}), the commutivity
condition is automatic and hence the conclusion of the Proposition
holds in this case, too. 
\end{enumerate}

\section{Duality, the Commutant and Matricial Functions\label{Sec:Duality and the commutant}}

Our first objective in this section is to identify the commutants
of induced representations of the Hardy algebra $H^{\infty}(E,Z)$.
We will then use the results obtained to develop properties of the
matricial functions to which elements in $H^{\infty}(E,Z)$ give rise.
In a bit more detail about what we are facing, suppose that $\sigma$
is a faithful normal representation of $M$ on a Hilbert space $H_{\sigma}$.
Then the induced representation of $H^{\infty}(E,Z)$ determined by
$\sigma$, $\sigma^{\mathcal{F}(E)}$, acts on the Hilbert space $\mathcal{F}(E)\otimes_{\sigma}H_{\sigma}$
and for $F\in H^{\infty}(E,Z$), $\sigma^{\mathcal{F}(E)}(F)=F\otimes I$
. In \cite[Theorem 3.9]{Muhly2004a}, we showed that in the ``unweighted\textquotedbl{}
case $\sigma^{\mathcal{F}(E)}(H^{\infty}(E))$ is naturally unitarily
equivalent to $\iota^{\mathcal{F}(E^{\sigma})}(H^{\infty}(E^{\sigma}))$
acting on $\mathcal{F}(E^{\sigma})\otimes_{\iota}H$, where, recall,
$E^{\sigma}$ is the intertwining space $\mathfrak{I}(\sigma,\sigma^{E}\circ\varphi)$
and $\iota$ is the identity representation of $\sigma(M)'$ on $H$.
(In Section 3 of \cite{Muhly2004a}, we show that $E^{\sigma}$ is
naturally a correspondence over $\sigma(M)'$: If $X,Y\in E^{\sigma}$,
then $\langle X,Y\rangle:=X^{*}Y.$ Also, for $a,b\in\sigma(M)'$,
$a\cdot X\cdot b:=(I_{E}\otimes a)Xb$.) We shall show that a similar
result holds in the weighted setting. The main difficulty is to choose
an appropriate weight sequence to go along with $E^{\sigma}$.

Recall the ``Fourier coefficient operators'' calculated with respect
to the gauge automorphism group acting on $\mathcal{L}(\mathcal{F}(E))$
\cite[Paragraph 2.9]{Muhly2011a}. That is, if $P_{n}$ is the projection
of $\mathcal{F}(E)$ onto $E^{\otimes n}$ and if 
\[
W_{t}:=\sum_{n=0}^{\infty}e^{int}P_{n},
\]
then $\{W_{t}\}_{t\in\mathbb{R}}$ is a one-parameter, $2\pi$-periodic,
unitary subgroup of $\mathcal{L}(\mathcal{F}(E))$ such that if $\{\gamma_{t}\}_{t\in\mathbb{R}}$
is defined by the formula $\gamma_{t}=Ad(W_{t})$, then $\{\gamma_{t}\}_{t\in\mathbb{R}}$
is an ultraweakly continuous group of $*$-automorphisms of $\mathcal{L}(\mathcal{F}(E))$,
called \emph{the gauge automorphism group} of $\mathcal{L}(\mathcal{F}(E))$.
This group leaves each of the algebras $\mathcal{T}_{+}(E,Z)$ and
$H^{\infty}(E,Z)$ invariant and so do each of the \emph{Fourier coefficient
operators} $\{\Phi_{j}\}_{j\in\mathbb{Z}}$ defined on $\mathcal{L}(\mathcal{F}(E))$
by the formula 
\begin{equation}
\Phi_{j}(F):=\frac{1}{2\pi}\int_{0}^{2\pi}e^{-int}\gamma_{t}(F)\, dt,\qquad F\in\mathcal{L}(\mathcal{F}(E)),\label{FourierOperators}
\end{equation}
where the integral converges in the ultraweak topology. In fact, if
$\Sigma_{k}$ is defined by the formula 
\begin{equation}
\Sigma_{k}(F):=\sum_{|j|<k}(1-\frac{|j|}{k})\Phi_{j}(F),\label{eq:Cesaro}
\end{equation}
$F\in\mathcal{L}(\mathcal{F}(E))$, then for $F\in\mathcal{L}(\mathcal{F}(E))$,
$\lim_{k\to\infty}\Sigma_{k}(F)=F$, in the ultra-weak topology. For
our purposes it will be convenient to write $P_{n}^{\sigma}$ for
the projection $P_{n}\otimes I_{H}=\sigma^{\mathcal{F}(E)}(P_{n})$
on $\mathcal{F}(E)\otimes_{\sigma}H$. Similarly one defines $W_{t}^{\sigma}$,
$\gamma_{t}^{\sigma}$, $\Phi_{j}^{\sigma}$, and $\Sigma_{k}^{\sigma}$
and we see immediately that both $H^{\infty}(E,Z)\otimes I_{H}$ and
its commutant are left invariant by the maps $\gamma_{t}^{\sigma}$
and $\Phi_{j}^{\sigma}$ (for all $t,j$). Thus, given $A\in(H^{\infty}(E,Z)\otimes I_{H})'$,
$\lim_{k\to\infty}\Sigma_{k}(A)=A$ in the ultra-weak topology. It
follows that for $A\in(H^{\infty}(E,Z)\otimes I_{H})'$, $A$ determines
and is uniquely determined by its sequence of Fourier coefficients
$\Phi_{j}^{\sigma}(A)$.

The natural Hilbert space isomorphism $U$ mapping $\mathcal{F}(E^{\sigma})\otimes_{\iota}H$ 
to $\mathcal{F}(E)\otimes_{\sigma}H$  that we constructed
in \cite[Theorem 3.9]{Muhly2004a} is defined on the summands of $\mathcal{F}(E^{\sigma})\otimes_{\iota}H$
via the formula 
\begin{equation}
U_{k}(\eta_{1}\otimes\cdots\otimes\eta_{k}\otimes h)=(I_{k-1}\otimes\eta_{1})(I_{k-2}\otimes\eta_{2})\cdots(I_{1}\otimes\eta_{k-1})\eta_{k}(h)
\end{equation}
where $U_{k}$ denotes the restriction of $U$ to $(E^{\sigma})^{\otimes k}\otimes H$,
$I_{j}$ stands for the identity on $E^{\otimes j}$, $\eta_{1},\ldots,\eta_{k}$
lie in $E^{\sigma}$ and $h\in H$ \cite[Lemma 3.8]{Muhly2004a}.
We shall also write $I_{\sigma k}$ for the identity on $(E^{\otimes k})^{\sigma}$.

We shall need the following technical lemmas. We shall write $\mathcal{A}_{m}$
for the algebra that we have previously written $\varphi_{m}(M)^{c}$,
and we shall write $\mathcal{A}'_{m}$ for the algebra $\varphi_{\sigma m}(\sigma(M)')^{c}$,
where $\varphi_{\sigma m}$ denotes the homomorphism from $\sigma(M)'$
to $\mathcal{L}((E^{\sigma})^{\otimes m})$ that gives the left action
of $\sigma(M)'$ on $(E^{\sigma})^{\otimes m}$.
\begin{lem}
\label{U} For every $k,m\in\mathbb{N}$, $B\in\mathcal{L}(E^{\otimes k})$
and $A\in\mathcal{L}(E^{\otimes k+m})$, 
\begin{equation}
U_{k+m}^{*}(A\otimes I_{H})U_{k+m}(I_{\sigma m}\otimes U_{k}^{*}(B\otimes I_{H})U_{k})=U_{k+m}^{*}(A(B\otimes I_{m})\otimes I_{H})U_{k+m}.
\end{equation}
In particular, when $A=I_{k+m}$, 
\begin{equation}
(I_{\sigma m}\otimes U_{k}^{*}(B\otimes I_{H})U_{k})=U_{k+m}^{*}((B\otimes I_{m})\otimes I_{H})U_{k+m}.\label{BU}
\end{equation}
\end{lem}
\begin{proof}
Fix $\eta_{1},\eta_{2},\ldots,\eta_{k+m}$ in $E^{\sigma}$. Then,
for $\theta\in E^{\sigma}$ and $g\in H$, we have 
\[
U_{k+m}(\eta_{1}\otimes\cdots\otimes\eta_{m}\otimes U_{k}^{*}(\theta(g)))=U_{k+m}(\eta_{1}\otimes\cdots\otimes\eta_{m}\otimes\theta\otimes g)=
\]
\[
(I_{k+m-1}\otimes\eta_{1})\cdots(I_{k}\otimes\eta_{m})\theta(g)=(I_{k}\otimes(I_{m-1}\otimes\eta_{1})\cdots\eta_{m})(\theta(g)).
\]
Thus for $f\in E^{\otimes k}\otimes H$, 
\[
U_{k+m}(\eta_{1}\otimes\cdots\otimes\eta_{m}\otimes U_{k}^{*}f)=(I_{k}\otimes(I_{m-1}\otimes\eta_{1})\cdots\eta_{m}))f.
\]
When $h\in H$ and $f=(B\otimes I_{H})U_{k}(\eta_{m+1}\otimes\cdots\otimes\eta_{m+k}\otimes h),$
we get 
\begin{multline*}
U_{k+m}(\eta_{1}\otimes\cdots\otimes\eta_{m}\otimes U_{k}^{*}(B\otimes I_{H})U_{k}(\eta_{m+1}\otimes\cdots\otimes\eta_{m+k}\otimes h)\\
=(B\otimes((I_{m-1}\otimes\eta_{1})\cdots\eta_{m}))U_{k}(\eta_{m+1}\otimes\cdots\otimes\eta_{m+k}\otimes h)\\
=(B\otimes I_{m})(I_{k}\otimes((I_{m-1}\otimes\eta_{1})\cdots\eta_{m}))\\\times(I_{k-1}\otimes\eta_{m+1})\cdots(I_{1}\otimes\eta_{m+k-1})\eta_{m+k}(h)\\
=(B\otimes I_{m}\otimes I_{H})U_{k+m}(\eta_{1}\otimes\eta_{2}\otimes\cdots\otimes\eta_{m+k}\otimes h).
\end{multline*}
Therefore 
\[
U_{k+m}(I_{\sigma m}\otimes U_{k}^{*}(B\otimes I_{H})U_{k})=(B\otimes I_{m}\otimes I_{H})U_{k+m}.
\]
Applying $U_{k+m}^{*}(A\otimes I_{H})$ to both sides completes the
proof.\end{proof}
\begin{lem}
\label{mult} If $A\in\mathcal{A}_{k}$ and $B\in\mathcal{A}_{m}$,
then there are operators $D\in\mathcal{A}_{k}'$ and $E\in\mathcal{A}_{m}'$
such that $U_{k}^{*}(A\otimes I_{H})U_{k}=D\otimes I_{H}$ and $U_{m}^{*}(B\otimes I_{H})U_{m}=E\otimes I_{m}$.
Further, 
\[
D\otimes E\otimes I_{H}=U_{k+m}^{*}(B\otimes A\otimes I_{H})U_{k+m}.
\]
\end{lem}
\begin{proof}
For the first assertion, it suffices to show the existence of $D\in\mathcal{A}_{k}'$
so that $U_{k}^{*}(A\otimes I_{H})U_{k}=D\otimes I_{H}$. To this
end, we need first to check that $U_{k}^{*}(A\otimes I_{H})U_{k}$
commutes with $I_{\sigma k}\otimes\sigma(a)$ for all $a\in M$, by \cite[Theorem 6.23]{R1974b}. However,
note that $U_{k}(I_{\sigma k}\otimes\sigma(a))U_{k}^{*}=\varphi_{k}(a)\otimes I_{H}$,
from which follows $U_{k}^{*}(A\otimes I_{H})U_{m}(I_{\sigma k}\otimes\sigma(a))=(I_{\sigma k}\otimes\sigma(a))U_{k}^{*}(A\otimes I_{H})U_{k}$.
Thus, there is a $D\in\mathcal{L}((E^{\sigma})^{\otimes k})$ such
that $D\otimes I_{H}=U_{k}^{*}(A\otimes I_{H})U_{k}$. Note also that,
for $b\in\sigma(M)'$, we have $U_{k}(\varphi_{\sigma k}(b)\otimes I_{H})U_{k}^{*}=I_{k}\otimes b$.
Consequently, $U_{k}(\varphi_{\sigma k}(b)\otimes I_{H})U_{k}^{*}(A\otimes I_{H})=(A\otimes I_{H})U_{k}(\varphi_{\sigma k}(b)\otimes I_{H})U_{k}^{*}$
and it follows that $D\in\varphi_{\sigma k}(\sigma(M)')'$, proving
the first assertion. 

To prove the second, observe that by properties of tensor products,
it suffices to prove it separately in two cases: for the case $A=I_{k}$
and $D=I_{\sigma k}$, and for the case $B=I_{m}$ and $E=I_{\sigma m}$.
But the first case is Equation~\ref{BU} and, thus, it suffices to
prove 
\begin{equation}
U_{k+m}(D\otimes I_{\sigma m}\otimes I_{H})=(I_{m}\otimes A\otimes I_{H})U_{k+m}.
\end{equation}
So fix first arbitrary $\xi\in E^{\otimes m}$, $\eta_{k}\in(E^{\sigma})^{\otimes k}$
and $h\in H$ and compute 
\begin{multline*}
(I_{m}\otimes(D\eta_{k}))(\xi\otimes h)=\xi\otimes(D\eta_{k})(h)=\xi\otimes U_{k}(D\eta_{k}\otimes h)\\
=\xi\otimes U_{k}(D\otimes I_{H})U_{k}^{*}\eta_{k}(h)=(I_{m}\otimes A\otimes I_{H})(\xi\otimes\eta_{k}(h))\\
=(I_{m}\otimes A\otimes I_{H})(I_{m}\otimes\eta_{k})(\xi\otimes k).
\end{multline*}
It follows that, on $E^{\otimes m}\otimes H$, 
\begin{equation}
(I_{m}\otimes(D\eta_{k}))=(I_{m}\otimes A\otimes I_{H})(I_{m}\otimes\eta_{k}).\label{eq:1}
\end{equation}
Now let $\eta_{m}\in(E^{\sigma})^{\otimes m}$ and $h\in H$ be arbitrary
and apply Equation~\ref{eq:1} to $\eta_{m}(h)$ to get 
\begin{multline*}
(I_{m}\otimes(D\eta_{k}))\eta_{m}(h)\\=(I_{m}\otimes A\otimes I_{H})(I_{m}\otimes\eta_{k})\eta_{m}(h)\\=(I_{m}\otimes A\otimes I_{H})U_{k+m}(\eta_{k}\otimes\eta_{m}\otimes h).
\end{multline*}
Since the left hand side is equal $U_{k+m}(D\eta_{k}\otimes\eta_{m}\otimes h)=U_{k+m}(D\otimes I_{\sigma m}\otimes I_{H})(\eta_{k}\otimes\eta_{m}\otimes h)$,
we are done.
\end{proof}
The proof of the following lemma is straightforward and is omitted.
\begin{lem}
\label{theta}
\begin{enumerate}
\item For $\theta:H\rightarrow E^{\otimes m}\otimes H$ and $A\in\mathcal{L}(E^{\otimes k})$,
\[
(I_{k}\otimes\theta)(A\otimes I_{H})=(A\otimes I_{m}\otimes I_{H})(I_{k}\otimes\theta).
\]

\item For $\theta\in E^{\sigma}$, viewed as a map from $H$ to $E\otimes H$,
for $\eta\in(E^{\sigma})^{\otimes k}$ and for $h\in H$, 
\[
U^{*}(I\otimes\theta)U(\eta\otimes h)=\theta\otimes\eta\otimes h.
\]

\end{enumerate}
 \end{lem}
In order to motivate the choice of the weights on the Fock space of
$E^{\sigma}$, consider an element $Y\in(H^{\infty}(E,Z))'$ with
$Y=\Phi_{1}(Y)$. (Recall that $\Phi_{1}$ is the first Fourier coefficient
operator.) Since $Y$ commutes with $\varphi_{\infty}(M)$, its restriction
to $H$ (the zero'th summand in $\mathcal{F}(E)\otimes H$) is an
element, $\theta$, in $E^{\sigma}$. So $Yh=\theta(h)$ for every
$h\in H$. Now consider its restriction to $E\otimes H$. Since it
commutes with each $W_{\xi}$, we have, for all $h\in H$, $YW_{\xi}h=W_{\xi}\theta(h)$.

Using the definition of the weighted shift $W_{\xi}$, we see that
\[
Y(Z_{1}\otimes I_{H})(\xi\otimes h)=(Z_{2}\otimes I_{H})(\xi\otimes\theta(h))=(Z_{2}\otimes I_{H})(I_{1}\otimes\theta)(\xi\otimes h).
\]
It follows that 
\[
Y=(Z_{2}\otimes I_{H})(Z_{1}^{-1}\otimes I_{1}\otimes I_{H})(I_{1}\otimes\theta)
\]
and 
\[
U_{2}^{*}YU_{1}=U_{2}^{*}(Z_{2}\otimes I_{H})(Z_{1}\otimes I_{1}\otimes I_{H})^{-1}U_{2}U_{2}^{*}(I_{1}\otimes\theta)U_{1}.
\]
Thus, for $\eta\otimes h\in E^{\sigma}\otimes H$, 
\[
U_{2}^{*}YU_{1}(\eta\otimes h)=(C_{2}\otimes I_{H})(\theta\otimes\eta\otimes h)
\]
where $C_{2}=U_{2}^{*}(Z_{2}\otimes I_{H})(Z_{1}\otimes I_{1}\otimes I_{H})^{-1}U_{2}\in\mathcal{L}((E^{\sigma})^{\otimes2})$.
Thus $U_{2}^{*}YU_{1}$ behaves like a restriction of weighted shift
on $\mathcal{F}(E^{\sigma})\otimes H$. We shall see that, in fact,
$U^{*}YU$ is a weighted shift there. The computation above suggests
what the second weight would be and how to find the other weights.
\begin{lem}
\label{CZweights} Set $Z'_{1}:=Z_{1}$ and $Z'_{m}:=Z^{(m)}(Z^{(m-1)}\otimes I_{1})^{-1}$
for $m>1$. Also, set $C_{m}\otimes I_{H}:=U_{m}^{*}(Z'_{m}\otimes I_{H})U_{m}$
for all $m\geq1$. Then 
\begin{enumerate}
\item $\{Z_{m}^{'}\}$ is a uniformly bounded sequence, with $Z_{m}'\in\mathcal{A}_{m}$
for all $m$.
\item $\{C_{m}\}$ is a well defined, uniformly bounded sequence, with $C_{m}\in\mathcal{A}_{m}^{'}$
for all $m$.
\end{enumerate}
\end{lem}
\begin{proof}
Since $Z_{m}\in\mathcal{A}_{m}$ for all $m$, it is clear that $Z_{m}'\in\mathcal{A}_{m}$
for all $m$. To show that $\{Z_{m}'\}$ is uniformly bounded, observe
that 
\begin{multline*}
||(Z_{m}^{'})^{*}Z_{m}^{'}||=||(Z^{(m-1)*}\otimes I_{1})^{-1}Z^{(m)*}Z^{(m)}(Z^{(m-1)}\otimes I_{1})^{-1}||\\
=||(Z^{(m-1)*}\otimes I_{1})^{-1}R_{m}^{-2}(Z^{(m-1)}\otimes I_{1})^{-1}||\\
=||R_{m}^{-1}(Z^{(m-1)}\otimes I_{1})^{-1}(Z^{(m-1)*}\otimes I_{1})^{-1}R_{m}^{-1}||.
\end{multline*}
From the definition of $R_{m}$ and the fact that $\frac{1}{||X_{1}^{-1}||}\leq X_{1}$,
it follows that 
\begin{multline*}
\frac{1}{||X_{1}^{-1}||}(R_{m-1}^{2}\otimes I_{1})\leq(R_{m-1}\otimes I_{1})(I_{m-1}\otimes X_{1})(R_{m-1}\otimes I_{1})\\
=(R_{m-1}^{2}\otimes I_{1})(I_{m-1}\otimes X_{1})\leq R_{m}^{2}.
\end{multline*}
 The inequality in this formula is justified by \emph{the argument for} equation \eqref{eq:X-R_relation}. The point is that because of the sums involved in the definition of $R_k^2$ in \eqref{eq:Def_Rk} it is clear that 
 \[ \sum_{i = 1}^k R^2_{k-l}\otimes X_{l} = R^2_k, \qquad k\geq 1,\] also. Thus since the terms in this sum are all non-negative, we see that \[
 (R_{k-1}^{2}\otimes I_{1})(I_{k-1}\otimes X_{1}) = R_{k-1}^{2}\otimes X_{1} \leq R_{k}^{2}.\] We thus have 
\begin{multline*}
||(Z_{m}^{'})^{*}Z_{m}^{'}||=||R_{m}^{-1}(Z^{(m-1)*}Z^{(m-1)}\otimes I_{1})^{-1}R_{m}^{-1}||\\
=||R_{m}^{-1}(R_{m-1}^{2}\otimes I_{1})R_{m}^{-1}||\leq||X_{1}^{-1}||,
\end{multline*}
which shows that $\{Z_{m}^{'}\}$ is uniformly bounded. The fact that
$C_{m}$ is well defined and lies in $\mathcal{A}_{m}'$ follows from
Lemma~\ref{mult} and, of course, the uniform boundedness $\{C_{m}\}$
follows from that of $\{Z_{m}^{'}\}$. 
\end{proof}
The following lemma will be used later.
\begin{lem}
\label{commZZprime} For every $k\geq1$, 
\[
Z_{k+1}'(Z_{k}\otimes I_{1})=Z_{k+1}(I_{1}\otimes Z_{k}').
\]
\end{lem}
\begin{proof}
For $k=1$, 
\begin{multline*}
Z_{2}'(Z_{1}\otimes I_{1})=Z^{(2)}(Z_{1}\otimes I_{1})^{-1}(Z_{1}\otimes I_{1})=Z^{(2)}\\
=Z_{2}(I_{1}\otimes Z_{1})=Z_{2}(I_{1}\otimes Z_{1}').
\end{multline*}
While for $k>1$, 
\begin{multline*}
Z_{k+1}'(Z_{k}\otimes I_{1})=Z^{(k+1)}(Z^{(k)}\otimes I_{1})^{-1}(Z_{k}\otimes I_{1})\\
=Z^{(k+1)}(I_{1}\otimes Z^{(k-1)}\otimes I_{1})^{-1}(Z_{k}\otimes I_{1})^{-1}(Z_{k}\otimes I_{1})\\
=Z^{(k+1)}(I_{1}\otimes Z^{(k-1)}\otimes I_{1})^{-1}=Z_{k+1}(I_{1}\otimes Z^{(k)})(I_{1}\otimes Z^{(k-1)}\otimes I_{1})^{-1}\\
=Z_{k+1}(I_{1}\otimes(Z^{(k)}(Z^{(k-1)}\otimes I_{1})^{-1}))=Z_{k+1}(I_{1}\otimes Z_{k}').
\end{multline*}

\end{proof}
The sequence $C=\{C_{m}\}$ is the sequence of weights we want, and
we would like to show that $\iota^{\mathcal{F}(E^{\sigma})}(H^{\infty}(E,C))$
is unitarily equivalent to the commutant of $\sigma^{\mathcal{F}(E)}(H^{\infty}(E,Z))$
via $U$, but we first want to show that there is an admissible sequence
$X'=\{X'_{k}\}$, with $X'_{k}\in\mathcal{A}_{k}^{'}$, such that
$C$ is a sequence of weights associated with $X'$. For this purpose,
note that $Z'_{m}=Z^{(m)}(Z^{(m-1)}\otimes I_{1})^{-1}=Z^{(m)}(Z^{(m-1)}\otimes I_{1})^{-1}.$
So one finds easily that for $1\leq n<m$, 
\begin{equation}
Z'_{m}(Z'_{m-1}\otimes I_{1})\cdots(Z'_{m-n}\otimes I_{n})=Z^{(m)}(Z^{(m-n-1)}\otimes I_{n+1})^{-1}\label{Z'partial}
\end{equation}
and when $n=m-1$, 
\begin{equation}
Z'_{m}(Z'_{m-1}\otimes I_{1})\cdots(Z'_{1}\otimes I_{m-1})=Z^{(m)}.\label{Z'}
\end{equation}
Now note that for every $m>k$ we have 
\begin{multline*}
(I_{\sigma k}\otimes C_{m-k})\otimes I_{H}=I_{\sigma k}\otimes U_{m-k}^{*}(Z'_{m-k}\otimes I_{H})U_{m-k}\\
=U_{m}^{*}(Z'_{m-k}\otimes I_{k}\otimes I_{H})U_{m}
\end{multline*}
 (see Equation~\ref{BU}). Thus 
\begin{multline}
(C_{m}\otimes I_{H})(I_{\sigma1}\otimes C_{m-1}\otimes I_{H})\cdots(I_{\sigma(n)}\otimes C_{m-n}\otimes I_{H})\\
=U_{m}^{*}(Z'_{m}\otimes I_{H})U_{m}U_{m}^{*}(Z'_{m-1}\otimes I_{1}\otimes I_{H})U_{m}\cdots U_{m}^{*}(Z'_{m-n}\otimes I_{n}\otimes I_{H})U_{m}\\
=U_{m}^{*}((Z^{(m)}(Z^{(m-n-1)}\otimes I_{n+1})^{-1})\otimes I_{H})U_{m}.\label{Cmn}
\end{multline}

In particular, 
\begin{multline*}
C^{(m)}\otimes I_{H}=(C_{m}\otimes I_{H})(I_{\sigma1}\otimes C_{m-1}\otimes I_{H})\cdots(I_{\sigma(m-1)}\otimes C_{1})\\
=U_{m}^{*}(Z'_{m}\otimes I_{H})U_{m}U_{m}^{*}(Z'_{m-1}\otimes I_{1}\otimes I_{H})U_{m}\cdots U_{m}^{*}(Z'_{1}\otimes I_{m-1}\otimes I_{H})U_{m}\\
=U_{m}^{*}(Z^{(m)}\otimes I_{H})U_{m},
\end{multline*}
where, in the last equality, we used Equation~\ref{Z'}.

It follows that 
\begin{equation}
C^{(m)*}C^{(m)}\otimes I_{H}=U_{m}^{*}(R_{m}^{-2}\otimes I_{H})U_{m}.
\end{equation}

\begin{thm}
\label{dualweightseq} For each $k\geq1$, there is a (uniquely determined)
operator $X_{k}'\in\mathcal{A}'_{k}$ such that 
\[
X'_{k}\otimes I_{H}=U_{k}^{*}(X_{k}\otimes I_{H})U_{k}.
\]
The sequence $X'=\{X'_{k}\}$ is an admissible sequence in the sense
of Definition~\ref{X} (with $E^{\sigma}$ replacing $E$) and $\{C_{m}\}$
is a sequence of weights associated with $X'$ as in Definition~\ref{Def:Associated}. \end{thm}
\begin{proof}
For positive integers $i_{1},i_{2},\ldots,i_{l}$ that sum to $k$,
we may use Lemma~\ref{mult} repeatedly to find that 
\[
X'_{i_{1}}\otimes X'_{i_{2}}\otimes\cdots\otimes X'_{i_{l}}\otimes I_{H}=
\]
\[
U_{k}^{*}(X_{i_{l}}\otimes\cdots\otimes X_{i_{2}}\otimes X_{i_{1}}\otimes I_{H})U_{k}.
\]
Note, however, that conjugation by $U_{k}$ reverses the order of
the factors in the tensor product. While this is an important feature
of our theory, once we sum over all choices of indices $\{i_{j}\}$
that sum to $k$, as we do in Equation~\ref{eq:Def_Rk} to get $R'_{k}$,
the reversal will have no effect. Thus 
\[
R'_{k}\otimes I_{H}=U_{k}^{*}(R_{k}\otimes I_{H})U_{k},
\]
 
\[
(R'_{k})^{-1}\otimes I_{H}=U_{k}^{*}((R_{k})^{-1}\otimes I_{H})U_{k},
\]
and 
\begin{align*}
(R'_{k})^{-2}\otimes I_{H} & =U_{k}^{*}((R_{k})^{-2}\otimes I_{H})U_{k}=C^{(k)*}C^{(k)}\otimes I_{H}.\\
\end{align*}
This shows that $\{C_{m}\}$ is a sequence of weights associated with
$X'$. The only thing that is left to establish is the finiteness
of $\overline{lim}||X'_{k}||^{1/k}$. But this is immediate since $U_{k}^{*}X_{k}U_{k}=X_{k}'$
and $\overline{lim}||X_{k}||^{1/k}<\infty$.
\end{proof}
Our goal now is to prove that the commutant of $\sigma^{\mathcal{F}(E)}($$H^{\infty}(E,Z))$
is $U\iota^{\mathcal{F}(E^{\sigma})}(H^{\infty}(E^{\sigma},C))U^{*}$.

We first fix some notation. We write $W'_{\eta}$ for the weighted
shift associated with $\eta\in E^{\sigma}$ that acts on $\mathcal{F}(E^{\sigma})$.
Thus
\[
W'_{\eta}(\theta_{k})=C_{k+1}(\eta\otimes\theta_{k})
\]
for $\theta_{k}\in(E^{\sigma})^{\otimes k}$ and $H^{\infty}(E^{\sigma},C)$
is the ultraweakly closed algebra generated by $\{W'_{\eta}:\eta\in E^{\sigma}\}\cup\varphi_{\sigma\infty}(\sigma(M)')$.

As we did in Equation~\ref{Wxi}, we can define, for $\eta\in(E^{\sigma})^{\otimes m}$,
the operator $W'_{\eta}$ by 
\begin{equation}
W'_{\eta}\zeta=C^{(m+k)}(I_{\sigma m}\otimes C^{(k)})^{-1}(\eta\otimes\zeta)
\end{equation}
for $\zeta\in(E^{\sigma})^{\otimes k}$. We then have $W'_{\eta}\in H^{\infty}(E^{\sigma},C)$.

Now let $Y\in B(\mathcal{F}(E)\otimes H)$ be in the commutant of
$\sigma^{\mathcal{F}(E)}(H^{\infty}(E,Z))$ and assume that $Y=\Phi_{j}(Y)$
for some $j\geq1$. Consider the restriction of $Y$ to $H$ (the
zeroth term in $\mathcal{F}(E)\otimes H$). Then $Y$ maps $H$ into
$E^{\otimes j}\otimes H$ since $Y=\Phi_{j}(Y)$ and it belongs to
$(E^{\otimes j})^{\sigma}$ (since $Y$ commutes with $\varphi_{\infty}(M)$).
Write $Y|H=g\in(E^{\otimes j})^{\sigma}$. For every $k>0$ and every
$\xi_{1},\xi_{2},\ldots,\xi_{k}$ in $E$ and $h\in H$, we have 
\begin{equation}
YW_{\xi_{1}}\cdots W_{\xi_{k}}h=W_{\xi_{1}}\cdots W_{\xi_{k}}Yh=W_{\xi_{1}}\cdots W_{\xi_{k}}g(h).
\end{equation}
The left-hand-side of this equation is equal to $Y(Z^{(k)}\otimes I_{H})(\xi_{1}\otimes\cdots\otimes\xi_{k}\otimes h)$
while the right-hand-side is $(Z^{(k+j)}(I_{k}\otimes Z^{(j)})^{-1}\otimes I_{H})(I_{k}\otimes g)(\xi_{1}\otimes\cdots\otimes\xi_{k}\otimes h)$.
Thus the restriction of $Y$ to $E^{\otimes k}\otimes H$ is 
\begin{multline*}
(Z^{(k+j)}(I_{k}\otimes Z^{(j)})^{-1}\otimes I_{H})(I_{k}\otimes g)(Z^{(k)}\otimes I_{H})^{-1}|E^{\otimes k}\otimes H\\
=(Z^{(k+j)}(I_{k}\otimes Z^{(j)})^{-1}\otimes I_{H})(Z^{(k)}\otimes I_{j}\otimes I_{H})^{-1}(I_{k}\otimes g)|E^{\otimes k}\otimes H\\
=(Z^{(k+j)}\otimes I_{H})(Z^{(k)}\otimes I_{j}\otimes I_{H})^{-1}(I_{k}\otimes Z^{(j)}\otimes I_{H})^{-1}(I_{k}\otimes g)|E^{\otimes k}\otimes H.
\end{multline*}
By Lemma~\ref{mult}, there is a $D_{j}\in\mathcal{A}'_{j}$ such
that $D_{j}\otimes I_{H}=U_{j}^{*}(Z^{(j)}\otimes I_{H})U_{j}$ and
so we have $D_{j}\otimes I_{\sigma k}\otimes I_{H}=U_{k+j}^{*}(I_{k}\otimes Z^{(j)}\otimes I_{H})U_{k+j}$.
Consequently, we find that for $\eta_{k}\in(E^{\sigma})^{\otimes k}$
and $h\in H$, 
\begin{multline*}
U_{k+j}^{*}YU_{k}(\eta_{k}\otimes h)\\
=U_{k+j}^{*}(Z^{(k+j)}\otimes I_{H})(Z^{(k)}\otimes I_{j}\otimes I_{H})^{-1}U_{k+j}(D_{j}\otimes I_{\sigma k}\otimes I_{H})(g\otimes\eta_{k}\otimes h).
\end{multline*}
Now write $\theta_{j}$ for $D_{j}g$ to get 
\[
U_{k+j}^{*}YU_{k}(\eta_{k}\otimes h)=U_{k+j}^{*}(Z^{(k+j)}\otimes I_{H})(Z^{(k)}\otimes I_{j}\otimes I_{H})^{-1}U_{k+j}(\theta_{j}\otimes\eta_{k}\otimes h).
\]
Since 
\begin{multline*}
U_{k+j}^{*}(Z^{(k+j)}\otimes I_{H})(Z^{(k)}\otimes I_{j}\otimes I_{H})^{-1}U_{k+j}\\
=C_{k+j}(I_{\sigma}\otimes C_{k+j-1})\cdots(I_{\sigma(j-1)}\otimes C_{k+1})\otimes I_{H},
\end{multline*}
by Equation~\ref{Cmn}, we find that 
\begin{equation}
U^{*}YU=W_{\theta_{j}}'\otimes I_{H}\in\iota^{\mathcal{F}(E^{\sigma})}(H^{\infty}(E^{\sigma},C)).
\end{equation}
This was done for $Y=\Phi_{j}(Y)\in(H^{\infty}(E,Z)\otimes I_{H})'$
when $j\geq1$. For $Y=\Phi_{0}(Y)\in\sigma^{\mathcal{F}(E)}(H^{\infty}(E,Z))'$,
$Y$ commutes with $\varphi_{\infty}(M)\otimes I_{H}$. Thus $Y|H=b$
for some $b\in\sigma(M)'$ and it follows that $U^{*}YU=\varphi_{\sigma\infty}(b)\otimes I_{H}$.
To see this, we fix $\xi\in E^{\otimes k}$ and $h\in H$ and compute
\begin{multline}
Y(Z^{(k)}\xi\otimes h)=YW_{\xi}h=W_{\xi}Yh=Z^{(k)}\xi\otimes bh\\
=(I_{k}\otimes b)(Z^{(k)}\xi\otimes h).\label{b}
\end{multline}
Since $Z^{(k)}$ is invertible, $Y=I_{\mathcal{F}(E)}\otimes b$.
Thus $YU_{k}(\eta_{k}\otimes h)=(I_{k}\otimes b)\eta_{k}(h)=(\varphi_{\sigma\infty}(b)\eta_{k})(h)=U_{k}(\varphi_{\sigma\infty}(b)\otimes I_{H})(\eta_{k}\otimes h)$
and, thus, $U^{*}YU=\varphi_{\sigma\infty}(b)\otimes I_{H}\in H^{\infty}(E^{\sigma},C)\otimes I_{H}$.

These observations prove part of the following theorem.
\begin{thm}
The commutant of the algebra $\sigma^{\mathcal{F}(E)}(H^{\infty}(E,Z))$
is $U(\iota^{\mathcal{F}(E^{\sigma})}(H^{\infty}(E^{\sigma},C)))U^{*}$. \end{thm}
\begin{proof}
The discussion preliminary to the statement of the theorem shows that
$(\sigma^{\mathcal{F}(E)}(H^{\infty}(E,Z)))'\subseteq U(\iota^{\mathcal{F}(E^{\sigma})}H^{\infty}(E^{\sigma},C)))U^{*}$.
We shall focus on the reverse inclusion. For $X=\varphi_{\sigma\infty}(b)$,
with $b\in\sigma(M)'$, it follows from \cite[Lemma 3.8]{Muhly2004a}
that $U(X\otimes_{\iota}I_{H})U^{*}=I_{\mathcal{F}(E)}\otimes b\in(\sigma^{\mathcal{F}(E)}(H^{\infty}(E,Z)))'$.
It is left to analyze $U(\iota^{\sigma(\mathcal{F}(E))}(W_{\theta}'))U^{*}$
for $\theta\in E^{\sigma}$. Using Lemma 3.8 of \cite{Muhly2004a}
again, we see that $U_{k+1}(T_{\theta}\otimes I_{H})U_{k}^{*}=I_{k}\otimes\theta$.
Thus, for $k\geq0$, 
\begin{multline*}
U_{k+1}(W_{\theta}'\otimes I_{H})U_{k}^{*}=U_{k+1}(C_{k+1}\otimes I_{H})U_{k+1}^{*}U_{k+1}(T_{\theta}\otimes I_{H})U_{k}^{*}\\
=(Z_{k+1}'\otimes I_{H})(I_{k}\otimes\theta).
\end{multline*}
For $a\in M$, 
\begin{multline*}
(Z_{k+1}'\otimes I_{H})(I_{k}\otimes\theta)(\varphi_{k}(a)\otimes I_{H})=(Z_{k+1}'\otimes I_{H})(\varphi_{k+1}(a)\otimes I_{H})(I_{k}\otimes\theta)\\
=(\varphi_{k+1}(a)\otimes I_{H})(Z_{k+1}'\otimes I_{H})(I_{k}\otimes\theta).
\end{multline*}
Thus $U(W_{\theta}'\otimes I_{H})U^{*}$ commutes with $\varphi_{\infty}(a)\otimes I_{H}$.
It remains to show that it also commutes with $\sigma^{\mathcal{F}(E)}(W_{\xi})=W_{\xi}\otimes I_{H}$,
for $\xi\in E$. So fix $k\geq0$, $\xi\in E$ and $\theta\in E^{\sigma}$
and compute using Lemma~\ref{commZZprime}: 
\begin{multline*}
(Z_{k+1}'\otimes I_{H})(I_{k}\otimes\theta)(W_{\xi}\otimes I_{H})=(Z_{k+1}'\otimes I_{H})(I_{k}\otimes\theta)(Z_{k}\otimes I_{H})(T_{\xi}\otimes I_{H})\\
=(Z_{k+1}'\otimes I_{H})(Z_{k}\otimes I_{1}\otimes I_{H})(I_{k}\otimes\theta)(T_{\xi}\otimes I_{H})\\
=(Z_{k+1}'\otimes I_{H})(Z_{k}\otimes I_{1}\otimes I_{H})(T_{\xi}\otimes I_{H})(I_{k}\otimes\theta)\\
=(Z_{k+1}\otimes I_{H})(I_{1}\otimes Z_{k}'\otimes I_{H})(T_{\xi}\otimes I_{H})(I_{k}\otimes\theta)\\
=(Z_{k+1}\otimes I_{H})(T_{\xi}\otimes I_{H})(Z_{k}'\otimes I_{H})(I_{k}\otimes\theta)\\
=(W_{\xi}\otimes I_{H})(Z_{k}'\otimes I_{H})(I_{k}\otimes\theta).
\end{multline*}
Consequently, $U(W_{\theta}'\otimes I_{H})U^{*}$ commutes with $W_{\xi}\otimes I_{H}$
and the proof is complete.
\end{proof}
In \cite{Muhly2013}, we associated to every element of $H^{\infty}(E)$
its Berezin transform and showed that it is a matricial function (in
a sense that will be made precise below). We also proved a converse
of this result. Our goal now is to extend this study to the current,
weighted, setting and we start with the following definition.
\begin{defn}
\label{def:Matricial_family_of_sets} Let $NRep(M)$ denote the collection
of all normal representations of $M$ on a separable Hilbert space.
A family of sets $\{\mathcal{U}(\sigma)\}_{\sigma\in NRep(M)}$, with
$\mathcal{U}(\sigma)\subseteq\mathfrak{I}(\sigma^{E}\circ\varphi,\sigma)$,
satisfying $\mathcal{U}(\sigma)\oplus\mathcal{U}(\tau)\subseteq\mathcal{U}(\sigma\oplus\tau)$
is called a \emph{matricial family} of sets. 
\end{defn}
We shall be interested here mainly with the following matricial families.

\begin{examples}\label{examples matricial families} 
\begin{enumerate}
\item For a given admissible sequence $X$, the families $\{D(X,\sigma)\}_{\sigma\in NRep(M)}$
and $\{\overline{D(X,\sigma)}\}_{\sigma\in NRep(M)}$ are matricial
families. 
\item For $\sigma\in NRep(M)$, write $\mathcal{AC}(\sigma)$ for the set
of all $\mathfrak{z}\in\overline{D(X,\sigma)}$ such that the representation
$\sigma\times\mathfrak{z}$ extends to an ultraweakly continuous representation
of $H^{\infty}(E,Z)$. Then  $\{\mathcal{AC}(\sigma)\}_{\sigma\in NRep(M)}$
is a matricial family.
\end{enumerate}
\end{examples}

The Berezin transform, $\widehat{F}=\{\widehat{F}_{\sigma}\}_{\sigma\in NRep(M)}$,
of an element $F\in H^{\infty}(E,Z)$ is defined by


\[
\widehat{F}_{\sigma}(\mathfrak{z})=(\sigma\times\mathfrak{z})(F)
\]
where $\mathfrak{z}$ runs over $\mathcal{AC}(\sigma)$.  Of course, when $F \in \mathcal{T}(E,Z)$, then $\widehat{F}_{\sigma}(\mathfrak{z})$ makes sense for all $\mathfrak{z}\in \overline{D(X,\sigma)}$.  What one can say about $\mathcal{AC}(\sigma)$ and its structure (beyond simply being a subset of $\overline{D(X,\sigma)}$) is a real mystery that deserves more study.

A Berezin transform, $\widehat{F}=\{\widehat{F}_{\sigma}\}_{\sigma\in NRep(M)}$, clearly satisfies the equation 
\begin{equation}
\widehat{F}_{\sigma\oplus\tau}(\mathfrak{z}\oplus\mathfrak{w})=\widehat{F}_{\sigma}(\mathfrak{z})\oplus\widehat{F}_{\tau}(\mathfrak{w}),\qquad\mathfrak{z}\oplus\mathfrak{w}\in D_{X,\sigma}\oplus D_{X,\tau}.\label{eq:Direct_sum_decomposition}
\end{equation}
In fact, the Berezin transforms have an additional property. To describe
it, note that if $\sigma$ and $\tau$ are representations of $M$,
then an operator $C:H_{\sigma}\rightarrow H_{\tau}$ is said to intertwine
$\sigma$ and $\tau$ if $C\sigma(a)=\tau(a)C$ for all $a\in M$,
in which case we write $C\in\mathcal{I}(\sigma,\tau)$. Similarly,
if $C$ intertwines the representations $\sigma\times\mathfrak{z}$
and $\tau\times\mathfrak{w}$ of $H^{\infty}(E,Z)$, then we will
write $C\in\mathcal{I}(\sigma\times\mathfrak{z},\tau\times\mathfrak{w})$.
\begin{defn}
\label{def:Matricial_family_of_functions} Suppose $\{\mathcal{U}(\sigma)\}_{\sigma\in NRep(M)}$
is a matricial family of sets and suppose that for each $\sigma\in NRep(M)$,
$f_{\sigma}:\mathcal{U}(\sigma)\to B(H_{\sigma})$ is a function.
We say that $f:=\{f_{\sigma}\}_{\sigma\in NRep(M)}$ is a \emph{matricial
family of functions }in case 
\begin{equation}
Cf_{\sigma}(\mathfrak{z})=f_{\tau}(\mathfrak{w})C\label{eq:respects_intertwiners}
\end{equation}
for every $\mathfrak{z}\in\mathcal{U}(\sigma)$, every $\mathfrak{w}\in\mathcal{U}(\tau)$
and every $C\in\mathcal{I}(\sigma,\tau)$ such that 
\begin{equation}
C\mathfrak{z}=\mathfrak{w}(I_{E}\otimes C).\label{eq:basic_intertwining}
\end{equation}

When the family is $\{\mathcal{AC}(\sigma)\}_{\sigma\in NRep(M)}$
and $f=\{f_{\sigma}\}_{\sigma\in NRep(M)}$ is a Berezin transform,
then it is easy to see that the assumptions on an operator $C:H_{\sigma}\to H_{\tau}$
that $C\in\mathcal{I}(\sigma,\tau)$ and satisfies equation \eqref{eq:basic_intertwining}
express the fact that $C$ lies in $\mathcal{I}(\sigma\times\mathfrak{z},\tau\times\mathfrak{w})$.
But then, equation \eqref{eq:respects_intertwiners} is immediate.
It is simply a manifestation of the structure of the commutant of
the representation $(\sigma\oplus\tau)\times(\mathfrak{z}\oplus\mathfrak{w})$.
In this setting also, the defining hypothesis for a matricial family
can be written simply as 
\begin{equation}
\mathcal{I}(\sigma\times\mathfrak{z},\tau\times\mathfrak{w})\subseteq\mathcal{I}(f_{\sigma}(\mathfrak{z}),f_{\tau}(\mathfrak{w})),\label{eq:Resp_intertwiners}
\end{equation}
for all $\sigma,\tau\in NRep(M)$, $\mathfrak{z}\in\mathcal{AC}(\sigma),$
and $\mathfrak{w}\in\mathcal{AC}(\tau)$. Consequently, we sometimes
say that a matricial family \emph{respects intertwiners}. Observe
that if a family respects intertwiners, then it automatically satisfies
equations like \eqref{eq:Direct_sum_decomposition}. 
\end{defn}
One of the principal results of \cite[Theorem 4.4]{Muhly2013} is the converse:
In the ``unweighted case\textquotedbl{}, every matricial family $\{f_{\sigma}\}_{\sigma\in NRep(M)}$ is the
Berezin transform of some $F\in H^{\infty}(E)$. The two main ingredients
that were used in the proof of this converse are the double commutant
result of \cite[Corollary 3.10]{Muhly2004a} and the analysis of the
representations in $\mathcal{AC}(\sigma)$ obtained in \cite{Muhly2011a}.
In the weighted setting, we still have the double commutant theorem,
thanks to Jennifer Good. We state her theorem as
\begin{thm}
\cite[Theorem 4.4]{Good2015} \label{thm: Jenni} Suppose $Z$ is
a sequence of invertible weights associated with an admissible weight
sequence. Suppose, also, that $\sigma$ is a normal representation
of $M$ on a Hilbert space. Then $(\sigma^{\mathcal{F}(E)}(H^{\infty}(E,Z)))''=\sigma^{\mathcal{F}(E)}(H^{\infty}(E,Z))$.
\end{thm}
However, we do not have, yet, an analysis of $\mathcal{AC}(\sigma)$
that is as complete as the one in \cite{Muhly2011a}. Consequently,
the following result is somewhat weaker than a full converse.
\begin{thm}
\label{thm:Double_Commutant} If $f=\{f_{\sigma}\}_{\sigma\in NRep(M)}$
is a matricial family of functions, with $f_{\sigma}$ defined on
$\mathcal{AC}(\sigma)$ and mapping to $B(H_{\sigma}),$ then there
is an $F\in H^{\infty}(E,Z)$ such that $f$ and the Berezin transform
of $F$ agree on $D(X,\sigma)$ , i.e., 
\[
f_{\sigma}(\mathfrak{z})=\widehat{F}_{\sigma}(\mathfrak{z})
\]
for every $\sigma$ and every $\mathfrak{z}\in D(X,\sigma)$. \end{thm}
\begin{proof}
Let $\pi:M\to B(H_{\pi})$ be a faithful normal representation of
$M$ of infinite multiplicity. Set $\sigma_{0}=\pi^{\mathcal{F}(E)}\circ\varphi_{\infty}$,
acting on $H_{\sigma_{0}}=\mathcal{F}(E)\otimes_{\pi}H_{\pi}$, and
define $\mathfrak{s}_{0}$ by the formula 
\[
\mathfrak{s}_{0}(\xi\otimes h)=W_{\xi}h,\qquad\xi\in E,\, h\in\mathcal{F}(E)\otimes_{\pi}H_{\pi}.
\]
Then $\sigma_{0}\times\mathfrak{s}_{0}$ is an induced representation
of $H^{\infty}(E,Z)$. In \cite[Proposition 2.3]{Muhly2011a} we show
that $\sigma_{0}\times\mathfrak{s}_{0}$ is unique up to unitary equivalence
in the sense that if $\pi'$ has the same properties as $\pi$ and
if $\sigma_{0}'\otimes\mathfrak{s}_{0}'$ is constructed from $\pi'$
in a similar fashion to $\sigma_{0}\times\mathfrak{s}_{0}$, then
$\sigma_{0}'\otimes\mathfrak{s}_{0}'$ is unitarily equivalent to
$\sigma_{0}\times\mathfrak{s}_{0}$. Further, if $\sigma\times\mathfrak{z}$
is any induced representation of $H^{\infty}(E,Z)$, then there is
a subspace $\mathcal{K}$ of $H_{\pi}$ that reduces $\pi$ such that
$\sigma\times\mathfrak{z}$ is unitarily equivalent to $\sigma_{0}\times\mathfrak{s}_{0}\vert\mathcal{F}(E)\otimes_{\pi}\mathcal{K}$.
This was proved for the unweighted case, $H^{\infty}(E)$, but the
same proofs hold here.

Observe that by construction $\sigma_{0}\times\mathfrak{s}_{0}$ is
absolutely continuous in the sense that it extends to an ultraweakly
continuous representation of $H^{\infty}(E,Z)$ acting $\mathcal{F}(E)\otimes_{\pi}H_{\pi}$.
So $\mathfrak{s}_{0}\in\mathcal{AC}(\sigma_{0})$ by definition.

Suppose that $\{f_{\sigma}\}_{\sigma\in NRep(M)}$ is a matricial
family of functions defined on $\{\mathcal{AC}(\sigma)\}_{\sigma\in NRep(M)}$.
Our hypotheses guarantee that for every $C\in\mathcal{I}(\sigma_{0}\times\mathfrak{s}_{0},\sigma_{0}\times\mathfrak{s}_{0})$,
$Cf_{\sigma_{0}}(\mathfrak{s}_{0})=f_{\sigma_{0}}(\mathfrak{s}_{0})C$.
That is, $f_{\sigma_{0}}(\mathfrak{s}_{0})$ lies in the double commutant
of $(\sigma_{0}\times\mathfrak{s}_{0})(H^{\infty}(E))$. However,
$\sigma_{0}\times\mathfrak{s}_{0}$ is the restriction of $\pi^{\mathcal{F}(E)}$
to $H^{\infty}(E,Z)$, where $\sigma_{0}=\pi^{\mathcal{F}(E)}\circ\varphi_{\infty}$,
and $\pi^{\mathcal{F}(E)}(H^{\infty}(E,Z))$ is its own double commutant
by Theorem \ref{thm: Jenni}. Thus there is an $F\in H^{\infty}(E,Z)$
so that 
\begin{equation}
f_{\sigma_{0}}(\mathfrak{s}_{0})=\widehat{F}_{\sigma_{0}}(\mathfrak{s}_{0})\;(=F\otimes I_{H_{\pi}}).\label{eq:f-F_equality}
\end{equation}
If $\sigma$ is an arbitrary representation in $NRep(M)$ and if $\mathfrak{z}\in\mathcal{AC}(\sigma)$,
then for every $C\in\mathcal{I}(\sigma_{0}\times\mathfrak{s}_{0},\sigma\times\mathfrak{z})$,
$f_{\sigma}(\mathfrak{z})C=Cf_{\sigma_{0}}(\mathfrak{s}_{0})$ because
$\{f_{\sigma}\}_{\sigma\in\Sigma}$ preserves intertwiners by hypothesis.
However, by \eqref{eq:f-F_equality}, $Cf_{\sigma_{0}}(\mathfrak{s}_{0})=C\widehat{F}_{\sigma_{0}}(\mathfrak{s}_{0})$.
Hence we have 
\begin{equation}
f_{\sigma}(\mathfrak{z})C=C\widehat{F}_{\sigma_{0}}(\mathfrak{s}_{0})=\widehat{F}_{\sigma}(\mathfrak{z})C,\label{C}
\end{equation}
where the second equality is justified because $\{\widehat{F}_{\sigma}\}$
is a matricial family and we assume that $C\in\mathcal{I}(\sigma_{0}\times\mathfrak{s}_{0},\sigma\times\mathfrak{z})$.

Now fix a normal representation $\sigma$ of $M$ on $H_{\sigma}$
and $\mathfrak{z}\in D(X,\sigma)$. Write $\sigma_{1}\times\mathfrak{s}_{1}$
for the restriction of $\sigma^{\mathcal{F}(E)}$ 
to $H^{\infty}(E,Z)$. As was noted above, there is a subspace $\mathcal{K}$
of $H_{\pi}$ that reduces $\pi$ such that $\sigma_{1}\times\mathfrak{s}_{1}$
is unitarily equivalent to $\sigma_{0}\times\mathfrak{s}_{0}\vert\mathcal{F}(E)\otimes_{\pi}\mathcal{K}$.
Applying the unitary equivalence to the projection of $\mathcal{F}(E)\otimes_{\pi}H_{\pi}=H_{\sigma_{0}}$
onto $\mathcal{F}(E)\otimes_{\pi}\mathcal{K}$, we get a map $Q$
from $H_{\sigma_{0}}$ onto $H_{\sigma_{1}}$ that intertwines the
representations $\sigma_{0}\times\mathfrak{s}_{0}$ and $\sigma_{1}\times\mathfrak{s}_{1}$.
Since $\mathfrak{z}$ is assumed to be in the ``open'' disc $D(X,\sigma)$,
$K_{Z}(\mathfrak{z}):H_{\sigma}\to H_{\sigma_{1}}$ is an isometry.
Therefore, $C_{0}:=K_{Z}(\mathfrak{z})^{*}Q$ maps $H_{\sigma_{0}}$
onto $H_{\sigma}$.

We claim that $C_{0}\in\mathcal{I}(\sigma_{0}\times\mathfrak{s}_{0},\sigma\times\mathfrak{z})$.
For this it will suffice to show that $K_{Z}(\mathfrak{z})^{*}\in\mathcal{I}(\sigma_{1}\times\mathfrak{s}_{1},\sigma\times\mathfrak{z})$.
The fact that $K_{Z}(\mathfrak{z})^{*}\in\mathcal{I}(\sigma_{1},\sigma)$
follows immediately from Lemma\ref{Kintertwines} (3). Using Lemma\ref{Kintertwines}
(2) we see, upon taking adjoints, that for $\xi\in E^{\otimes k}$,
$k\geq0$, 
\[
\mathfrak{z}L_{\xi}K_{Z}(\mathfrak{z})^{*}=K_{Z}(\mathfrak{z})^{*}(W\otimes I_{H_{\sigma_{1}}})L_{\xi}.
\]
But the left hand side is equal to $\mathfrak{z}(I_{E}\otimes K_{Z}(\mathfrak{z})^{*})L_{\xi}$
and, thus, we get 
\[
\mathfrak{z}(I_{E}\otimes K_{Z}(\mathfrak{z})^{*})=K_{Z}(\mathfrak{z})^{*}(W\otimes I_{H_{\sigma_{1}}}),
\]
showing that $K_{Z}(\mathfrak{z})^{*}\in\mathcal{I}(\sigma_{1}\times\mathfrak{s}_{1},\sigma\times\mathfrak{z})$
and completing the proof that $C_{0}\in\mathcal{I}(\sigma_{0}\times\mathfrak{s}_{0},\sigma\times\mathfrak{z})$.

It follows now from (\ref{C}) that $f_{\sigma}(\mathfrak{z})C_{0}=\widehat{F}_{\sigma}(\mathfrak{z})C_{0}$
and, using the fact that $C_{0}$ is a surjective map, we conclude
that $f_{\sigma}(\mathfrak{z})=\widehat{F}_{\sigma}(\mathfrak{z})$
for $\mathfrak{z}\in D(X,\sigma)$.
\end{proof}

\section{Weighted crossed products}

In this section we fix an automorphism $\alpha$ of $M$ and we set
$E=\,_{\alpha}M$. That is, as a right $M$-module, $E$ is just $M$
and the right action of $M$ is just the multiplication from $M$.
The left action is implemented by $\alpha$. So, really, $\varphi=\alpha$
in this case. Finally, the inner product on $E$ is given by $\langle\xi,\eta\rangle:=\xi^{*}\eta$.
Note that the tensor power of order $k$, $E^{\otimes k}$, is isomorphic
to $_{\alpha^{k}}M$. Note, too, that $\mathcal{L}(E^{\otimes k})=\mathcal{K}(E^{\otimes k})=M$.
More precisely, an operator in $\mathcal{L}(E^{\otimes k})$ is given
by left multiplication by an element of $M$, i.e. if $\lambda\in\mathcal{L}(E)$
and if $b:=\lambda(1_{M})$, then $\lambda(\xi)=\lambda(1_{M}\xi)=\lambda(1_{M})\xi=b\xi$.
The relative commutant of $\varphi_{k}(M)$ is $\varphi_{k}(M)^{c}=\{b\in M\mid b\alpha^{k}(a)=\alpha^{k}(a)b\ a\in M\}$.
Since we are assuming $\alpha$ is an automorphism of $M$, $\varphi_{k}(M)^{c}$
is  the algebra of
multiplication by elements in $M\cap\alpha^k(M)'=Z(M)$ and will be
identified with $Z(M)$. 

We shall fix a sequence of positive elements $\{x_{k}\}$ in $Z(M)$
with the properties that $x_{1}$ is invertible and  $\overline{\lim}||x_{k}||^{1/k}$ is finite. So $\{x_{k}\}$ is
an admissible sequence in the sense of Definition \ref{X}. We shall
also fix an associated sequence of weights, but we shall write them
$z_{k}$ instead of $Z_{k}$ and we shall assume that each $z_{k}$
lies in $Z(M)$. Note that the canonical associated weights (Definition
\ref{Def:Associated}) all lie in $Z(M)$. 

Given a normal representation $\sigma$ of $M$ on $H$, every element
$\mathfrak{z}$ of $\mathfrak{I}(\sigma^{E}\circ\varphi,\sigma)$
is an operator from $E\otimes_{\sigma}H$ into $H$ that intertwines
$\varphi(\cdot)\otimes I_{H}$ and $\sigma(\cdot)$. Since $E\otimes_{\sigma}H$
can be identified with $H$ via the map $a\otimes h\mapsto\sigma(a)h$
and this identification carries $\varphi(\cdot)\otimes I_{H}$ into
$\sigma(\alpha(\cdot))$, we may view $\mathfrak{I}(\sigma^{E}\circ\varphi,\sigma)$
as 
\[
\{T\in B(H):T\sigma(\alpha(a))=\sigma(a)T,\;\; a\in M\}.
\]
If we identify $\mathfrak{z}\in\mathfrak{I}(\sigma^{E}\circ\varphi,\sigma)$
with $T\in B(H)$ in this way, then for $k\geq1$, $\mathfrak{z}^{(k)}$
is identified with $T^{k}$. Under these identifications, then, $\overline{D(X,\sigma)}$ may be identified with
\[
\{T\in B(H)\mid T\sigma(\alpha(a))=\sigma(a)T, a\in M, \hbox{and} \sum_{k=1}^{\infty}T^{k}\sigma(x_{k})T^{k*}\leq I\}.
\]
By Corollary~\ref{Dparametrizesrepresentations}, the operators in
$\overline{D(X,\sigma)}$ parameterize all the completely contractive
representations of the tensor algebra. In this setting we shall write
the tensor algebra $\mathcal{T}_{+}(\alpha,z)$.

We now describe the construction of $\mathcal{T}_{+}(\alpha,z)$.
Since the Fock correspondence associated with $E$ is $\sum_{k=0}^{\infty}\oplus_{\alpha^{k}}M$,
it can by identified, as a $W^{*}$-module, with the direct sum of
infinitely many copies of $M$. Thus the elements of $\mathcal{T}_{+}(\alpha,z)$
can each be written as an infinite matrix with entries in $M$. The
generators are $\varphi_{\infty}(a)$ for $a\in M$ and $W_{\xi}$
for $\xi\in E$. But since $E$ is generated, as a correspondence
(and also as a $W^{*}$-module) by a single element $I$ (the identity
operator in $M$ viewed as an element in $E$), we can write $W$
for $W_{I}$ and then $\mathcal{T}_{+}(E,z)$ is generated (as a norm
closed algebra) by $\{\varphi_{\infty}(a)\}_{a\in M}$ and $W$. Matricially
we can write 
\[
\varphi_{\infty}(a)=\rm{diag}[a,\alpha(a),\alpha^{2}(a),\ldots]
\]
and 
\[
W=\left(\begin{array}{ccccc}
0 & 0 & 0 & 0 & \cdots\\
z_{1} & 0 & 0 & 0 & \cdots\\
0 & z_{2} & 0 & 0 & \cdots\\
0 & 0 & z_{3} & 0 & \ddots\\
\vdots & \vdots & \vdots & \ddots & \ddots
\end{array}\right).
\]
These operators generate $\mathcal{T}_{+}(\alpha,z)$ as a norm closed
algebra, they generate $H^{\infty}(\alpha,z)$ as an ultraweakly closed
algebra, and they generate $\mathcal{T}(\alpha,z)$ as a $C^{*}$-algebra.

Now write 
\[
S=\left(\begin{array}{ccccc}
0 & 0 & 0 & 0 & \cdots\\
I & 0 & 0 & 0 & \cdots\\
0 & I & 0 & 0 & \cdots\\
0 & 0 & I & 0 & \ddots\\
\vdots & \vdots & \vdots & \ddots & \ddots
\end{array}\right),
\]
and let $D_{k}=\rm{diag}[0,0,\ldots,0,z_{1},z_{2},\ldots]$ (with $k+1$ zeros)
for $k\geq0$ and $D_{k}=\rm{diag}[z_{-k},z_{1-k},\ldots]$ for $k<0$. (Note: These $D_{k}$ are different from the $D_{k}$ in Section \ref{sec:Dilations}.)
Let $\mathcal{D}$ be the $C^{*}$-algebra generated by $\varphi_{\infty}(M)\cup\{D_{k}:k\in\mathbb{Z}\}$
and define the map $\beta:\mathcal{D}\rightarrow\mathcal{D}$ by the
formula $\beta(y)=S^{*}yS$. Since 
\[
\beta(D_{k})=D_{k-1}
\]
and 
\[
\beta(\varphi_{\infty}(a))=\varphi_{\infty}(\alpha(a))
\]
for $k\in\mathbb{Z}$ and $a\in M$, $\beta$ maps $\mathcal{D}$
into itself. Since $S$ is an isometry whose range projection lies
in $\mathcal{D}'$, $\beta$ is an endomorphism. It is surjective
 since all the generators are contained in its image,
but it may not be injective. Nevertheless, if we write $\mathcal{I}$ for the
ideal in $\mathcal{T}(\alpha,z)$ generated by the projection $\rm{diag}[I,0,0,\ldots]$
and write $q:\mathcal{T}(\alpha,z)\rightarrow\mathcal{T}(\alpha,z)/\mathcal{I}$
for the quotient map, then $\beta$ induces an automorphism - also
denoted $\beta$ - on $q(\mathcal{D})$.

The natural gauge group action on $\mathcal{T}(\alpha,z)$, discussed
at the beginning of Section~\ref{Sec:Duality and the commutant},
$\{\gamma_{\lambda}\}$, is manifested here through the formulas $\gamma_{\lambda}(y)=y$
for $y\in\mathcal{D}$ and $\gamma_{\lambda}(W)=\lambda W$, $\lambda\in\mathbb{T}$.
Also note that, since $\gamma_{\lambda}(P_{0})=P_{0}$, $\gamma_{\lambda}(\mathcal{I})\subseteq\mathcal{I}$
and, thus, $\gamma$ induces an action of $\mathbb{T}$ on $\mathcal{T}(\alpha,z)/\mathcal{I}$,
also denoted $\gamma$. We have $\gamma_{\lambda}(q(y))=q(y)$ for
$y\in\mathcal{D}$ and $\gamma_{\lambda}(q(W))=\lambda q(W)$.

The following theorem generalizes Theorem 3.1.1 of \cite{ODonovan1975}
.
\begin{thm}
\label{isomorphcrossedprod} With $M,\alpha$ and $z$ as above, the
$C^{*}$-algebra $\mathcal{T}(\alpha,z)/\mathcal{I}$ is $C^{*}$-isomorphic
to the crossed product algebra $q(\mathcal{D})\rtimes\beta$.
\end{thm}
Note that in the ``unweighted\textquotedbl{} case, i.e. the case
when $z_{i}=I$ for all $i$, the theorem clearly holds because then
$q(\mathcal{D})=\mathcal{D}\cong M$ and $\beta$ can be identified
with $\alpha$.
\begin{proof}
The algebra $q(\mathcal{D})\rtimes\beta$ is generated by a copy of
$q(\mathcal{D})$ and a unitary element $u$ that together satisfy
$q(y)u=u\beta(q(y))$ for $y\in\mathcal{D}$. In order to define a
$^{*}$-homomorphism from $q(\mathcal{D})\rtimes\beta$ to $\mathcal{T}(\alpha,z)/\mathcal{I}$,
we let $\rho_{0}:q(\mathcal{D})\rightarrow q(\mathcal{D})$ be the
identity map from the copy of $q(\mathcal{D})$ in $q(\mathcal{D})\rtimes\beta$
to $q(\mathcal{D})$ as a subalgebra of $\mathcal{T}(\alpha,z)/\mathcal{I}$.
Also, write $U=q(S)$ and observe that, for $y\in\mathcal{D}$, $\beta(q(y))=q(\beta(y))=q(S^{*}yS)=U^{*}q(y)U$.
Since $S^{*}S=I$ and $I-SS^{*}=P_{0}\in\mathcal{I}$, $U$ is a unitary
operator and it follows that, for $y\in\mathcal{D}$, 
\[
q(y)U=U\beta(q(y)).
\]
Consequently, there is a $^{*}$-homomorphism $\rho:q(\mathcal{D})\rtimes\beta\rightarrow\mathcal{T}(\alpha,z)/\mathcal{I}$
with $\rho(q(y))=\rho_{0}(q(y))$ for $y\in\mathcal{D}$ and $\rho(u)=U$.
The image of $\rho$ contains $q(\mathcal{D})$ and $U=q(S)$. Since
$q(W)=q(D_{0}S)=q(D_{0})U$ lies in the image of $\rho$, $\rho$
is surjective. It is left to show that it is injective.

For this, write $\tau$ for the natural action of $\mathbb{T}$ on
$q(\mathcal{D})\rtimes\beta$. Thus, for $\lambda\in\mathbb{T}$ and
$y\in\mathcal{D}$, we have $\tau_{\lambda}(q(y))=q(y)$ and $\tau_{\lambda}(u)=\lambda u$.
Then, for $y\in\mathcal{D}$ and $\lambda\in\mathbb{T}$, 
\[
\gamma_{\lambda}(\rho(q(y)))=\gamma_{\lambda}(q(y))=q(y)=\tau_{\lambda}(q(y))
\]
and 
\[
\gamma_{\lambda}(\rho(u))=\gamma_{\lambda}(U)=\lambda U=\rho(\lambda u)=\rho(\tau_{\lambda}(u)).
\]
Thus 
\[
\gamma_{\lambda}\circ\rho=\rho\circ\tau_{\lambda}.
\]
Since $\rho_{0}$ is injective, it now follows that $\rho$ is injective
and, thus, a $^{*}$-isomorphism, as required.\end{proof}
\begin{rem}
The algebra $\mathcal{T}(\alpha,z)/\mathcal{I}$ is the weighted crossed
product algebra and the last theorem shows that (under the conditions
here) it can be presented as an (unweighted) crossed product algebra.
\end{rem}
\bibliographystyle{plain}
\bibliography{/Users/Paul_Muhly/Dropbox/BP_Share/Master_Bib_File/Master130901-1}

\end{document}